\documentclass{article}
\usepackage{graphicx}
\usepackage{amsmath,amsfonts,latexsym,amscd,amssymb,theorem}
\usepackage{graphicx}
\usepackage[ansinew]{inputenc}
\usepackage[shortlabels]{enumitem}

\usepackage{epsfig}
\usepackage{amsmath}
\usepackage{amssymb}
\usepackage{afterpage}
\usepackage{selinput}

\newcommand{\ba}{\begin{eqnarray}}
\newcommand{\ea}{\end{eqnarray}}

\newtheorem{thm}{Theorem}[section]

\newtheorem{conjecture}{Conjecture}

\newtheorem{theorem}[thm]{Theorem}

\newtheorem{lemma}[thm]{Lemma}

\newtheorem{corollary}[thm]{Corollary}

\makeatletter
\newcommand*{\rom}[1]{\expandafter\@slowromancap\romannumeral #1@}
\makeatother

\begin{document}
\title{\textbf{About the Erd\"{o}s-Hajnal conjecture for seven-vertex tournaments}}
\maketitle


\begin{center}
\author{Soukaina ZAYAT \footnote{Department of Mathematics, Lebanese University, Hadath, Lebanon. (soukaina.zayat.96@outlook.com)}, Salman GHAZAL \footnote{Department of Mathematics, Lebanese University, Hadath, Lebanon. (salman.ghazal@ul.edu.lb)}}
\end{center}

\begin{abstract}
 A celebrated unresolved conjecture of Erd\"{o}s and Hajnal states that for every undirected graph $H$ there exists $ \epsilon(H) > 0 $ such that every undirected graph on $ n $ vertices that does not contain $H$ as an induced subgraph contains a clique or a stable set of size at least $ n^{\epsilon(H)} $. The conjecture has a directed equivalent version stating that for every tournament $H$ there exists $ \epsilon(H) > 0 $ such that every $H-$free $n-$vertex tournament $T$ contains a transitive subtournament of order at least $ n^{\epsilon(H)} $. Both the directed and the undirected versions of the conjecture are known to be true for small graphs (tournaments). So far the conjecture was proved only for some specific families of prime tournaments, tournaments constructed according to the so$-$called substitution procedure allowing to build bigger graphs, and for all five$-$vertex tournaments. Recently the conjecture was proved for all six$-$vertex tournament, with one exception, but the question about the correctness of the conjecture for all seven$-$vertex tournaments remained open. In this paper we prove the correctness of the conjecture for several seven$-$vertex tournaments.
\end{abstract}

\section{Introduction}
Let $ G $ be an undirected graph. We denote by $ V(G) $ the set of its vertices and by $ E(G) $ the set of its edges. We call $ \mid$$G$$\mid :=$ $ \mid$$V(G)$$\mid$ the \textit{size} of $G$. Let $X \subseteq V(G)$. The \textit{subgraph of} $G$ \textit{induced by} $X$ is denoted by $G$$\mid$$X$. A \textit{clique} in $G$ is a set of pairwise adjacent vertices and a \textit{stable set} in $G$ is a set of pairwise nonadjacent vertices. Let $D$ and $D'$ be two digraphs. We say that $D'$ is a \textit{subdigraph} of $D$ if $V(D') \subseteq V(D)$ and $E(D') \subseteq E(D)$. We say that $D$ \textit{contains} $D'$ if $D'$ is isomorphic to a subdigraph of $D$. A \textit{tournament} is a directed graph such that for every pair $u$ and $v$ of vertices, exactly one of the arcs $(u,v)$ or $(v,u)$ exists. If two tournaments $T$ and $H$ are isomorphic, we write $T$$\approx$$H$, else we write $T$$\ncong$$H$. Two tournaments $H$ and $T$ are \textit{symmetric} if $H^{c}$$\approx$$T$. A tournament is \textit{transitive} if it contains no directed cycle. A \textit{cyclic triangle} is a directed cycle of length $3$. For a tournament $H$ and a vertex $v \notin V(H)$, we denote by $H + v$ a tournament obtained from $H$ with $V(H)\cup\lbrace v \rbrace$ as the vertex set. We denote by $H^{c}$ the tournament obtained from $H$ by reversing directions of all arcs of $H$. If $(u,v)\in E(H)$, then we say that $u$ is \textit{adjacent to} $v$ (alternatively: $v$ is an out-neighbor of $u$), and we write $u\rightarrow v$. In this case, we also say that $v$ is \textit{adjacent from} $u$ (alternatively: $u$ is an in-neighbor of $v$), and we write $v\leftarrow u$. The \textit{out-neighborhood} $N^{+}(x)$ of a vertex $x$ in a tournament $H$ is the set of all out-neighbors of $x$. The \textit{in-neighborhood} $N^{-}(x)$ of a vertex $x$ in a tournament $H$ is the set of all in-neighbors of $x$. The \textit{out-degree of $x$ in} $H$, denoted by $d^{+}(x)$, is  $ \mid$$N^{+}(x)$$\mid$, and the \textit{in-degree of $x$ in} $H$, denoted by $d^{-}(x)$, is  $ \mid$$N^{-}(x)$$\mid$. We say that $H$ is \textit{regular} if all its vertices have the same out-degree. Let $X \subseteq V(H)$. The \textit{subtournament of} $H$ \textit{induced by} $X$ is denoted by $H$$\mid$$X$. Denote by $N^{+}_{X}(x)$ the set of all out-neighbors of $x$ in $X$, and  by $N^{-}_{X}(x)$ the set of all in-neighbors of $x$ in $X$. Let $d^{+}_{X}(x) := \mid$$N^{+}_{X}(x)$$\mid$ and let $d^{-}_{X}(x) := \mid$$N^{-}_{X}(x)$$\mid$. Let $S$ be a tournament. We say that $H$ \textit{contains} $S$ if $S$ is isomorphic to $H$$\mid$$X$ for some $X \subseteq V(H)$. If $H$ does not contain $S$, we say that $H$ is $S$$-$$free$.  For two sets of vertices $V_{1},V_{2}$ of $H$, we say that $V_{1}$ is \textit{complete to} $V_{2}$ (equivalently $V_{2}$ is \textit{complete from} $V_{1}$) if every vertex of $V_{1}$ is adjacent to every vertex of $V_{2}$. We say that a vertex $v$ is \textit{complete to} (resp. \textit{from}) a set $V$ if $\lbrace v \rbrace$ is complete to (resp. from) $V$, and we write $v \rightarrow V$ (resp. $v \leftarrow V$). If $H$ does not contain the tournaments $S_{1},...,S_{k}$ as subtournaments, we say that $H$ is $(S_{1},...,S_{k})-free$. 

Erd\"{o}s and Hajnal proposed the following conjecture \cite{jhp} (EHC):
\begin{conjecture} For any undirected graph $H$ there exists $ \epsilon(H) > 0 $ such that any $ H- $free undirected graph with $n$ vertices contains a clique or a stable set of size at least $ n^{\epsilon(H)}. $
\end{conjecture}
In 2001 Alon et al. proved \cite{fdo} that Conjecture $1$ has an equivalent directed version, where undirected graphs are replaced by tournaments and cliques and stable sets by transitive subtournaments, as follows:
\begin{conjecture} \label{a} For any tournament $H$ there exists $ \epsilon(H) > 0 $ such that every $ H- $free tournament with $n$ vertices contains a transitive subtournament of size at least $ n^{\epsilon(H)}. $
\end{conjecture}
A tournament $H$ \textit{satisfies the Erd\"{o}s-Hajnal Conjecture (EHC)} if there exists $ \epsilon(H) > 0 $ such that every $ H- $free tournament $T$ with $n$ vertices contains a transitive subtournament of size at least $ n^{\epsilon(H)}. $ \\
A class of tournaments $\mathcal{F}$ \textit{satisfy the Erd\"{o}s-Hajnal Conjecture (EHC)} (equivalently: $\mathcal{F}$ has the \textit{Erd\"{o}s-Hajnal property}) if there exists $ \epsilon(\mathcal{F}) > 0 $ such that every $\mathcal{F}$$- $free tournament $T$ with $n$ vertices contains a transitive subtournament of size at least $ n^{\epsilon(\mathcal{F})}. $ If $\lbrace H\rbrace$ satisfy $EHC$ we simply say that $H$ \textit{satisfies EHC}.

A set of vertices $S\subseteq V(H)$ of a tournament $H$ is called \textit{homogeneous} if for every $v\in V(H)\backslash S$, $\lbrace v \rbrace$ is complete to/from $S$. A homogeneous set $S$ is called \textit{non-trivial} if $\mid$$S$$\mid > 1$ and $S\neq V(H)$. A tournament is called \textit{prime} if it does not have non-trivial homogeneous sets.
The following theorem which is applied to tournaments and a corollary of the results in \cite{fdo} shows the importance of prime tournaments.
\begin{theorem} \label{o}
If Conjecture \ref{a} is false then the smallest counterexample is prime.
\end{theorem}
For an integer $t$, we call the complete bipartite graph $ K_{1,t} $ a \textit{star}. Let $S$ be a star with vertex set $ \lbrace $$c$,\textit{$l_{1}$},...,\textit{$l_{t}$}$\rbrace $, where $c$ is adjacent to \textit{$l_{1}$},...,\textit{$l_{t}$}. We call $c$ the \textit{center of the star}, and \textit{$l_{1}$},...,\textit{$l_{t}$} the \textit{leaves of the star}. Note that in the case $ t = 1 $ we may choose arbitrarily any one of the two vertices to be the center of the star, and the other vertex is then considered to be the leaf. Let $ \theta = (v_{1},...,v_{n}) $ be an ordering of the vertex set $V(T)$ of an $ n- $vertex tournament $T$. We say that a vertex $ v_{j} $ is \textit{between two vertices $ v_{i},v_{k} $ under} $ \theta = (v_{1},...,v_{n}) $ if $ i < j < k $ or $ k < j < i $.
An arc $ (v_{i},v_{j}) $ is a \textit{backward arc under} $ \theta $ if $ i > j $. The \textit{set of backward arcs of $T$ under} $ \theta $ is denoted by $E(\theta) $. The \textit{graph of backward arcs under} $ \theta $, denoted by $ B(T,\theta) $, is the undirected graph that has vertex set $V(T)$, and $ v_{i}v_{j} \in E(B(T,\theta)) $ if and only if $ (v_{i},v_{j}) $ or $ (v_{j},v_{i}) $ is a backward arc of $T$ under $ \theta $. A \textit{right star} in $ B(T,\theta) $ is an induced subgraph with vertex set $ \lbrace v_{i_{0}},...,v_{i_{t}} \rbrace $, such that $ B(T,\theta) $$ \mid $$ \lbrace v_{i_{0}},...,v_{i_{t}} \rbrace $ is a star with center $ v_{i_{t}} $, and $ i_{t} > i_{0},...,i_{t-1} $. In this case we also say that $ \lbrace v_{i_{0}},...,v_{i_{t}} \rbrace $ is a right star in $T$ under $\theta$.
A \textit{left star} in $ B(T,\theta) $ is an induced subgraph with vertex set $ \lbrace v_{i_{0}},...,v_{i_{t}} \rbrace $, such that $ B(T,\theta) $$ \mid $$ \lbrace v_{i_{0}},...,v_{i_{t}} \rbrace $ is a star with center $ v_{i_{0}} $, and $ i_{0} < i_{1},...,i_{t} $. In this case we also say that $ \lbrace v_{i_{0}},...,v_{i_{t}} \rbrace $ is a left star in $T$ under $\theta$.
A \textit{star} in $ B(T,\theta) $, is a left star or a right star. A $4$-\textit{path} in $ B(T,\theta) $, denoted by $P_4$,  is an induced subgraph with vertex set $\{v_{j_1},v_{j_2},v_{j_3},v_{j_4}\}$, such that $P_4$ is a path, and its ends are consecutive under $\theta$. In this case we also say that $ \lbrace v_{j_1},v_{j_2},v_{j_3},v_{j_4} \rbrace $ is a $4$-path in $T$ under $\theta$.\vspace{2mm} 

Let $T$ be a tournament and assume that there exists an ordering $ \theta $ of its vertices such that every connected component of $ B(T,\theta) $ is either a star or a singleton. In this case $T$ is called \textit{super galaxy} and $\theta$ is called a \textit{star ordering}. If in addition no center of a star is between leaves of another star, then the corresponding ordering is called a \textit{galaxy ordering} and $T$ is called a \textit{galaxy}. And if this star ordering satisfy the condition that is fully characterized in \cite{kg}, then the corresponding ordering is called a \textit{constellation ordering} and $T$ is called a \textit{constellation}. Constellations is a larger family of tournaments than galaxies \cite{kg}. In the galaxy case no center of a star is between leaves of another star. However in the constellation scenario this condition is replaced by a much weaker one.
\begin{theorem} \cite{polll} \label{b}
Every galaxy satisfies the Erd\"{o}s$ - $Hajnal Conjecture.
\end{theorem}
\begin{theorem} \cite{kg} \label{c}
Every constellation satisfies the Erd\"{o}s-Hajnal Conjecture.
\end{theorem}
Denote by $ K_{6} $ the $6$$-$vertex tournament with vertex set $\lbrace v_{1},...,v_{6} \rbrace $, such that under the ordering $\theta = (v_{1},...,v_{6}) $ of its vertices, the set of backward arcs $E(\theta)$ is $ \lbrace (v_{4},v_{1}),(v_{6},v_{3}),(v_{6},v_{1}),(v_{5},v_{2}) \rbrace $. We call this ordering $\theta$ the \textit{canonical ordering}. Note that $ K_{6} $ is a prime, not galaxy and the conjecture for $ K_{6} $ remains an open question.
\begin{theorem} \cite{bnmm} \label{d}
If $H$ is a $6$$ - $vertex tournament not isomorphic to $ K_{6} $ then it satisfies the E$-$H conjecture.
\end{theorem}
Tournament $ C_{5} $ is the unique tournament on five vertices such that each of its vertices has exactly two out-neighbors and two in-neighbors. Tournament $ C_{5} $ is prime and not a galaxy. Let $v_{1},...,v_{5}$ be the vertices of $ C_{5} $ such that $E(C_{5})=\lbrace (v_{1},v_{2}),(v_{2},v_{3}),(v_{3},v_{4}),(v_{4},v_{5}),(v_{5},v_{1}),(v_{2},v_{4}),(v_{1},v_{3}),(v_{5},v_{2}),(v_{4},v_{1}),(v_{3},v_{5}) \rbrace $.

Tournament $ L_{1} $ is obtained from $ C_{5} $ by adding one extra vertex $v_{6}$ and making it adjacent to exactly one vertex of $ C_{5} $. It does not matter to which one since all tournaments obtained by this procedure are isomorphic, then we may assume without loss of generality that $(v_{6},v_{5}) \in E(L_{1})$.

Tournament $ L_{2} $ is obtained from $ C_{5} $ by adding one extra vertex $v_{6}$ and making it adjacent from $ 3 $ vertices of $ C_{5} $ that induce a cyclic triangle. It does not matter which cyclic triangle since all tournaments obtained by this procedure are isomorphic, then we may assume without loss of generality that $\lbrace v_{6} \rbrace$ is adjacent from $\lbrace v_{1},v_{2},v_{4} \rbrace$.\vspace{2mm}\\
In this paper we prove a theorem that describes the structure of all seven-vertex tournaments, and we prove the Erd\"{o}s-Hajnal conjecture for several $(K_{6},L_{1},L_{2})-$free $7$-vertex tournaments having no constellation ordering, some $7$-vertex tournaments containing $C_5$, and for couples and triplets of $7$-vertex tournaments.
This paper is organized as follows:
\begin{itemize}
\item In section $2$ we define special classes of tournaments, and we state a theorem that describes the structure of all tournaments on seven vertices and reduces the question  about the correctness of EHC for seven-vertex tournaments to special classes.
\item In section $3$ we study the structure of $\epsilon$-critical tournaments with forbidden subtournaments.
\item In section $4$ we prove  EHC for some classes of seven-vertex tournaments.
\item In section $5$ we worke on the analogous conjecture of Erd\"{o}s and Hajnal and confirm it for couples and triplets of tournaments.
\item In section $6$ we give an elementary proof of Theorem \ref{r}.
\end{itemize}

\section{Analyzing the structure of 7-vertex tournaments}
In this section we define three crucial classes of $7$-vertex tournaments, denoted by $\mathcal{S}$, $\mathcal{R}$, and $\mathcal{H}$. Then we state a theorem that describes the structure of all $7$-vertex tournaments and reduces the question about the correctness of EHC for tournaments on $7$ vertices to the tournaments in the classes $\mathcal{S}$, $\mathcal{R}$, and $\mathcal{H}$. To this end we need to define tournaments on seven vertices.\vspace{2mm}\\ 
 We will start by defining the prime, non-isomorphic, non-symmetric $(K_{6},L_{1},L_{2})-$free 7-vertex tournaments,  having no constellation ordering, denoted by $ S_{i} $ for $i=1,...,15$. Let $V(S_{i}) = \lbrace a,b,c,d,e,f,v \rbrace$ for $i=1,...,15$
\vspace{1mm}\\
Let $i\in\{1,...,15\}$. To make it easy define $S_{i}$, we will define an ordering $ \theta_{i} $ of $V(S_i)$ and the set $E(\theta_{i})$ of backward arcs under this ordering $E(\theta_{i})$.
\begin{itemize}[-]
\item $S_{1}$: $\theta_{1} := (d,f,v,a,b,e,c)$, $E(\theta_{1}) = \lbrace (v,d),(b,d),(c,d),(e,f),(e,a)\rbrace$. 
\item $S_{2}$: $\theta_{2} := (v,d,e,f,a,b,c)$, $E(\theta_{2}) = \lbrace (f,v),(a,d),(b,d),(c,a) \rbrace $.
\item $S_{3}$: $ \theta_{3} := (f,v,b,e,c,d,a)$, $E(\theta_{3}) = \lbrace (e,f),(d,e),(d,f),(a,b) \rbrace$. 
\item $S_{4}$: $\theta_{4} := (v,d,f,c,a,e,b)$, $E(\theta_{4}) = \lbrace (f,v),(e,f),(b,c),(b,d) \rbrace $.
\item $S_{5}$: $\theta_{5} := (v,a,b,c,f,e,d)$, $E(\theta_{5}) = \lbrace (f,v),(d,a),(e,b),(d,b) \rbrace $.
\item $S_{6}$: $\theta_{6} := (v,a,b,c,f,e,d)$, $E(\theta_{6}) = \lbrace (f,v),(d,a),(e,b),(e,a) \rbrace $. 
\item $S_{7}$: $\theta_{7} := (v,b,e,a,c,f,d)$, $E(\theta_{7}) = \lbrace (f,v),(d,b),(d,c),(f,e),(a,b) \rbrace $. 
\item $S_{8}$: $\theta_{8} := (v,c,a,f,e,d,b)$, $E(\theta_{8}) = \lbrace (f,v),(b,c),(b,f),(d,a) \rbrace $.
\item $S_{9}$: $\theta_{9} := (v,a,b,c,d,f,e)$, $E(\theta_{9}) = \lbrace (f,v),(c,a),(e,c),(e,b) \rbrace $.
\item $S_{10}$: $\theta_{10} := (d,c,v,e,a,f,b)$, $E(\theta_{10}) = \lbrace (f,v),(f,e),(e,c),(a,d),(b,d) \rbrace $. 
\item $S_{11}$: $\theta_{11} := (b,c,v,a,d,e,f)$, $E(\theta_{11}) = \lbrace (f,c),(f,v),(d,v),(a,b),(e,b) \rbrace $.
\item $S_{12}$: $\theta_{12} := (b,c,a,f,v,d,e)$, $E(\theta_{12}) = \lbrace (d,a),(e,c),(d,b),(a,b),(e,f) \rbrace $.
\item $S_{13}$: $\theta_{13} := (b,c,d,a,f,v,e)$, $E(\theta_{13}) = \lbrace (f,b),(e,c),(v,d),(a,b),(e,f) \rbrace $.
\item $S_{14}$: $\theta_{14} := (b,c,d,a,e,v,f)$, $E(\theta_{14}) = \lbrace (f,b),(e,c),(v,d),(a,b),(f,e) \rbrace $.
\item $S_{15}$: $\theta_{15} := (a,b,d,c,e,v,f) $, $E(\theta_{15}) = \lbrace (e,a),(v,a),(f,a),(v,b),(f,b),(f,d) \rbrace $.
\end{itemize}
Now we will define the non-isomrphic, non-symmetric 7-vertex tournaments, denoted by $R_{i}$ for $i=1,...,11$, that are obtained from $K_{6}$ by adding one extra vertex $v_{7}$. The following are the in-neighbors of the added vertex $v_{7}$ in $R_{i}$ for $i = 1,2,...,11$.\\
$N^{-}_{R_{1}}(v_{7}) = \phi$, $N^{-}_{R_{2}}(v_{7}) = \lbrace v_{1} \rbrace$, $N^{-}_{R_{3}}(v_{7}) = \lbrace v_{1},v_{3},v_{4},v_{5},v_{6} \rbrace$, $N^{-}_{R_{4}}(v_{7}) = \lbrace v_{1},v_{2} \rbrace$, $N^{-}_{R_{5}}(v_{7}) = \lbrace v_{1},v_{2},v_{4} \rbrace$, $N^{-}_{R_{6}}(v_{7}) = \lbrace v_{1},v_{3},v_{4},v_{6} \rbrace$, $N^{-}_{R_{7}}(v_{7}) = \lbrace v_{1},v_{5} \rbrace$, $N^{-}_{R_{8}}(v_{7}) = \lbrace v_{1},v_{3},v_{5} \rbrace$, $N^{-}_{R_{9}}(v_{7}) = \lbrace v_{2},v_{5},v_{6} \rbrace$, $N^{-}_{R_{10}}(v_{7}) = \lbrace v_{5},v_{6} \rbrace$, $N^{-}_{R_{11}}(v_{7}) = \lbrace v_{3},v_{5},v_{6} \rbrace$.
\vspace{1.5mm}\\
Finally we will define the prime, non-isomrphic, non-symmetric 7-vertex tournaments, denoted by $H_{i}$ for $i=1,...,59$, that are obtained from $L_{1}$ or $L_{2}$ by adding one extra vertex $v_{7}$. The following are the out-neighbors or the in-neighbors of the added vertex $v_{7}$ in $H_{i}$ for $i = 1,2,...,59$.
\vspace{1.5mm}\\
$H_{1} := L_{1} + v_{7}$, such that $N^{+}(v_{7}) = \lbrace v_{6} \rbrace$. $H_{2} := L_{1} + v_{7}$, such that $N^{-}(v_{7}) = \lbrace v_{3} \rbrace$. $H_{3} := L_{1} + v_{7}$, such that $N^{-}(v_{7}) = \lbrace v_{3},v_{4} \rbrace$. $H_{4} := L_{1} + v_{7}$, such that $N^{-}(v_{7}) = \lbrace v_{2} \rbrace$. $H_{5} := L_{1} + v_{7}$, such that $N^{-}(v_{7}) = \lbrace v_{1},v_{3} \rbrace$. $H_{6} := L_{1} + v_{7}$, such that $N^{-}(v_{7}) = \lbrace v_{1} \rbrace$. $H_{7} := L_{2} + v_{7}$, such that $N^{+}(v_{7}) = \lbrace v_{3} \rbrace$. $H_{8} := L_{2} + v_{7}$, such that $N^{+}(v_{7}) = \lbrace v_{6} \rbrace$. $H_{9} := L_{1} + v_{7}$, such that $N^{-}(v_{7}) = \lbrace v_{4} \rbrace$. $H_{10} := L_{2} + v_{7}$, such that $N^{-}(v_{7}) = \lbrace v_{2} \rbrace$. $H_{11} := L_{1} + v_{7}$, such that $N^{-}(v_{7}) = \lbrace v_{1},v_{4} \rbrace$. $H_{12} := L_{2} + v_{7}$, such that $N^{+}(v_{7}) = \lbrace v_{5} \rbrace$. $H_{13} := L_{1} + v_{7}$, such that $N^{+}(v_{7}) = \lbrace v_{3} \rbrace$. $H_{14} := L_{1} + v_{7}$, such that $N^{-}(v_{7}) = \lbrace v_{3},v_{4},v_{5} \rbrace$. $H_{15} := L_{1} + v_{7}$, such that $N^{+}(v_{7}) = \lbrace v_{2} \rbrace$. $H_{16} := L_{2} + v_{7}$, such that $N^{+}(v_{7}) = \lbrace v_{1} \rbrace$. $H_{17} := L_{1} + v_{7}$, such that $N^{+}(v_{7}) = \lbrace v_{1} \rbrace$. $H_{18} := L_{2} + v_{7}$, such that $N^{-}(v_{7}) = \lbrace v_{3} \rbrace$. $H_{19} := L_{1} + v_{7}$, such that $N^{-}(v_{7}) = \lbrace v_{5} \rbrace$. $H_{20} := L_{2} + v_{7}$, such that $N^{-}(v_{7}) = \lbrace v_{6} \rbrace$. $H_{21} := L_{1} + v_{7}$, such that $N^{+}(v_{7}) = \lbrace v_{4} \rbrace$. $H_{22} := L_{1} + v_{7}$, such that $N^{+}(v_{7}) = \lbrace v_{1},v_{3},v_{6} \rbrace$. $H_{23} := L_{2} + v_{7}$, such that $N^{+}(v_{7}) = \lbrace v_{4} \rbrace$. $H_{24} := L_{2} + v_{7}$, such that $N^{-}(v_{7}) = \lbrace v_{5} \rbrace$. $H_{25} := L_{1} + v_{7}$, such that $N^{-}(v_{7}) = \lbrace v_{6} \rbrace$. $H_{26} := L_{1} + v_{7}$, such that $N^{+}(v_{7}) = \lbrace v_{4},v_{5} \rbrace$. $H_{27} := L_{2} + v_{7}$, such that $N^{-}(v_{7}) = \lbrace v_{2},v_{4} \rbrace$. $H_{28} := L_{2} + v_{7}$, such that $N^{-}(v_{7}) = \lbrace v_{1},v_{4} \rbrace$. $H_{29} := L_{1} + v_{7}$, such that $N^{+}(v_{7}) = \lbrace v_{1},v_{5} \rbrace$. $H_{30} := L_{2} + v_{7}$, such that $N^{+}(v_{7}) = \lbrace v_{3},v_{6} \rbrace$. $H_{31} := L_{2} + v_{7}$, such that $N^{+}(v_{7}) = \lbrace v_{5},v_{6} \rbrace$. $H_{32} := L_{2} + v_{7}$, such that $N^{+}(v_{7}) = \lbrace v_{1},v_{5} \rbrace$. $H_{33} := L_{2} + v_{7}$, such that $N^{+}(v_{7}) = \lbrace v_{2},v_{5} \rbrace$. $H_{34} := L_{2} + v_{7}$, such that $N^{+}(v_{7}) = \lbrace v_{3},v_{4} \rbrace$. $H_{35} := L_{2} + v_{7}$, such that $N^{-}(v_{7}) = \lbrace v_{4},v_{6} \rbrace$. $H_{36} := L_{2} + v_{7}$, such that $N^{-}(v_{7}) = \lbrace v_{1},v_{6} \rbrace$. $H_{37} := L_{1} + v_{7}$, such that $N^{+}(v_{7}) = \lbrace v_{3},v_{4},v_{5} \rbrace$. $H_{38} := L_{1} + v_{7}$, such that $N^{+}(v_{7}) = \lbrace v_{3},v_{4} \rbrace$. $H_{39} := L_{2} + v_{7}$, such that $N^{+}(v_{7}) = \lbrace v_{1},v_{2},v_{4} \rbrace$. $H_{40} := L_{1} + v_{7}$, such that $N^{-}(v_{7}) = \lbrace v_{4},v_{6} \rbrace$. $H_{41} := L_{2} + v_{7}$, such that $N^{-}(v_{7}) = \lbrace v_{2},v_{5} \rbrace$. $H_{42} := L_{1} + v_{7}$, such that $N^{-}(v_{7}) = \lbrace v_{1},v_{6} \rbrace$. $H_{43} := L_{2} + v_{7}$, such that $N^{-}(v_{7}) = \lbrace v_{1},v_{3} \rbrace$. $H_{44} := L_{2} + v_{7}$, such that $N^{-}(v_{7}) = \lbrace v_{2},v_{6} \rbrace$. $H_{45} := L_{1} + v_{7}$, such that $N^{-}(v_{7}) = \lbrace v_{2},v_{6} \rbrace$. $H_{46} := L_{2} + v_{7}$, such that $N^{+}(v_{7}) = \lbrace v_{2},v_{3},v_{5} \rbrace$. $H_{47} := L_{1} + v_{7}$, such that $N^{-}(v_{7}) = \lbrace v_{3},v_{6} \rbrace$. $H_{48} := L_{2} + v_{7}$, such that $N^{+}(v_{7}) = \lbrace v_{1},v_{3} \rbrace$. $H_{49} := L_{2} + v_{7}$, such that $N^{+}(v_{7}) = \lbrace v_{2},v_{3} \rbrace$. $H_{50} := L_{1} + v_{7}$, such that $N^{+}(v_{7}) = \lbrace v_{2},v_{3} \rbrace$. $H_{51} := L_{1} + v_{7}$, such that $N^{+}(v_{7}) = \lbrace v_{2},v_{4} \rbrace$. $H_{52} := L_{1} + v_{7}$, such that $N^{+}(v_{7}) = \lbrace v_{1},v_{3} \rbrace$. $H_{53} := L_{2} + v_{7}$, such that $N^{+}(v_{7}) = \lbrace v_{1},v_{2},v_{3} \rbrace$. $H_{54} := L_{1} + v_{7}$, such that $N^{-}(v_{7}) = \lbrace v_{5},v_{6} \rbrace$. $H_{55} := L_{2} + v_{7}$, such that $N^{-}(v_{7}) = \lbrace v_{3},v_{5} \rbrace$. $H_{56} := L_{1} + v_{7}$, such that $N^{+}(v_{7}) = \lbrace v_{1},v_{2},v_{3} \rbrace$. $H_{57} := L_{2} + v_{7}$, such that $N^{-}(v_{7}) = \lbrace v_{3},v_{6} \rbrace$. $H_{58} := L_{2} + v_{7}$, such that $N^{+}(v_{7}) = \lbrace v_{1},v_{2},v_{6} \rbrace$. $H_{59} := L_{2} + v_{7}$, such that $N^{+}(v_{7}) = \lbrace v_{2},v_{4} \rbrace$.
\vspace{1.5mm}\\
We are ready to define $\mathcal{S}$, $\mathcal{H}$, and $\mathcal{R}$.\\
$\cal S$ :$= \lbrace S_{i}: i = 1,...,15 \rbrace$, $\cal R$ $:= \lbrace R_{i}: i = 1,...,11 \rbrace$, and $\cal H$ $:= \lbrace H_{i}: i = 1,...,59 \rbrace$.\\
Note that their are some tournaments in tha above classes that are isomorphic to their complement. To this end let us define the following classes, denoted by $\mathbf{S}$, $\mathbf{H}$, $\mathbf{R}$.\\
$\mathbf{S}$ $:= \lbrace S_{i},S_{i}^{c}: i = 1,...,15 \rbrace$, $\mathbf{R}$ $:= \lbrace R_{i},R_{i}^{c}: i = 1,...,11 \rbrace$, and $\mathbf{H}$ $:= \lbrace H_{i},H_{i}^{c}: i = 1,...,59 \rbrace$.\vspace{1.5mm}\\
The following theorem describes the structure of all $7$-vertex tournament:
\begin{theorem} \label{r}
Let $H$ be a seven-vertex tournament. Then one of the following holds:\\
$ 1.$ $ H $ is a galaxy.\\
$ 2.$ $ H$ is a constellation and not a galaxy.\\
$ 3.$ $ H$ belongs to class $\mathbf{H}$.\\
$ 4.$ $ H$ belongs to class $\mathbf{R}$.\\
$ 5.$ $ H$ is not prime $K_{6}-$free tournament. \\
$ 6.$ $ H$ belongs to class $\mathbf{S}$. 
\end{theorem}
The proof of Theorem \ref{r} is presented in Section \ref{landscape}.\\
Theorem \ref{r} and the hereditary property of the Erd\"{o}s-Hajnal conjecture implies the following theorem that reduces the question for the correctness of the conjecture for seven-vertex tournaments to the tournaments in classes $\cal S$, $\cal R$, $\cal H$ and which states:
\begin{theorem}
If $H$ is a seven$ - $vertex tournament not in class $\mathbf{H}$, class $\mathbf{R}$, or class $\mathbf{S}$ then $H$ satisfies the EHC.
\end{theorem}
\begin{proof} We will use Theorem \ref{r}. If (1) holds, then the result follows from Theorem \ref{b}. If (2) holds, then the result follows from Theorem \ref{c}. If (5) holds, then the result follows from Theorem \ref{d} and Theorem \ref{o}. If (3) holds, then $H \in$ $\mathbf{H}$. If (4) holds, then $H \in$ $\mathbf{R}$, and finally if (6) holds, then $H \in$ $\mathbf{S}$.  $\hfill {\square}$ \vspace{1.5mm}\\
\end{proof}
\section{The structure of $H$-free $\epsilon$-critical tournaments}
Let $H$ be a tournament. In this section  we prove  structural properties of $\epsilon$-critical tournaments and $H$-free $\epsilon$-critical tournaments, and we borrow from \cite{polll,bnmm} some properties of $\epsilon$-critical tournaments.\\
Denote by $tr(T)$ the largest size of a transitive subtournament of a tournament $T$. For $X \subseteq V(T)$, write $tr(X)$ for $tr(T$$\mid$$X)$. Let $X, Y \subseteq V(T)$ be disjoint. Denote by $e_{X,Y}$ the number of directed arcs $(x,y)$, where $x \in X$ and $y \in Y$. The \textit{directed density from $X$ to} $Y$ is defined as $d(X,Y) = \frac{e_{X,Y}}{\mid X \mid.\mid Y \mid} $. We call $T$ $ \epsilon - $\textit{critical} for $ \epsilon > 0 $ if $tr(T) < $ $ \mid $$T$$ \mid^{\epsilon} $ but for every proper subtournament $S$ of $T$ we have: $tr(S) \geq $ $ \mid $$S$$ \mid^{\epsilon}. $
\begin{lemma} \cite{polll} \label{e} For every $N$ $ > 0 $, there exists $ \epsilon(N) > 0 $ such that for every $ 0 < \epsilon < \epsilon(N)$ every $ \epsilon $$-$critical tournament $T$ satisfies $ \mid $$T$$\mid$ $ \geq N$.
\end{lemma}
\begin{proof}
Since every tournament contains a transitive subtournament of size $2$ so it suffices to take $\epsilon(N) = log_{N}(2)$. $\hfill {\square}$
\end{proof}
\begin{lemma} \cite{polll} \label{f} Let $T$ be an $ \epsilon $$-$critical tournament with $\mid$$T$$\mid$ $ =n$ and $\epsilon$,$c,f > 0 $ be constants such that $ \epsilon < log_{c}(1 - f)$. Then for every $A \subseteq V(T)$ with $ \mid $$A$$ \mid$ $ \geq cn$ and every transitive subtournament $G$ of $T$ with $ \mid $$G$$ \mid$ $\geq f.tr(T)$ and $V(G) \cap A = \phi$, we have: $A$ is not complete from $V(G)$ and $A$ is not complete to $V(G)$.
\end{lemma}
\begin{lemma} \cite{polll} \label{v} Let $T$ be an $ \epsilon $$-$critical tournament with $\mid$$T$$\mid = n$ and $\epsilon$,$c > 0 $ be constants such that $ \epsilon < log_{\frac{c}{2}}(\frac{1}{2}). $ Then for every two disjoint subsets $X, Y \subseteq V(T)$ with $ \mid $$X$$ \mid$ $ \geq cn, \mid $$Y$$ \mid$ $ \geq cn $ there exist an integer $k \geq \frac{cn}{2} $ and vertices $ x_{1},...,x_{k} \in X $ and $ y_{1},...,y_{k} \in Y $ such that $ y_{i} $ is adjacent to $ x_{i} $ for $i = 1,...,k. $
\end{lemma}
\begin{lemma} \cite{polll} \label{s} Let $A_{1},A_{2}$ be two disjoint sets such that $d(A_{1},A_{2}) \geq 1-\lambda$ and let $0 < \eta_{1},\eta_{2} \leq 1$. Let $\widehat{\lambda} = \frac{\lambda}{\eta_{1}\eta_{2}}$. Let $X \subseteq A_{1}, Y \subseteq A_{2}$ be such that $\mid$$X$$\mid$ $\geq \eta_{1} \mid$$A_{1}$$\mid$ and $\mid$$Y$$\mid$ $\geq \eta_{2} \mid$$A_{2}$$\mid$. Then $d(X,Y) \geq 1-\widehat{\lambda}$. 
\end{lemma}
\begin{lemma} \cite{ml} \label{h}
Every tournament on $2^{k-1}$ vertices contains a transitive subtournament of size at least $k$.
 \end{lemma}
The following is introduced in \cite{bnmm}.\\
Let $ c > 0, 0 < \lambda < 1 $ be constants, and let $w$ be a $ \lbrace 0,1 \rbrace - $ vector of length $ \mid $$w$$ \mid $. Let $T$ be a tournament with $ \mid $$T$$ \mid$ $ = n. $ A sequence of disjoint subsets $ \chi = (A_{1}, A_{2},..., A_{\mid w \mid}) $ of $V(T)$ is a smooth $ (c,\lambda, w)- $structure if:\\
$\bullet$ whenever $ w_{i} = 0 $ we have $ \mid $$A_{i}$$ \mid$ $ \geq cn $ (we say that $ A_{i} $ is a \textit{linear set}).\\
$\bullet$ whenever $ w_{i} = 1 $ the tournament $T$$\mid$$ A_{i} $ is transitive and $ \mid $$A_{i}$$ \mid$ $ \geq c.tr(T) $ (we say that $ A_{i} $ is a \textit{transitive set}).\\
$\bullet$ $ d(\lbrace v \rbrace, A_{j}) \geq 1 - \lambda $ for $v \in A_{i} $ and $ d(A_{i}, \lbrace v \rbrace) \geq 1 - \lambda $ for $v \in A_{j}, i < j $ (we say that $\chi$ is \textit{smooth}).
\begin{theorem} \cite{bnmm} \label{t}
Let $S$ be a tournament, let $w$ be a $ \lbrace 0,1 \rbrace - $vector, and let $ 0 < \lambda_{0} < \frac{1}{2} $ be a constant. Then there exist $ \epsilon_{0}, c_{0} > 0 $ such that for every $ 0 < \epsilon < \epsilon_{0} $, every $ S- $free $ \epsilon $$- $critical tournament contains a smooth $ (c_{0}, \lambda_{0},w)- $structure.
\end{theorem}

\begin{corollary} \label{i}
Let $S$ be a tournament, let $w$ be a $ \lbrace 0,1 \rbrace - $vector, and let $m \in \mathbb{N}^{*}$, $ 0 < \lambda_{1} < \frac{1}{2} $ be constants. Then there exist $ \epsilon_{1}, c_{1} > 0 $ such that for every $ 0 < \epsilon < \epsilon_{1} $, every $ S- $free $ \epsilon $$- $critical tournament contains a smooth $ (c_{1}, \lambda_{1},w)- $structure $(F_{1},...,F_{\mid w \mid})$ such that $\forall 1 \leq j \leq$ $ \mid$$ w$$ \mid$, $\mid$$F_{j}$$\mid$ is divisible by $m$.
\end{corollary}
\begin{proof}
Let $\lambda_{0} = \frac{\lambda_{1}}{2}$. By Theorem \ref{t}, there exist $ \epsilon_{0}, c_{0} > 0 $ such that for every $ 0 < \epsilon < \epsilon_{0} $, every $ S- $free $ \epsilon $$- $critical tournament contains a smooth $ (c_{0}, \lambda_{0},w)- $structure. Denote this structure by $(A_{1},...,A_{\mid w \mid})$. Now $\forall 1 \leq j \leq$ $ \mid$$ w$$ \mid$ take an arbitrary subset $F_{j}$ of $A_{j}$ such that $\mid$$F_{j}$$\mid$ is divisible by $m$ and $\mid$$F_{j}$$\mid$ is of maximum size. Notice that $\forall 1 \leq j \leq$ $ \mid$$ w$$ \mid$, $\mid$$F_{j}$$\mid$ $\geq$ $ \mid$$A_{j}$$\mid -$ $ m$. Taking $\epsilon_{0}$ small enough we may assume that $tr(T) \geq \frac{2m}{c_{0}}$ by Lemmas \ref{h} and \ref{e}.\\
$\bullet$ $\forall 1 \leq j \leq$ $ \mid$$ w$$ \mid$ with $w(j) = 1$ we have: $\mid$$A_{j}$$\mid$ $ \geq c_{0}tr(T) \geq c_{0}\frac{2m}{c_{0}} = 2m$. Then $\mid$$F_{j}$$\mid$ $\geq$ $ \mid$$A_{j}$$\mid -$ $ m \geq$ $ \mid$$A_{j}$$\mid - $ $\frac{\mid A_{j}\mid}{2} = \frac{\mid A_{j}\mid}{2}$.\\
$\bullet$ $\forall 1 \leq j \leq$ $ \mid$$ w$$ \mid$ with $w(j) = 0$ we have: $\mid$$A_{j}$$\mid$ $ \geq c_{0}n \geq c_{0}tr(T) \geq 2m$. Then $\mid$$F_{j}$$\mid$ $\geq$ $ \frac{\mid A_{j}\mid}{2}$.\\
Now Lemma \ref{s} implies that $(F_{1},...,F_{\mid w \mid})$ is a smooth $(\frac{c_{0}}{2}, 2\lambda_{0},w)-$structure of $T$. Thus taking $c_{1} = \frac{c_{0}}{2}$, we complete the proof. $\hfill {\square}$ \\ 
\end{proof}

Let $(A_{1},...,A_{\mid w \mid})$ be a smooth $(c,\lambda ,w)$$-$structure of a tournament $T$, let $i \in \lbrace 1,...,\mid$$w$$\mid \rbrace$, and let $v \in A_{i}$. For $j\in \lbrace 1,2,...,\mid$$w$$\mid \rbrace \backslash \lbrace i \rbrace$, denote by $A_{j,v}$ the set of the vertices of $A_{j}$ adjacent from $v$ for $j > i$ and adjacent to $v$ for $j<i$.
\begin{lemma} \label{u} Let $0<\lambda<1$, $0<\gamma \leq 1$ be constants and let $w$ be a $\lbrace 0,1 \rbrace$$-$vector. Let $(A_{1},...,A_{\mid w \mid})$ be a smooth $(c,\lambda ,w)$$-$structure of a tournament $T$ for some $c>0$. Let $j\in \lbrace 1,...,\mid$$w$$\mid \rbrace$. Let $A_{j}^{*}\subseteq A_{j}$ such that $\mid$$A_{j}^{*}$$\mid$ $\geq \gamma \mid$$A_{j}$$\mid$ and let $A= \lbrace x_{1},...,x_{k} \rbrace \subseteq \displaystyle{\bigcup_{i\neq j}A_{i}}$ for some positive integer $k$. Then $\mid$$\displaystyle{\bigcap_{x\in A}A^{*}_{j,x}}$$\mid$ $\geq (1-k\frac{\lambda}{\gamma})\mid$$A_{j}^{*}$$\mid$. In particular $\mid$$\displaystyle{\bigcap_{x\in A}A_{j,x}}$$\mid$ $\geq (1-k\lambda)\mid$$A_{j}$$\mid$.
\end{lemma}
\begin{proof}
The proof is by induction on $k$. without loss of generality assume that $x_{1} \in A_{i}$ and $j<i$. Since $\mid$$A_{j}^{*}$$\mid$ $\geq \gamma \mid$$A_{j}$$\mid$ then by Lemma \ref{s}, $d(A^{*}_{j},\lbrace x_{1}\rbrace) \geq 1-\frac{\lambda}{\gamma}$. So $1-\frac{\lambda}{\gamma} \leq d(A^{*}_{j},\lbrace x_{1}\rbrace) = \frac{\mid A^{*}_{j,x_{1}}\mid}{\mid A_{j}^{*}\mid}$. Then $\mid$$A^{*}_{j,x_{1}}$$\mid$ $\geq (1-\frac{\lambda}{\gamma})$$\mid$$A_{j}^{*}$$\mid$ and so true for $k=1$.
Suppose the statement is true for $k-1$.\\ $\mid$$\displaystyle{\bigcap_{x\in A}A^{*}_{j,x}}$$\mid$ $=\mid$$(\displaystyle{\bigcap_{x\in A\backslash \lbrace x_{1}\rbrace}A^{*}_{j,x}})\cap A^{*}_{j,x_{1}}$$\mid$ $= \mid$$\displaystyle{\bigcap_{x\in A\backslash \lbrace x_{1}\rbrace}A^{*}_{j,x}}$$\mid$ $+$ $\mid$$A^{*}_{j,x_{1}}$$\mid$ $- \mid$$(\displaystyle{\bigcap_{x\in A\backslash \lbrace x_{1}\rbrace}A^{*}_{j,x}})\cup A^{*}_{j,x_{1}}$$\mid$ $\geq (1-(k-1)\frac{\lambda}{\gamma})\mid$$A_{j}^{*}$$\mid$ $+$ $(1-\frac{\lambda}{\gamma})\mid$$A_{j}^{*}$$\mid$ $-$ $\mid$$A_{j}^{*}$$\mid$ $= (1-k\frac{\lambda}{\gamma})\mid$$A_{j}^{*}$$\mid$. $\hfill {\square}$       
\end{proof} 
\begin{lemma}\label{g}
Let $f,c,\epsilon > 0$ be constants, where $0 <  f,c < 1$ and $0 < \epsilon < min\lbrace log_{\frac{c}{2}}(1-f), log_{\frac{c}{4}}(\frac{1}{2})\rbrace$. Let $T$ be an $ \epsilon $$-$critical tournament with $\mid$$T$$\mid$ $ =n$, and   let $T_{1},T_{2}$ be two disjoint transitive subtournaments of $T$ with $ \mid $$T_{1}$$ \mid$ $\geq f.tr(T)$ and $ \mid $$T_{2}$$ \mid$ $\geq f.tr(T)$. Let $A_{1},A_{2}$ be two disjoint subsets of $V(T)$ with $ \mid $$A_{1}$$ \mid$ $\geq cn$, $ \mid $$A_{2}$$ \mid$ $\geq cn$, and $A_{1},A_{2} \subseteq V(T) \backslash (V(T_{1})\cup V(T_{2}))$. Then there exist vertices $a,x,s_{1},s_{2}$ such that $a\in A_{1}, x\in A_{2}, s_{1}\in T_{1}, s_{2}\in T_{2}$, $\lbrace a,s_{1}\rbrace \leftarrow x$, and $a\leftarrow s_{2}$.
Similarly, there exist vertices $d_{1},d_{2},u_{1},u_{2}$ such that $d_{1}\in A_{1}, d_{2}\in A_{2}, u_{1}\in T_{1}, u_{2}\in T_{2}$, $\lbrace d_{1},u_{1}\rbrace \leftarrow d_{2}$, and $u_{2}\leftarrow d_{1}$.
 \end{lemma}
\begin{proof}
We will prove only the first statement  because the latter can be proved analogously. 
We have $\mid$$A_{1}$$\mid$ $\geq cn$ and $\mid$$A_{2}$$\mid$ $\geq cn$. Let $A_{1}^{*} = \lbrace a \in A_{1}; \exists s \in T_{2}$ and $a \leftarrow s \rbrace$ and let $A_{2}^{*} = \lbrace x \in A_{2}; \exists v \in T_{1}$ and $v \leftarrow x \rbrace$. Then $A_{1}\backslash A_{1}^{*}$ is complete to $T_{2}$ and $A_{2}\backslash A_{2}^{*}$ is complete from $T_{1}$. Now assume that $\mid$$A_{1}^{*}$$\mid$ $< \frac{\mid A_{1} \mid}{2}$, then $\mid$$A_{1}\backslash A_{1}^{*}$$\mid$ $\geq \frac{\mid A_{1} \mid}{2} \geq \frac{c}{2}n$. Since $\mid$$ T_{2}$$\mid$ $\geq f.tr(T)$ and since $\epsilon < log_{\frac{c}{2}}(1-f)$, then Lemma \ref{f} implies that $A_{1}\backslash A_{1}^{*}$ is not complete to $T_{2}$, a contradiction. Then $\mid$$A_{1}^{*}$$\mid$ $\geq \frac{\mid A_{1} \mid}{2} \geq \frac{c}{2}n$. Similarly we prove that $\mid$$A_{2}^{*}$$\mid$ $\geq  \frac{c}{2}n$. Now since $\epsilon < log_{\frac{c}{4}}(\frac{1}{2})$, then Lemma \ref{v} implies that $\exists k \geq \frac{c}{4}n$, $\exists a_{1},...,a_{k} \in A_{1}^{*}$, $\exists x_{1},...,x_{k} \in A_{2}^{*}$, such that $a_{i} \leftarrow x_{i}$ for $i = 1,...,k$.  So, $\exists a_{1} \in A_{1}^{*}$, $\exists x_{1} \in A_{2}^{*}$, $\exists s_{1} \in T_{1}$, $\exists s_{2} \in T_{2}$ such that $\lbrace a_{1},s_{1}\rbrace \leftarrow x_{1}$ and $a_{1}\leftarrow s_{2}$.  $\hfill {\square}$
\end{proof}
\begin{lemma} \label{q}
Let $f_{1},...,f_{m},c,\epsilon > 0$ be constants, where $0 <  f_{1},...,f_{m},c < 1$ and $0 < \epsilon < log_{\frac{c}{2m}}(1-f_{i})$ for $i=1,...,m$. Let $T$ be an $ \epsilon $$-$critical tournament with $\mid$$T$$\mid$ $ =n$, and   let $T_{1},...,T_{m}$ be m disjoint transitive subtournaments of $T$ with $ \mid $$T_{i}$$ \mid$ $\geq f_{i}.tr(T)$ for $i=1,...,m$. Let $A \subseteq V(T) \backslash (\bigcup_{i=1}^{m} V(T_{i}))$ with $ \mid $$A$$ \mid$ $ \geq cn$.
Then there exist vertices $s_{1},...,s_{m},a$ such that $a\in A$, $s_{i}\in T_{i}$ for $i=1,...,m$, and $\lbrace a \rbrace$ is complete to $\lbrace s_{1},...,s_{m} \rbrace$. Similarly there exist vertices $u_{1},...,u_{m},b$ such that $b\in A$, $u_{i}\in T_{i}$ for $i=1,...,m$, and $\lbrace b \rbrace$ is complete from $\lbrace u_{1},...,u_{m} \rbrace$.
\end{lemma}
\begin{proof}
We will prove only the first statement  because the latter can be proved analogously.
Let $A_{i} \subseteq A$ such that $A_{i}$ is complete from $T_{i}$ for $i = 1,...,m$. Let $1\leq j \leq m$. If $\mid$$A_{j}$$\mid \geq \frac{\mid A \mid}{2m}\geq \frac{c}{2m}n$, then this will contradicts Lemma \ref{f} since $\mid$$T_{j}$$\mid \geq f_{j}tr(T)$ and $\epsilon < log_{\frac{c}{2m}}(1-f_{j})$. Then $\forall i \in \lbrace 1,...,m \rbrace$, $\mid$$A_{i}$$\mid < \frac{\mid A \mid}{2m}$. Let $A^{*} = A\backslash (\bigcup_{i=1}^{m}A_{i})$, then $\mid$$A^{*}$$\mid > cn-m.\frac{cn}{2m} \geq \frac{c}{2}n$. Then $A^{*} \neq \phi$. So there exist vertices $s_{1},...,s_{m},a$ such that $a\in A^{*}$, $s_{i}\in T_{i}$ for $i=1,...,m$, and $\lbrace a \rbrace$ is complete to $\lbrace s_{1},...,s_{m} \rbrace$. $\hfill {\square}$ 
\end{proof}
\vspace{4mm}\\
We omit the proof of the following two Lemmas since they are completely analogous to the proof of Lemma \ref{q}.
\begin{lemma} \label{w} 
Let $f,c,\epsilon > 0$ be constants, where $0 < f, c < 1$ and $0 < \epsilon < log_{c}(1-\frac{f}{4})$. Let $T$ be an $ \epsilon $$-$critical tournament with $\mid$$T$$\mid$ $ =n$, and let $A,B \subseteq V(T)$ be two disjoint subsets with $ \mid $$A$$ \mid$ $ \geq cn$ and $ \mid $$B$$ \mid$ $ \geq cn $. Let $S$ be a transitive subtournament of $T$ with $ \mid $$S$$ \mid$ $\geq f.tr(T)$ and $V(S) \subseteq V(T)\backslash (A\cup B)$. Then there exist vertices $a,b,s$ such that $a\in A, b\in B, s\in S$, and $a \leftarrow s \leftarrow b$.
\end{lemma}
   \begin{lemma} \label{x} 
Let $f_{1},f_{2},c,\epsilon > 0$ be constants, where $0 <  f_{1},f_{2},c < 1$ and $0 < \epsilon < min\lbrace log_{\frac{c}{4}}(1-f_{1}), log_{\frac{c}{4}}(1-f_{2})\rbrace$. Let $T$ be an $ \epsilon $$-$critical tournament with $\mid$$T$$\mid$ $ =n$, and   let $S_{1},S_{2}$ be two disjoint transitive subtournaments of $T$ with $ \mid $$S_{1}$$ \mid$ $\geq f_{1}.tr(T)$ and $ \mid $$S_{2}$$ \mid$ $\geq f_{2}.tr(T)$. Let $A \subseteq V(T) \backslash (V(S_{1})\cup V(S_{2}))$ with $ \mid $$A$$ \mid$ $ \geq cn$. Then there exist vertices $s_{1},a,s_{2}$ such that $a\in A, s_{1}\in S_{1}, s_{2}\in S_{2}$, and $s_{1} \leftarrow a \leftarrow s_{2}$.
\end{lemma}

Let $s$ be a $\lbrace 0,1 \rbrace$$-$vector. Denote by $s_{c}$ the vector obtained from $s$ by replacing every subsequence of consecutive $1'$s by single $1$.

A tournament $T$ is a \textit{regular super galaxy under an ordering} $\alpha$ of its vertices if $T$ is a super galaxy under $\alpha$, such that every connected component of $ B(T,\alpha) $ is a star. Let $G$ be a regular super galaxy under an ordering $\alpha =(u_{1},...,u_{g})$ of its vertices with $\mid$$G$$\mid$ $= g$.  Let $Q_{1},...,Q_{l}$ be the stars of $G$ under $\alpha$. Let $L_i$ be the set of leaves of $Q_i$, and let $W_i=G$$\mid$$L_i$ for $i=1,...,l$. Let $\widehat{G}$ be the graph with vertex set $V(G)$ and arc set $E(G)-(\displaystyle{\bigcup_{i=1}^{l}E(W_i)})$. We call $\widehat{G}$ the \textit{mutant super galaxy corresponding to $G$ under} $\alpha$ and $\alpha$ is called the \textit{forest ordering of} $\widehat{G}$. For $i \in \lbrace 0,...,l \rbrace$ define $G^{i} = G$$\mid$$\bigcup_{j=1}^{i} V(Q_{j})$ where $G^{l} = G$, and $G^{0}$ is the empty tournament.\\
A tournament $\mathcal{K}$ is a \textit{super path-galaxy under an ordering} $ \theta =(v_1,...,v_h) $ of its vertices if every  connected component of $ B(\mathcal{K},\theta) $ is either a star or a $4$-path, such that only one of the connected components  is a $4$-path (if exists), and there exists no center of a star appearing in the ordering $\theta$ between the ends of the $4$-path. Let $Q_1,...,Q_l$ be the stars of $\mathcal{K}$ under $\theta$, and let $P :=\{v_{i_1},v_{i_2},v_{i_3},v_{i_4}\}$ be the $4$-path of $\mathcal{K}$ under $\theta$, such that $P= v_{i_{j_1}},...,v_{i_{j_4}}$. Let $G:=\mathcal{K}$$\mid$$V(\bigcup_{i=1}^lQ_i)$ and let $\alpha =(u_1,...,u_g)=(v_{i_1},...,v_{i_g})$ be the restriction of $\theta$ to $V(G)$. Let $\widehat{\mathcal{K}}$ be the digraph obtained from $\mathcal{K}$ by replacing $E(G)$ by $E(\widehat{G})$, deleting the arc connecting the vertices $v_{i_{j_1}}$ and $v_{i_{j_3}}$, and deleting the arc connecting the vertices $v_{i_{j_2}}$ and $v_{i_{j_4}}$. We call call $\theta$ the \textit{forest ordering of} $\widehat{\mathcal{K}}$. Note that every regular super galaxy is a super path-galaxy. Also notice that in case $P$ does not exist, then $G=\mathcal{K}$ and $\widehat{G}=\widehat{\mathcal{K}}$. 
 Let $s^{\mathcal{K},\theta}$ be a $\lbrace 0,1 \rbrace$$-$vector such that $s^{\mathcal{K},\theta}(i) = 1$ if and only if $v_{i}$ is a leaf of one of the stars of $\mathcal{K}$ under $\theta$, or an end of $P$. Let $w = s_{c}^{\mathcal{K},\theta}$ and let $i_{r}$ be such that $w_{i_{r}}=1$. Let $j$ be such that $s^{\mathcal{K},\theta}_{j}=1$. We say that $s^{\mathcal{K},\theta}_{j}$ \textit{corresponds to} $w_{i_{r}}$ if $s^{\mathcal{K},\theta}_{j}$ belongs to the subsequence of consecutive $1'$s that is replaced by the entry $w_{i_{r}}$. For $k \in \lbrace 1,...,l \rbrace$, let $\alpha_{k} = (u_{k_{1}},...,u_{k_{t_{k}}})$ with $t_{k}=$ $\mid$$ G^{k}$$ \mid$, be the restriction of $\alpha$ to $V(G^{k})$.  Let $s^{\mathcal{K},\theta}_{G^{k}}$ be the restriction of $s^{\mathcal{K},\theta}$ to the $0's$ and $1's$ corresponding to $V(G^{k})$ (notice that $s^{\mathcal{K},\theta}_{G^{k}}= s^{G^{k},\alpha_{k}}$), and let $^{c}s^{\mathcal{K},\theta}_{G^{k}}$ be the vector obtained from $s^{\mathcal{K},\theta}_{G^{k}}$ by replacing every subsequence of consecutive $1's$ corresponding to the same entry of $s^{\mathcal{K},\theta}_{c}$ by single $1$. We say that a smooth $(c,\lambda ,w)$$-$structure of a tournament $T$ \textit{corresponds}  \textit{to $G^{k}$ under $(\mathcal{K},\theta)$} if $w =$ $ ^{c}s^{\mathcal{K},\theta}_{G^{k}}$. Notice that $s^{\mathcal{K},\theta}_{G^{l}}=s^{G,\alpha}$ and $^{c}s^{\mathcal{K},\theta}_{G^{l}}=s^{G,\alpha}_{c}$.\\
Let $\nu =$ $^{c}s^{\mathcal{K},\theta}_{G^{k}}$. Let $\delta^{\nu}:$ $\lbrace j: \nu_{j} = 1 \rbrace \rightarrow \mathbb{N}$ be a function that assigns to every nonzero entry of $\nu$ the number of consecutive $1'$s of $s^{\mathcal{K},\theta}_{G^{k}}$ replaced by that entry of $\nu$.
Fix $k \in \lbrace 0,...,l \rbrace$. Let $(A_{1},...,A_{\mid w \mid})$ be a smooth $(c,\lambda ,w)$$-$structure corresponding to $G^{k}$ under $(\mathcal{K},\theta)$.
Let $r$ be such that $w(r) = 1$. Assume that $A_{r} = \lbrace s^{1}_{r},...,s_{r}^{\mid A_{r} \mid} \rbrace$ and $(s^{1}_{r},...,s_{r}^{\mid A_{r} \mid})$ is a transitive ordering. Write $m(r) = \lfloor\frac{\mid A_{r} \mid}{\delta^{w}(r)}\rfloor$. Denote $A^{j}_{r} = \lbrace s^{(j-1)m(r)+1}_{r},...,s_{r}^{jm(r)} \rbrace$ for $j \in \lbrace 1,...,\delta^{w}(r) \rbrace$. For every $v \in A^{j}_{r}$ denote $\xi(v) = (\mid$$\lbrace t < r: w(t) = 0 \rbrace$$\mid$ $+$ $\displaystyle{\sum_{t < r: w(t) = 1}\delta^{w}(t) })$ $+$ $j$. In other words, $\xi(v)$ is the index $i$ of the vertex $u_{k_i}$ associated with the set $A^j_{r}$. For every $v \in A_{r}$ such that $w(r) = 0$ denote $\xi(v) = (\mid$$\lbrace t < r: w(t) = 0 \rbrace$$\mid$ $+$ $\displaystyle{\sum_{t < r: w(t) = 1}\delta^{w}(t) })$ $+$ $1$. In other words, $\xi(v)$ is the index $i$ of the vertex $u_{k_i}$ associated with the set $A_{r}$. We say that $\widehat{G}^{k}$ is \textit{well-contained in} $(A_{1},...,A_{\mid w \mid})$ that corresponds to $\mathcal{K}^{k}$ under $(\mathcal{K},\theta)$ if there is a homomorphism $f$ of $\widehat{\mathcal{K}}^{k}$ into $T$$\mid$$\bigcup_{i = 1}^{\mid w \mid}A_{i}$ such that $\xi(f(u_{k_j})) = j$ for every $j \in \lbrace 1,...,t_{k} \rbrace$, where $\theta_{k} = (u_{k_{1}},...,u_{k_{t_{k}}})$.
\begin{lemma} \label{supergalaxy}
Let $\mathcal{K}$ be a super path-galaxy under an ordering $\theta$ of its vertices with $\mid$$\mathcal{K}$$\mid$ $= h$, and let $Q_1,...,Q_l$ be the stars of $\mathcal{K}$ under $\theta$. Let $G:=\mathcal{K}$$\mid$$V(\bigcup_{i=1}^lQ_i)$ and let $\alpha =(u_1,...,u_g)$ be the restriction of $\theta$ to $V(G)$.  Let $0 < \lambda < \frac{1}{(2g)^{g+2}}$, $c > 0$ be constants, and $w$ be a $\lbrace 0,1 \rbrace$$-$vector. Fix $k \in \lbrace 0,...,l \rbrace$ and let $\widehat{\lambda} = (2g)^{l-k}\lambda$ and $\widehat{c} = \frac{c}{(2g)^{l-k}}$. There exist $ \epsilon_{k} > 0$ such that $\forall 0 < \epsilon < \epsilon_{k}$, for every $\epsilon$$-$critical tournament $T$ with $\mid$$T$$\mid$ $= n$ containing $\chi = (A_{1},...,A_{\mid w \mid})$ as a smooth $(\widehat{c},\widehat{\lambda},w)$$-$structure corresponding to $G^{k}$ under $(\mathcal{K},\theta)$, we have $\widehat{G}^{k}$ is well-contained in $\chi$.  
\end{lemma}
\begin{proof}
The proof is by induction on $k$. For $k=0$ the statement is obvious since $\widehat{G}^{0}$ is the empty digraph. Let $\alpha_{k}= (h_{1},...,h_{\mid G^{k} \mid})$. Suppose that $\chi = (A_{1},...,A_{\mid w \mid})$ is a smooth $(\widehat{c},\widehat{\lambda},w)$$-$structure in $T$ corresponding to $G^{k}$ under $(\mathcal{K},\theta)$.  Let $h_{p_{0}}$ be the center of $Q_{k}$ and $h_{p_{1}},...,h_{p_{q}}$ be its leaves for some integer $q>0$. 
$\forall 0 \leq i \leq q$, let $R_{i} = \lbrace v \in \displaystyle{\bigcup_{j=1}^{\mid w \mid}A_{j}};$ $ \xi(v) = p_{i} \rbrace$.
 Since we can assume that $\epsilon < log_{\frac{\widehat{c}}{2g}}(1-\frac{\widehat{c}}{g})$, then by Lemma \ref{q} there exist vertices $r_{0},r_{1},...,r_{q}$ such that $r_{i} \in R_{i}$ for $i=0,1,...,q$ and \\
$\ast$ $r_{1},...,r_{q}$ are all adjacent from $r_{0}$ if $z>f$.\\
$\ast$ $r_{1},...,r_{q}$ are all adjacent to $r_{0}$ if $z<f$.\\
So $T$$\mid$$\lbrace r_{0},r_{1},...,r_{q} \rbrace$ contains a copy of $\widehat{G}^{k}$$\mid$$V(Q_{k})$. Denote this copy by $Y$. Since $r_0$ is the center of $Q_k$ and $r_1,...,r_q$ are its leaves, then there exists $ z \in \lbrace 1,...,\mid$$w$$\mid \rbrace$ with $w(z)=0$, such that $R_{0} = A_{z}$, and there exist $ x_1,...,x_q \in \lbrace 1,...,\mid$$w$$\mid \rbrace$, with $w(x_i)=1$  and $R_i\subseteq A_{x_i}$ for $i=1,...,q$ (note here that we don't necessarily have $x_1,...,x_q$ are distinct).   
For all $ i \in \lbrace 1,...,\mid$$w$$\mid \rbrace - \lbrace z,x_1,...,x_q \rbrace$, let $A_{i}^{*} = \displaystyle{\bigcap_{p\in V(Y)}A_{i,p}}$.  Then by Lemma \ref{u}, $\mid$$A_{i}^{*}$$\mid$ $\geq (1-\mid$$Y$$\mid\widehat{\lambda})\mid$$A_{i}$$\mid$ $\geq (1-g\widehat{\lambda})\mid$$A_{i}$$\mid$ $\geq \frac{1}{2g}\mid$$A_{i}$$\mid$ since $\widehat{\lambda} \leq \frac{2g-1}{2g^{2}}$.
Write $\mathcal{H} = \lbrace 1,...,\mid $$G^{k}$$ \mid \rbrace - \lbrace p_{0},...,p_{q} \rbrace$. If $\lbrace v\in A_{x_i}: \xi(v) \in \mathcal{H} \rbrace \neq \phi$, then define $J_{x_i} = \lbrace \eta \in \mathcal{H}: \exists v \in A_{x_i}$ and $\xi(v)= \eta \rbrace$. Now for all $ \eta \in J_{x_i}$, let $A_{x_i}^{*\eta}= \lbrace v \in A_{x_i}: \xi(v)=\eta$ and $v \in \displaystyle{\bigcap_{p\in V(Y)}A_{x_i,p}} \rbrace$. Then by Lemma \ref{u}, for all $ \eta \in J_{f}$, we have $\mid$$A_{x_i}^{*\eta}$$\mid$ $ \geq \frac{1-g^2\widehat{\lambda}}{g}\mid $$A_{x_i}$$\mid $ $\geq \frac{\mid A_{x_i}\mid}{2g}$ since $\widehat{\lambda} \leq \frac{1}{2g^2}$. Now for all $ \eta \in J_{x_i}$, select arbitrary $\lceil \frac{\mid A_{x_i}\mid}{2g}\rceil$ vertices of $A_{x_i}^{*\eta}$ and denote the union of these $\mid$$J_{x_i}$$\mid$ sets by $A_{x_i}^{*}$. 
So we have defined $t$ sets $A_{1}^{*},...,A^{*}_{t}$. We have $\mid$$A_{i}^{*}$$\mid$ $\geq \frac{\mid A_{x_i}\mid}{2g}$ for $i=1,...,t$. Now Lemma \ref{s} implies that $\chi^{*}=(A_{1}^{*},...,A^{*}_{t})$ form a smooth $(\frac{\widehat{c}}{2g},2g\widehat{\lambda},w^{*})$$-$structure of $T$ corresponding  to $G^{k-1}$ under $(\mathcal{K},\theta)$, where $\frac{\widehat{c}}{2g}= \frac{c}{(2g)^{l-(k-1)}}, 2g\widehat{\lambda}=(2g)^{l-(k-1)}\lambda$, and $w^{*}$ is an appropriate $\lbrace 0,1 \rbrace$$-$vector.
Now take $\epsilon_{k} < min \lbrace \epsilon_{k-1}, log_{\frac{\widehat{c}}{2g}}(1-\frac{\widehat{c}}{g}) \rbrace$. So by the  induction hypothesis $\widehat{G}^{k-1}$ is well-contained in $\chi^{*}$. Now by merging the well-contained copy of $\widehat{G}^{k-1}$ and $Y$ we get a copy of $\widehat{G}^{k}$. $\hfill{ \square}$\vspace{3mm}\\ 
\end{proof}
Let $\mathcal{K}$ be a super path-galaxy under $ \theta =(v_1,...,v_h) $. We say that a smooth $(c,\lambda ,w)$$-$structure of a tournament $T$ \textit{corresponds}  \textit{to $\mathcal{K}$ under $\theta$} if $w =$ $s^{\mathcal{K},\theta}_{c}$. \\
Let $\delta^{s^{\mathcal{K},\theta}_{c}}:$ $\lbrace j: s^{\mathcal{K},\theta}_{cj} = 1 \rbrace \rightarrow \mathbb{N}$ be a function that assigns to every nonzero entry of $s^{\mathcal{K},\theta}_{c}$ the number of consecutive $1'$s of $s^{\mathcal{K},\theta}$ replaced by that entry of $s^{\mathcal{K},\theta}_{c}$. Let $(A_{1},...,A_{\mid w \mid})$ be a smooth $(c,\lambda ,w)$$-$structure corresponding to $\mathcal{K}$ under $\theta$.
   We say that $\widehat{\mathcal{K}}$ is \textit{well-contained in} $(A_{1},...,A_{\mid w \mid})$ that corresponds to $\mathcal{K}$ under $\theta$, if there is a homomorphism $f$ of $\widehat{\mathcal{K}}$ into $T$$\mid$$\bigcup_{i = 1}^{\mid w \mid}A_{i}$ such that $\xi(f(v_{j})) = j$ for every $j \in \lbrace 1,...,h \rbrace$.
\begin{lemma} \label{path-galaxy}
Let $\mathcal{K}$ be a super path-galaxy under an ordering $\theta$ of its vertices with $\mid$$\mathcal{K}$$\mid$ $= h$. Let $0 < \lambda < \frac{1}{(2g)^{g+2}}$, $c > 0$ be constants, with $g=h-4$. Let $\tilde{c}=3c$, $\tilde{\lambda}=\frac{\lambda}{3}$, $w$ a $\{0,1\}$-vector, and let $\epsilon >0$ be small enough. Let $\chi = (A_{1},...,A_{\mid w \mid})$ be a smooth $(\tilde{c},\tilde{\lambda},w)$$-$structure of an $n$-vertex $\epsilon$$-$critical tournament $T$, corresponding to $\mathcal{K}$ under $\theta$, then $\widehat{\mathcal{K}}$ is well-contained in $\chi$.  
\end{lemma}
\begin{proof}
Let $Q_{1},...,Q_{l}$ be the stars of $\mathcal{K}$ under $\theta =(v_1,...,v_h)$, let $G:=\mathcal{K}$$\mid$$V(\bigcup_{i=1}^lQ_i)$, and let $\theta_G := (v_{i_1},...,v_{i_g})=(u_1,...,u_g)$ be the restriction of $\theta$ to $V(G)$.  For all $i\in \{ i_1,...,i_g\}$, let  $M_{i} := \lbrace v \in \bigcup_{j=1}^{\mid w \mid}A_{j};$ $ \xi(v) = i \rbrace$. For all $ 1\leq j \leq $ $\mid$$w$$\mid$, let $R_{j}:=  \displaystyle{\bigcup_{M_{i}\subseteq A_{j}}M_{i}}$ (notice that we may have: $R_{j}= \phi$). Denote by $N_{1},...,N_{\mid w^{*}\mid}$ the non-empty sets $R_{j}$, where $w^*$ is a $\{0,1\}$-vector. Then $\chi^{*}=(N_{1},...,N_{\mid w^{*}\mid})$ is a smooth $(c,\lambda,w^{*})$$-$structure of $T$ corresponding to $G$ under $(u_1,...,u_g)$. Let $\epsilon_{l}$ be the $\epsilon$ from Lemma \ref{supergalaxy}. Taking $\epsilon < \epsilon_{l}$ and since $\lambda < \frac{1}{(2g)^{g+2}}$, then we can use Lemma \ref{supergalaxy} and conclude taking $k=l$ that $G$ is well-contained in $\chi^{*}$. Denote this well-contained copy of $G$ by $X$. Let $\{v_{j_1},...,v_{j_4}\}$ be the vertex set of the $4$-path in $B(\mathcal{K},\theta)$, and let $v_{j_1}$ and $v_{j_2}$ be its ends. For all $i\in \{ j_1,..., j_4\}$, let $M_{i} := \lbrace v \in \displaystyle{\bigcup_{j=1}^{\mid w \mid}S_{j}};$ $ \xi(v) = i \rbrace$. Let $M_i^*:= \displaystyle{\bigcap_{p\in V(X)}M_{i,p}}$. Then by Lemma \ref{u}, $\mid$$M_{i}^{*}$$\mid$ $\geq (1-g\tilde{\lambda})\mid$$M_{i}$$\mid$ $\geq \frac{1}{2}\mid$$M_{i}$$\mid$ $\geq \frac{\tilde{c}}{2}n$ for $i=i_3,i_4$, since $\tilde{\lambda} \leq \frac{1}{2g}$, and $\mid$$M_{i}^{*}$$\mid$ $\geq \frac{1-g(g+1)\tilde{\lambda}}{h}\tilde{c}tr(T)\geq  \frac{\tilde{c}}{2h}tr(T)$ for $i=i_1,i_2$, since $\tilde{\lambda} \leq \frac{1}{2g(g+1)}$. Now since we can assume that $\epsilon < min\lbrace log_{\frac{\tilde{c}}{8}}(\frac{1}{2}), log_{\frac{\tilde{c}}{4}}(1-\frac{\tilde{c}}{2h})\rbrace$, then by Lemma \ref{g}, there exist vertices $m_i\in M_i^*$ for $i=i_1,i_2,i_3,i_4$, such that $T$$\mid$$\{m_{i_1},...,m_{i_4}\}$ contains a copy of $\widehat{\mathcal{K}}$$\mid$$\{v_{j_1},...,v_{j_4}\}$. Denote this copy by $Y$. Now by merging $X$ and $Y$, we get a well-contained copy of $\widehat{\mathcal{K}}$ in $\chi$. This completes the proof. $\hfill{ \square}$\vspace{3.5mm}
\end{proof}

 In what follows we define  crucial super path-galaxies $\mathcal{K}_i$ under $\theta_{\mathcal{K}_i}$ for $i=1,...,7$, that will be useful in our latter analysis.  To this end, we will define the connected components of $B(\mathcal{K}_i,\theta_{\mathcal{K}_i})$ for $i=1,...,7$ (see Figure \ref{fig:supergalaxies}):
\begin{itemize}[-]
\item The connected components of $B(\mathcal{K}_1,\theta_{\mathcal{K}_1})$, with $\theta_{\mathcal{K}_1}:=(v_1,...,v_{23})$, are: The four right stars $\{v_{12},v_{14},$ $v_{20}\}, \{v_7,v_{18},v_{23}\}, \{v_3,v_{10},v_{19}\}, \{v_1,v_5\}$, the two left stars $\{v_2,v_{13},v_{16},v_{22}\},\{v_6,v_9,v_{15},v_{17}\}$, and the $4$-path $\{v_4,v_8,v_{11},v_{21}\}$,  such that $v_{11}v_{4},v_4v_{21},v_8v_{21}$ are the edges of the $4$-path $\{v_4,v_8,v_{11},v_{21}\}$.
\item The connected components of $B(\mathcal{K}_2,\theta_{\mathcal{K}_2})$, with $\theta_{\mathcal{K}_2}:=(v_1,...,v_{17})$, are: The two right stars $\{v_{8},v_{11}\},$ $ \{v_4,v_{13}\}$, the two left stars $\{v_1,v_{6},v_{9},v_{15},v_{17}\},\{v_2,v_{7},v_{12},v_{16}\}$, and the $4$-path $\{v_3,v_5,v_{10},v_{14}\}$,  such that $v_3v_{10},v_{10}v_{14},v_5v_{14}$ are the edges of the $4$-path $\{v_3,v_5,v_{10},v_{14}\}$.
\item The connected components of $B(\mathcal{K}_3,\theta_{\mathcal{K}_3})$, with $\theta_{\mathcal{K}_3}:=(v_1,...,v_{12})$, are: The two right stars $\{v_{2},v_{4},v_7,$ $v_{10}\}$ and $\{v_{3},v_{5},v_8,v_{12}\}$, and the $4$-path $\{v_1,v_6,v_{9},v_{11}\}$,  such that $v_9v_{1},v_{1}v_{11},v_{11}v_{6}$ are the edges of the $4$-path $\{v_1,v_6,v_{9},v_{11}\}$.
\item The connected components of $B(\mathcal{K}_4,\theta_{\mathcal{K}_4})$, with $\theta_{\mathcal{K}_4}:=(v_1,...,v_{16})$, are: The four right stars $\{v_{7},v_{10},$ $v_{13},v_{16}\}$, $\{v_{1},v_{3},v_9\}$, $\{v_{2},v_{5},v_{11}\}$, $\{v_{12},v_{14}\}$, and the $4$-path $\{v_4,v_6,v_{8},v_{15}\}$,  such that $v_8v_{4},v_{4}v_{15},v_{15}v_{6}$ are the edges of the $4$-path $\{v_4,v_6,v_{8},v_{15}\}$.
\item The connected components of $B(\mathcal{K}_5,\theta_{\mathcal{K}_5})$, with $\theta_{\mathcal{K}_5}:=(v_1,...,v_{19})$, are: The three right stars $\{v_{5},v_{7},$ $v_{16}\}$, $\{v_{6},v_{15}\}$, $\{v_{10},v_{14},v_{19}\}$, the two left stars $\{v_1,v_4,v_{9},v_{13}\}$, $\{v_3,v_{12},v_{18}\}$, and the $4$-path $\{v_2,v_8,v_{11},v_{17}\}$,  such that $v_2v_{11},v_{2}v_{17},v_{17}v_{8}$ are the edges of the $4$-path $\{v_2,v_8,v_{11},v_{17}\}$.
\item The connected components of $B(\mathcal{K}_6,\theta_{\mathcal{K}_6})$, with $\theta_{\mathcal{K}_6}:=(v_1,...,v_{15})$, are: The three right stars $\{v_{8},v_{10},$ $v_{15}\}$, $\{v_{5},v_{11}\}$, $\{v_{4},v_{9},v_{14}\}$, one left stars $\{v_1,v_6,v_{12}\}$, and the $4$-path $\{v_2,v_3,v_{7},v_{13}\}$,  such that $v_2v_{7},v_{7}v_{13},v_{13}v_{3}$ are the edges of the $4$-path $\{v_2,v_3,v_{7},v_{13}\}$.
\item The connected components of $B(\mathcal{K}_7,\theta_{\mathcal{K}_7})$, with $\theta_{\mathcal{K}_7}:=(v_1,...,v_{17})$, are: The two right stars $\{v_{5},v_{9},$ $v_{17}\}$, $\{v_{4},v_{11},v_{14}\}$, the two left stars $\{v_1,v_8,v_{13},v_{15}\}$, $\{v_3,v_7,v_{10}\}$, and the $4$-path $\{v_2,v_6,v_{12},v_{16}\}$,  such that $v_2v_{12},v_{2}v_{16},v_{16}v_{6}$ are the edges of the $4$-path $\{v_2,v_6,v_{12},v_{16}\}$.
\end{itemize}
\begin{figure}[h]
\centering
	\includegraphics[scale=0.8, width=16cm, height=9cm]{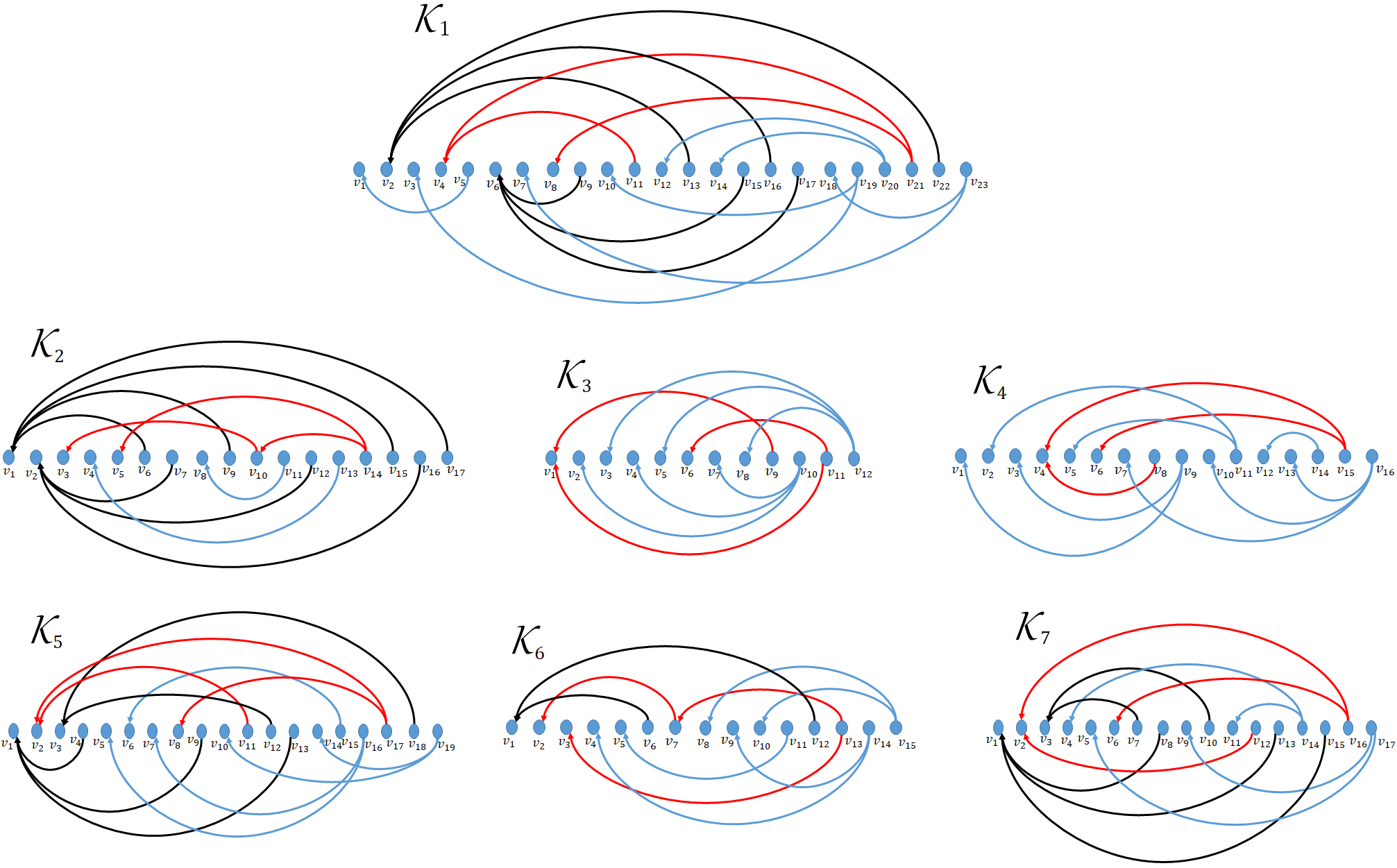}
	\caption{Super path-galaxies $\mathcal{K}_1$,...,$\mathcal{K}_7$. All the non drawn arcs are forward. The arcs corresponding to a $4$-path (resp. right stars) (resp. left stars) are colored in red (resp. blue) (resp. black).}
	\label{fig:supergalaxies}
\end{figure}  
\section{EHC for some classes of tournaments in $\cal S \cup \cal H$}
For each of the following tournaments we will define two special orderings of their vertex set and the set of backward arcs under each ordering (see Figure \ref{fig:orderingsS}). 
\begin{itemize}[-]
\item \textit{Star ordering of} $S_{1}$, $\theta_{S_{1}} := (d,f,v,a,b,e,c)$, $E(\theta_{S_{1}})= \lbrace (v,d),(b,d),(c,d),(e,f),(e,a) \rbrace $.\\ \textit{Cyclic ordering of} $S_{1}$, $\theta_{C_{1}} :=(f,v,b,e,c,d,a)$,  $E(\theta_{C_{1}})= \lbrace (d,e),(d,f),(e,f),(a,b),(a,c) \rbrace $.
\item \textit{Forest ordering of} $ S_{2} $, $\theta_{f_{2}} := (e,f,v,a,b,d,c)$, $E(\theta_{f_{2}})= \lbrace (v,e),(d,f),(d,e),(c,a) \rbrace $.\\ \textit{Cyclic ordering of} $ S_{2} $, $\theta_{C_{2}} :=(f,v,e,a,b,d,c)$,  $E(\theta_{C_{2}})= \lbrace (e,f),(d,f),(d,e),(c,a) \rbrace $.
\item \textit{Forest ordering of} $ S_{3} $, $\theta_{f_{3}} := (f,v,b,c,d,e,a)$, $E(\theta_{f_{3}})= \lbrace (e,f),(d,f),(e,c),(a,b) \rbrace $.\\ \textit{Cyclic ordering of} $ S_{3} $, $\theta_{C_{3}} :=(f,v,b,e,c,d,a)$,  $E(\theta_{C_{3}})= \lbrace (e,f),(d,f),(d,e),(a,b) \rbrace $.
\item \textit{Path ordering of} $ S_{4} $, $\theta_{S_{4}} := (v,d,f,c,a,e,b)$, $E(\theta_{S_{4}})= \lbrace (f,v),(b,c),(b,d),(e,f) \rbrace $.\\ \textit{Cyclic ordering of} $ S_{4} $, $\theta_{C_{4}} :=(v,d,e,f,a,b,c)$,  $E(\theta_{C_{4}})= \lbrace (f,v),(b,d),(a,e),(c,a),(c,e) \rbrace $.
\item \textit{Forest ordering of} $ S_{5} $, $\theta_{f_{5}} := (v,a,c,f,e,d,b)$, $E(\theta_{f_{5}})= \lbrace (b,c),(b,f),(d,a),(f,v) \rbrace $.\\ \textit{Cyclic ordering of} $ S_{5} $, $\theta_{C_{5}} :=(a,f,v,c,e,d,b)$,  $E(\theta_{C_{5}})= \lbrace (b,f),(b,c),(c,f),(v,a),(d,a) \rbrace $.
\item \textit{Forest ordering of} $ S_{9} $, $\theta_{f_{9}} := (v,b,c,a,d,f,e)$, $E(\theta_{f_{9}})= \lbrace (f,v),(a,b),(e,c),(e,b) \rbrace $.\\ \textit{Cyclic ordering of} $ S_{9} $, $\theta_{C_{9}} :=(v,c,a,b,d,f,e)$,  $E(\theta_{C_{9}})= \lbrace (f,v),(b,c),(e,b),(e,c) \rbrace $.
\item \textit{Forest ordering of} $ S_{10} $, $\theta_{f_{10}} := (d,c,v,e,a,f,b)$, $E(\theta_{f_{10}})= \lbrace (a,d),(b,d),(e,c),(f,e),(f,v) \rbrace $.\\ \textit{Cyclic ordering of} $ S_{10} $, $\theta_{C_{10}} :=(d,e,c,v,a,f,b)$,  $E(\theta_{C_{10}})= \lbrace (b,d),(a,d),(v,e),(f,e),(f,v) \rbrace $.
\end{itemize}
\begin{figure}[h]
\centering
	\includegraphics[scale=0.8, width=16cm, height=8.2cm]{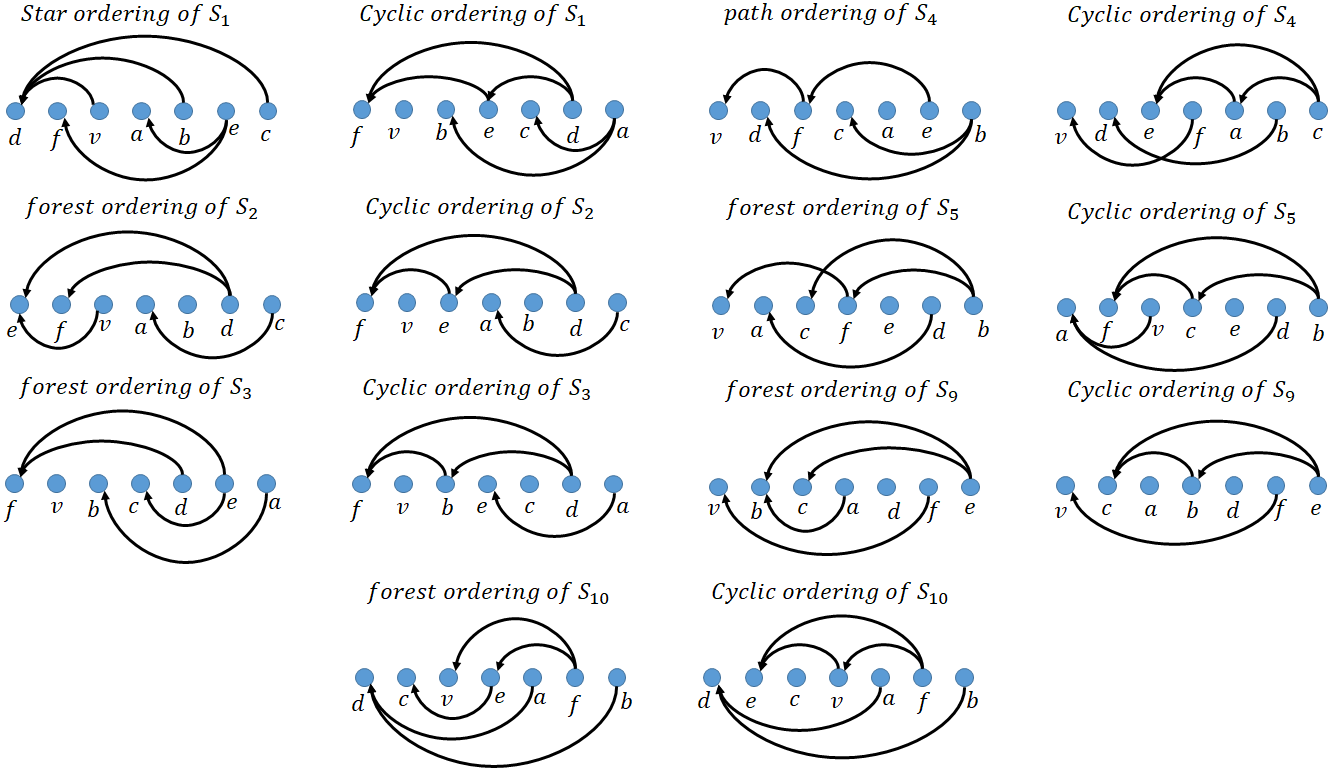}
	\caption{Crucial orderings of the vertices of the tournaments $S_{1},S_{2},S_{3},S_{4},S_{5},S_{9},$ and $S_{10}$. All the non drawn arcs are forward.}
	\label{fig:orderingsS}
\end{figure}
For each of the following tournaments we will define two special orderings of their vertex set $\theta_{f_{i}}$ and $\theta_{C_{i}}$. $\theta_{f_{i}}$ is called the \textit{forest ordering of} $H_{i}$ and $\theta_{C_{i}}$ is called the \textit{cyclic ordering of} $H_{i}$ (see Figure \ref{fig:orderingsH}). 
\begin{itemize}[-]
\item $\theta_{f_{18}} := (v_{7},v_{1},v_{2},v_{3},v_{4},v_{6},v_{5})$, $E(\theta_{f_{18}})= \lbrace (v_{5},v_{2}),(v_{5},v_{1}),(v_{4},v_{1}),(v_{3},v_{7}),(v_{6},v_{3}) \rbrace $.\\ $\theta_{C_{18}} := (v_{7},v_{2},v_{4},v_{1},v_{6},v_{3},v_{5})$, $E(\theta_{C_{18}})= \lbrace (v_{1},v_{2}),(v_{5},v_{2}),(v_{5},v_{1}),(v_{3},v_{7}),(v_{3},v_{4}) \rbrace $.
\item $\theta_{f_{37}} := (v_{1},v_{2},v_{7},v_{3},v_{4},v_{6},v_{5})$, $E(\theta_{f_{37}})= \lbrace (v_{5},v_{2}),(v_{5},v_{1}),(v_{4},v_{1}),(v_{6},v_{7}) \rbrace $.\\ $\theta_{C_{37}} := (v_{2},v_{4},v_{1},v_{3},v_{6},v_{7},v_{5})$, $E(\theta_{C_{37}})= \lbrace (v_{1},v_{2}),(v_{5},v_{2}),(v_{5},v_{1}),(v_{7},v_{3}),(v_{7},v_{4}),(v_{3},v_{4}) \rbrace $.
\item $\theta_{f_{3}} := (v_{3},v_{4},v_{7},v_{5},v_{1},v_{2},v_{6})$, $E(\theta_{f_{3}})= \lbrace (v_{1},v_{3}),(v_{2},v_{4}),(v_{6},v_{5}),(v_{2},v_{3}) \rbrace $.\\ $\theta_{C_{3}} := (v_{7},v_{2},v_{4},v_{1},v_{3},v_{6},v_{5})$, $E(\theta_{C_{3}})= \lbrace (v_{1},v_{2}),(v_{5},v_{2}),(v_{5},v_{1}),(v_{3},v_{7}),(v_{4},v_{7}),(v_{3},v_{4}) \rbrace $.
\item $\theta_{f_{8}} := (v_{1},v_{2},v_{3},v_{4},v_{6},v_{5},v_{7})$, $E(\theta_{f_{8}})= \lbrace (v_{5},v_{2}),(v_{5},v_{1}),(v_{4},v_{1}),(v_{7},v_{6}),(v_{6},v_{3}) \rbrace $.\\ $\theta_{C_{8}} := (v_{2},v_{4},v_{1},v_{6},v_{3},v_{5},v_{7})$, $E(\theta_{C_{8}})= \lbrace (v_{1},v_{2}),(v_{5},v_{2}),(v_{5},v_{1}),(v_{7},v_{6}),(v_{3},v_{4}) \rbrace $.
\item $\theta_{f_{20}} := (v_{7},v_{1},v_{2},v_{3},v_{4},v_{6},v_{5})$, $E(\theta_{f_{20}})= \lbrace (v_{5},v_{2}),(v_{5},v_{1}),(v_{4},v_{1}),(v_{6},v_{7}),(v_{6},v_{3}) \rbrace $.\\ $\theta_{C_{20}} := (v_{7},v_{2},v_{4},v_{1},v_{6},v_{3},v_{5})$, $E(\theta_{C_{20}})= \lbrace (v_{1},v_{2}),(v_{5},v_{2}),(v_{5},v_{1}),(v_{6},v_{7}),(v_{3},v_{4}) \rbrace $.
\item $\theta_{f_{2}} := (v_{3},v_{7},v_{4},v_{5},v_{1},v_{2},v_{6})$, $E(\theta_{f_{2}})= \lbrace (v_{1},v_{3}),(v_{2},v_{4}),(v_{6},v_{5}),(v_{2},v_{3}) \rbrace $.\\ $\theta_{C_{2}} := (v_{7},v_{2},v_{4},v_{1},v_{3},v_{6},v_{5})$, $E(\theta_{C_{2}})= \lbrace (v_{1},v_{2}),(v_{5},v_{2}),(v_{5},v_{1}),(v_{3},v_{7}),(v_{3},v_{4}) \rbrace $.
\item $\theta_{f_{26}} := (v_{1},v_{2},v_{3},v_{7},v_{4},v_{6},v_{5})$, $E(\theta_{f_{26}})= \lbrace (v_{5},v_{2}),(v_{5},v_{1}),(v_{4},v_{1}),(v_{6},v_{7}) \rbrace $.\\ $\theta_{C_{26}} := (v_{2},v_{4},v_{1},v_{3},v_{6},v_{7},v_{5})$, $E(\theta_{C_{26}})= \lbrace (v_{1},v_{2}),(v_{5},v_{2}),(v_{5},v_{1}),(v_{7},v_{4}),(v_{3},v_{4}) \rbrace $.
\item $\theta_{f_{1}} := (v_{1},v_{2},v_{3},v_{4},v_{6},v_{5},v_{7})$, $E(\theta_{f_{1}})= \lbrace (v_{5},v_{2}),(v_{5},v_{1}),(v_{4},v_{1}),(v_{7},v_{6}) \rbrace $.\\ $\theta_{C_{1}} := (v_{2},v_{4},v_{1},v_{3},v_{6},v_{5},v_{7})$, $E(\theta_{C_{1}})= \lbrace (v_{1},v_{2}),(v_{5},v_{2}),(v_{5},v_{1}),(v_{7},v_{6}),(v_{3},v_{4}) \rbrace $.
\item $\theta_{f_{25}} := (v_{7},v_{1},v_{2},v_{3},v_{4},v_{6},v_{5})$, $E(\theta_{f_{25}})= \lbrace (v_{5},v_{2}),(v_{5},v_{1}),(v_{4},v_{1}),(v_{6},v_{7}) \rbrace $.\\ $\theta_{C_{25}} := (v_{7},v_{2},v_{4},v_{1},v_{3},v_{6},v_{5})$, $E(\theta_{C_{25}})= \lbrace (v_{1},v_{2}),(v_{5},v_{2}),(v_{5},v_{1}),(v_{6},v_{7}),(v_{3},v_{4}) \rbrace $.
\item $\theta_{f_{31}} := (v_{1},v_{2},v_{3},v_{4},v_{7},v_{6},v_{5})$, $E(\theta_{f_{31}})= \lbrace (v_{5},v_{2}),(v_{5},v_{1}),(v_{4},v_{1}),(v_{6},v_{3}) \rbrace $.\\ $\theta_{C_{31}} := (v_{2},v_{4},v_{1},v_{6},v_{3},v_{7}, v_{5})$, $E(\theta_{C_{31}})= \lbrace (v_{1},v_{2}),(v_{5},v_{2}),(v_{5},v_{1}),(v_{7},v_{6}),(v_{3},v_{4}) \rbrace $.
\item $\theta_{f_{12}} := (v_{1},v_{2},v_{3},v_{4},v_{6},v_{7},v_{5})$, $E(\theta_{f_{12}})= \lbrace (v_{5},v_{2}),(v_{5},v_{1}),(v_{4},v_{1}),(v_{6},v_{3}) \rbrace $.\\ $\theta_{C_{12}} := (v_{2},v_{4},v_{1},v_{6},v_{3},v_{7}, v_{5})$, $E(\theta_{C_{12}})= \lbrace (v_{1},v_{2}),(v_{5},v_{2}),(v_{5},v_{1}),(v_{3},v_{4}) \rbrace $.
\end{itemize}
\begin{figure}[h]
\centering
	\includegraphics[scale=0.8, width=16cm, height=10.2cm]{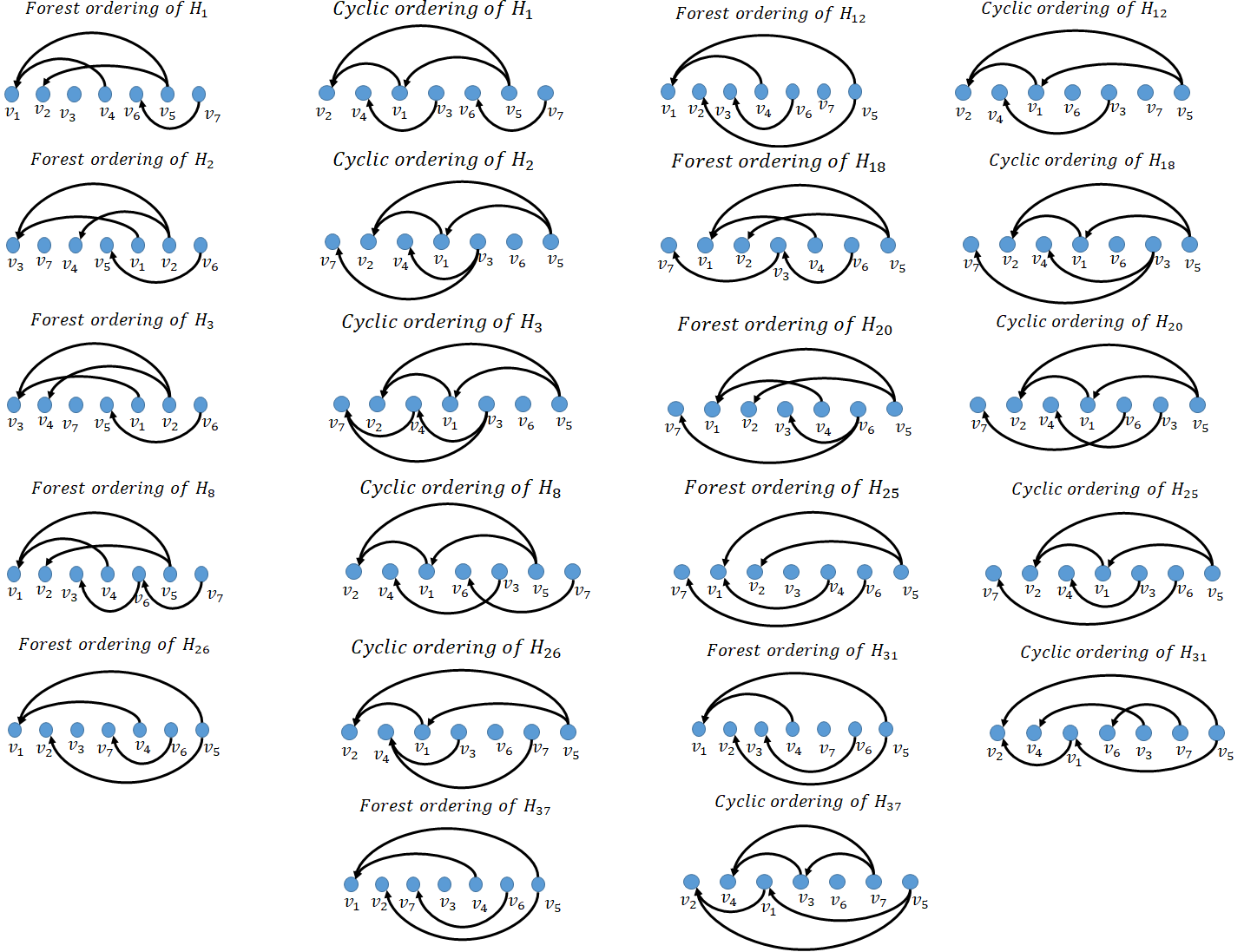}
	\caption{Crucial orderings of the vertices of the tournaments $H_{1},H_{2},H_{3},H_{8},H_{12},H_{18},H_{20},H_{25},$ $H_{26},H_{31},$ and $H_{37}$. All the non drawn arcs are forward.}
	\label{fig:orderingsH}
\end{figure}
We will define four classes of tournaments in $\mathcal{S}\cup\mathcal{H}$. \\
- $\mathcal{C}_1:=\{S_1,...,S_4,S_9,H_1,H_{12},H_{20},H_{25},H_{26},H_{31}\}$.\\
- $\mathcal{C}_2:=\{S_5,S_{10}\}$.\\
- $\mathcal{C}_3:=\{H_3,H_{37}\}$.\\
- $\mathcal{C}_4:=\{H_2,H_8,H_{18}\}$. 
\begin{lemma}\label{C1} 
 Let $\tilde{c}>0$, $0 < \tilde{\lambda} < \frac{1}{3(38)^{21}}$ be constants, let $\epsilon >0$ be small enough, and $w=(1,0,1,0,0,0,0,1,0,$ $0,0,0,1,0)$. Let $\overline{c}=\frac{13\tilde{c}}{12},\overline{\lambda}=\frac{12\tilde{\lambda}}{13}$. Let $(W_{1},...,W_{14})$ be a smooth $(\overline{c},\overline{\lambda} ,w)$$-$structure of an $n$-vertex $\epsilon$$-$critical tournament $T$, such that $\mid$$W_8$$\mid$ is divisible by $13$. Then for all $H\in \mathcal{C}_1$, $T$ contains $H$. \end{lemma}
 \begin{proof}
Let $\hat{W}_8 =W_8-W_8^6$, and let $\chi =(W_1,...,W_4,W_6,W_7,\hat{W}_8,W_9,W_{10},W_{12},W_{13},W_{14})=(A_{1},...,A_{12})$. Clearly $\chi$ is a smooth $(\tilde{c},\tilde{\lambda},\tilde{w})$-structure, where $\tilde{w}=(1,0,1,0,0,0,1,0,$ $0,0,1,0)$. Then by Lemma \ref{path-galaxy}, $\widehat{\mathcal{K}}_1$ is well-contained in $\chi$. So there exist vertices $x_i\in W_i$ for $i=1,...,4,6,7,9,10,12,13,14$, and there exist vertices $x_8^i\in W_8^i$ for $i=1,...,5,7,...,13$, such that $T$$\mid$$\{x_1,...,x_4,x_6,x_7,x_8^1,...,x_8^5,x_8^7,...,x_8^{13},x_9,x_{10},x_{12},x_{13},x_{14}\}$ contains $\widehat{\mathcal{K}}_1$ and $(x_1,...,x_4,x_6,x_7,x_8^1,...,x_8^5,x_8^7,...,x_8^{13},x_9,x_{10},x_{12},x_{13},x_{14})$ is the forest ordering of $\widehat{\mathcal{K}}_1$. Denote this copy of $\widehat{\mathcal{K}}_1$ by $X$. Let $W_{i}^{*} = \displaystyle{\bigcap_{p\in V(X)}W_{i,p}}$ for $i=5,11$, and let $W_{8}^{6*} = \displaystyle{\bigcap_{p\in V(\tilde{X})}W^6_{8,p}}$, where $\tilde{X}=\{x_1,...,x_4,x_6,x_7,x_9,x_{10},$ $x_{12},x_{13},x_{14}\}$.  Then by Lemma \ref{u}, $\mid$$W_{8}^{6*}$$\mid$ $\geq \frac{1-143\tilde{\lambda}}{13}\overline{c}tr(T)\geq  \frac{\overline{c}}{26}tr(T)$ since $\overline{\lambda} \leq \frac{1}{286}$, and $\mid$$W_{i}^{*}$$\mid$ $\geq (1-23\overline{\lambda})\mid$$W_{i}$$\mid$ $\geq \frac{1}{2}\mid$$W_{i}$$\mid$ $\geq \frac{\overline{c}}{2}n$ for $i=5,11$, since $\overline{\lambda} \leq \frac{1}{46}$. Now since we can assume that $\epsilon < log_{\frac{\overline{c}}{2}}(1-\frac{\overline{c}}{104})$, then by Lemma \ref{w}, there exist vertices $x_5\in W_5^*$, $x_{11}\in W_{11}^*$, and $x_8^6\in W_8^{6*}$, such that $x_5\leftarrow x_8^6\leftarrow x_{11}$.

Consider first the tournament $T$$\mid$$\{x_2,x_6,x_8^1,x_8^7,x_8^8,x_8^9,x_8^{11},x_8^{13},x_{10},x_{13},x_{14}\}$, denoted by $F_1$. We are going to prove that $F_1$ contains $S_1$. If $x_8^8\leftarrow x_{13}$ or $x_8^{11}\leftarrow x_{13}$, then $(x_2,x_6,x_8^1,x,x_8^{13},x_{13},x_{14})$ is the cyclic ordering of $S_1$, where $x\in\{ x_8^8,x_8^{11}\}$. And if $x_8^8\rightarrow x_{13}$ or $x_8^{11}\rightarrow x_{13}$, then $(x_2,x_8^7,x_8^8,x_8^9,x_8^{11},x_{10},x_{13})$ is the star ordering of $S_1$. Now consider the tournament $T$$\mid$$\{x_2,x_4,x_7,x_8^1,...,x_8^5,x_8^9,...,x_8^{13},x_9,x_{10},x_{12},x_{14}\}$, denoted by $F_2$. We will prove that for all $H\in \{S_2,S_3,S_9,H_1,H_{26},H_{12},H_{25},H_{31}\}$, $F_2$ contains $H$. Assume first that $x_4\rightarrow x_8^2$ and $x_8^5\rightarrow x_{12}$. Then  $(x_4,x_8^2,x_8^5,x_8^{13},x_9,x_{12},x_{14})$, $(x_2,x_4,x_8^2,x_8^5,x_8^9,x_8^{11},x_{12})$, $(x_4,x_7,x_8^1,x_8^2,x_8^5,x_{12},x_{14})$, $(x_4,x_8^2,x_8^4,x_8^{5},x_9,x_{10},x_{12})$, $(x_4,x_8^2,x_8^3,x_8^5,x_8^{13},x_{12},x_{14})$, $(x_2,x_4,x_8^2,x_8^3,x_8^5,x_8^{11},$ $x_{12})$, $(x_4,x_8^2,...,x_8^5,x_{9},x_{12})$, $(x_4,x_8^2,x_8^4,x_8^5,x_8^{9},x_{9},x_{12})$ is the forest ordering of $S_2$, $S_9$, $S_3$, $H_{12}$, $H_1$, $H_{25}$, $H_{26}$, $H_{31}$ respectively. Now assume that $x_4\leftarrow x_8^2$ or $x_8^5\leftarrow x_{12}$. Then $(x_4,x_7,x,x_8^{13},x_9,x_{12},x_{14})$, $(x_4,x_7,x_8^1,x,x_9,$ $x_{12},x_{14})$, $(x_2,x_4,x_7,x,x_8^9,x_8^{11},x_{12})$, $(x_4,x_7,x,x_8^{10},x_8^{13},x_{12},x_{14})$, $(x_4,x_7,x,x_8^{9},x_8^{10},x_{9},x_{12})$, $(x_2,x_4,x_7,x,x_8^{10}$ $,x_8^{11},x_{12})$, $(x_4,x_7,x,x_8^{10},x_8^{11},x_8^{12},x_{12})$  $(x_4,x_7,x,x_8^{9},x_8^{10},x_{10},x_{12})$ is the cyclic ordering of $S_2$, $S_3$, $S_9$, $H_{1}$, $H_{12}$, $H_{25}$, $H_{26}$, $H_{31}$ respectively, where $x\in\{x_8^2,x_8^5\}$. So, for all $H\in \{S_2,S_3,S_9,H_1,H_{26},H_{12},H_{25},H_{31}\}$, $F_2$ contains $H$. To complete the proof, we will prove that $T$ contains $S_4$ and $H_{20}$. Denote by $F_3$ the tournament $T$$\mid$$\{x_1,x_2,x_5,x_6,x_8^1,x_8^6,x_8^{11},x_8^{13},x_{10},x_{11},x_{14}\}$. We have two cases. Either $x_5\rightarrow x_{11}$, or $x_5\leftarrow x_{11}$. If the former holds, $(x_5,x_8^1,x_8^6,x_8^{13},x_{10},x_{11},x_{14})$ is the path ordering of $S_4$. And if the latter holds, $(x_1,x_2,x_5,x_6,x_8^6,x_8^{11},x_{11})$ is the cyclic ordering of $S_4$. So $F_3$ contains $S_4$. Finally let $F_4:= T$$\mid$$\{x_2,x_3,x_4,x_7,x_8^2,x_8^4,x_8^{5},x_8^8,x_8^{12},x_{9},x_{12}\}$. If $x_3\rightarrow x_8^4$,     $x_4\rightarrow x_8^2$, and $x_8^5\rightarrow x_{12}$, then $(x_3,x_4,x_8^2,x_8^4,x_8^5,x_9,$ $x_{12})$ is the forest ordering of $H_{20}$. Else $(x_2,x_3,x_7,x_8^4,x_8^8,x_8^{12},x_9)$ or $(x_2,x_4,x_7,x,x_8^8,x_8^{12},x_{12})$ is the cyclic ordering of $H_{20}$, where $x\in \{x_8^2,x_8^5\}$. All the above discussion implies that  for all $H\in \{S_1,...,S_4,S_9,H_1,H_{12},$ $H_{20},H_{25},H_{26},H_{31}\}$, $T$ contains $H$. This terminates the proof. $\hfill { \square }$
 \end{proof}
 \begin{lemma}\label{C2} 
 Let $\tilde{c}>0$, $0 < \tilde{\lambda} < \frac{1}{3(26)^{15}}$ be constants, let $\epsilon >0$ be small enough, and let $w$ be a $\{0,1\}$-vector. Let $\chi =(W_{1},...,W_{9})$ be a smooth $(\tilde{c},\tilde{\lambda} ,w)$$-$structure of an $n$-vertex $\epsilon$$-$critical tournament $T$ corresponding to $\mathcal{K}_2$ under $\theta_{\mathcal{K}_2}$. Then for all $H\in \mathcal{C}_2$, $T$ contains $H$. \end{lemma}
 \begin{proof}
 By Lemma \ref{path-galaxy}, $\widehat{\mathcal{K}}_2$ is well-contained in $\chi$. Let  $(x_1,x_2,x_3^1,...,x_3^7,x_4,...,x_8,x_9^1,x_9^2,x_9^{3})$ be the forest ordering of the well-contained copy of $\widehat{\mathcal{K}}_2$. If $x_3^5\leftarrow x_6$, then $(x_1,x_2,x_3^4,...,x_3^7,x_6)$ is the cyclic ordering of $S_5$. Let's assume now that $(x_3^5, x_6)$ is an arc of $T$. Either $(x_3^3,x_4)$ and $(x_3^1,x_8)$ are arcs of $T$, or there exists $x\in \{x_3^1,x_3^3\}$ such that $x\leftarrow \{x_4,x_8\}$. If the former holds, $(x_3^1,x_3^2,x_3^3,x_4,x_6,x_7,x_8)$ is the forest ordering of $S_5$. Otherwise, the latter holds, and so $(x_2,x,x_3^5,x_4,x_5,x_6,x_8)$ is the cyclic ordering of $S_5$. Hence, $T$ contains $S_5$. To complete the proof, we will prove that $T$ contains $S_{10}$. If $x_6\leftarrow x_9^2$, then $(x_1,x_2,x_3^6,x_6,x_9^1,x_9^2,x_9^3)$ is the cyclic ordering of $S_{10}$, and so we are done. So let's assume that $(x_6,x_9^2)$ is an arc of $T$. We will discuss according to the orientation of the arcs of $T$$\mid$$\{x_3^1,x_3^3,x_4,x_8\}$. If  $(x_3^3,x_4)$ and $(x_3^1,x_8)$ are arcs of $T$, then $(x_2,x_3^1,x_3^3,x_4,x_6,x_8,x_9^2)$ is the forest ordering of $S_{10}$. Otherwise, there exists $x\in \{x_3^1,x_3^3\}$ such that $x\leftarrow \{x_4,x_8\}$, and so $(x_2,x,x_3^6,x_4,x_6,x_8,x_9^2)$ is the cyclic ordering of $S_{10}$. Hence $T$ contains $S_{10}$. This completes the proof. $\hfill { \square }$  
\end{proof}
\begin{lemma}\label{H37H3} 
 Let $\tilde{c}>0$, $0 < \tilde{\lambda} < \frac{1}{3(16)^{10}}$ be constants, let $\epsilon >0$ be small enough, and $w=(0,0,0,1,0,0,0,0,0)$. Let $\overline{c}=\frac{3\tilde{c}}{2},\overline{\lambda}=\frac{2\tilde{\lambda}}{3}$. Let $(W_{1},...,W_{9})$ be a smooth $(\overline{c},\overline{\lambda} ,w)$$-$structure of an $n$-vertex $\epsilon$$-$critical tournament $T$, such that $\mid$$W_4$$\mid$ is divisible by $12$. Then for all $H\in \mathcal{C}_3$, $T$ contains $H$. \end{lemma}
 \begin{proof}
We are going to prove that $T$ contains a digraph $D$, obtained by merging $\widehat{\mathcal{K}}_3$ and two special substructures. Then we will prove that the tournament induced by $T$$\mid$$V(D)$ contains $H_3$ and $H_{37}$. Denote by  $\hat{W}_4$, the set $W_4^2\cup W_4^3\cup W_4^6\cup W_4^7\cup (\bigcup_{i=9}^{12}W_4^i) $. Since $\chi =(W_3,\hat{W}_4,W_6,W_{7},W_{9})=(F_1,...,F_5)$ is a smooth $(\tilde{c},\tilde{\lambda},\tilde{w})$-structure corresponding to $\mathcal{K}_3$ under $\theta_{\mathcal{K}_3}$, then by Lemma \ref{path-galaxy}, $\widehat{\mathcal{K}}_3$ is well-contained in $\chi$. So there exist vertices $x_i\in W_i$ for $i=3,6,7,9$, and there exist vertices $x_4^i\in W_4^i$ for $i=2,3,6,7,9,...,12$, such that $(x_3,x_4^2,x_4^3,x_4^6,x_4^7,x_4^9,...,x_4^{12},x_6,x_7, x_9)$ is the forest ordering of $\widehat{\mathcal{K}}_3$. Denote this copy of $\widehat{\mathcal{K}}_3$ by $X$. Let $W_{i}^{*} = \displaystyle{\bigcap_{p\in V(X)}W_{i,p}}$ for $i=1,2,5,8$, and let $W_{4}^{i*} = \displaystyle{\bigcap_{p\in V(\tilde{X})}W^i_{4,p}}$ for $i=1,4,5,8$, where $\tilde{X}=\{x_3,x_6,x_7,x_9\}$.  Then by Lemma \ref{u}, for all  $i\in\{1,4,5,8\}$, $\mid$$W_{4}^{i*}$$\mid$ $\geq \frac{1-48\tilde{\lambda}}{12}\overline{c}tr(T)\geq  \frac{\overline{c}}{24}tr(T)$ since $\overline{\lambda} \leq \frac{1}{96}$, and $\mid$$W_{i}^{*}$$\mid$ $\geq (1-12\overline{\lambda})\mid$$W_{i}$$\mid$ $\geq \frac{1}{2}\mid$$W_{i}$$\mid$ $\geq \frac{\overline{c}}{2}n$ for $i=1,2,5,8$, since $\overline{\lambda} \leq \frac{1}{24}$. Now since we can assume that $\epsilon < min \{log_{\frac{\overline{c}}{4}}(1-\frac{\overline{c}}{24}),log_{\frac{\overline{c}}{8}}(\frac{1}{2})\}$, then by Lemma \ref{g}, there exist vertices $x_i\in W_i^*$ for $i=1,2,5,8$, and there exist vertices $x_{4}^i\in W_{4}^{i*}$ for $i=1,4,5,8$, such that $\{x_1,x_4^1\}\leftarrow x_{8}$, $x_1\leftarrow x_4^4$, $\{x_2,x_4^5\}\leftarrow x_{5}$, and $x_2\leftarrow x_4^8$. $D$ is the digraph with vertex set $V(X)\cup \{x_1,x_2,x_4^1, x_4^4,x_4^5,x_4^8, x_{5},x_8\}$, and arc set  $E(X)\cup \{(x_8,x_1),(x_8,x_4^1),(x_4^4,x_1),(x_5,x_2),(x_5,x_4^5),(x_4^8,x_2)\}$. Notice that if $(x_3,x_4^9)$ and $(x_4^{12},x_7)$ are arcs of $T$, then $(x_3,x_4^9,...,x_4^{12},x_7,x_9)$ and $(x_3,x_4^9,...,x_4^{12},x_6,x_7)$ is the forest ordering of $H_3$ and $H_{37}$ respectively. So $T$ contains $H_3$ and $H_{37}$, and we are done. So let's assume now that $(x_4^9,x_3)$ or $(x_7,x_4^{12})$ is an arc of $T$. Then there exists $x\in \{x_4^9,x_4^{12}\}$, such that $x_3\leftarrow x\leftarrow x_7$. Now we will discuss according to the orientation of the edges $x_2x_4^5$ and $x_4^8x_5$ to find a copy of $H_3$ in $T$. Either $(x_2,x_4^5)$ and $(x_4^8,x_5)$ are arcs of $T$, and so $(x_2,x_4^5,...,x_4^{8},x_5,x_9)$ is the forest ordering of $H_3$, or there exists $r\in \{x_4^5,x_4^{8}\}$, such that $x_2\leftarrow r\leftarrow x_5$, and so $(x_2,x_3,r,x,x_5,x_6,x_7)$ is the cyclic ordering of $H_3$. Hence $T$ contains $H_3$.  To complete the proof, we will study all the possible  orientations of the edges $x_1x_4^1$ and $x_4^4x_8$ and confirm the existence of  $H_{37}$ in $T$ as a subtournament. If $x_1\rightarrow x_4^1$ and $x_4^4\rightarrow x_8$, then $(x_1,x_4^1,...,x_4^{4},x_6,x_8)$ is the forest ordering of $H_{37}$. Otherwise, there exists $u\in \{x_4^1,x_4^4\}$, such that $x_1\leftarrow u\leftarrow x_8$. Then $(x_1,x_3,u,x,x_6,x_7,x_8)$ is the cyclic ordering of $H_{37}$. Hence $T$ contains $H_{37}$. This confirms our lemma. $\hfill { \square }$
 \end{proof}
\begin{lemma}\label{C4} 
 Let $\tilde{c}>0$, $0 < \tilde{\lambda} < \frac{1}{3(24)^{14}}$ be constants, let $\epsilon >0$ be small enough, and $w=(1,0,0,1,0,1,0,0,0$ $,0,1,0,0)$. Let $\overline{c}=\frac{3\tilde{c}}{2},\overline{\lambda}=\frac{2\tilde{\lambda}}{3}$. Let $(W_{1},...,W_{14})$ be a smooth $(\overline{c},\overline{\lambda} ,w)$$-$structure of an $n$-vertex $\epsilon$$-$critical tournament $T$, such that $\mid$$W_4$$\mid$ is divisible by $6$. Then for all $H\in \mathcal{C}_4$, $T$ contains $H$. \end{lemma}
\begin{proof}
Let $w^*:=(1,0,1,0,1,0,0,1,0,1,0)$ be a $\{0,1\}$-vector. Clearly, $\chi =(W_1,W_3,W_4^1\cup W_4^2\cup W_4^5\cup W_4^6,W_5,W_6,W_7,W_9,W_{11},W_{12},W_{14})$ is a smooth $(\tilde{c},\tilde{\lambda},w^*)$-structure corresponding to $\mathcal{K}_4$ under $\theta_{\mathcal{K}_4}$. Then by Lemma \ref{path-galaxy}, $\widehat{\mathcal{K}}_4$ is well-contained in $\chi$. Let $(x_1^1,x_1^2,x_1^3,x_3,x_4^1,x_4^2,x_4^5,x_4^6,x_5,x_6,x_7,x_9,x_{11}^1,x_{11}^2,x_{12},x_{14})$ be the forest ordering of the well-contained copy of $\widehat{\mathcal{K}}_4$. Denote this copy by $X$. Since we can assume that $\epsilon < min \{log_{\frac{\overline{c}}{8}}(1-\frac{\overline{c}}{12}),log_{\frac{\overline{c}}{8}}(1-\frac{\overline{c}}{2}),log_{\frac{\overline{c}}{2}}(1-\frac{\overline{c}}{48})\}$ and since $\overline{\lambda} \leq \frac{1}{144}$, then Lemmas \ref{u}, \ref{w} and \ref{x} implies that there exist vertices $x_i\in \bigcap_{p\in V(X)}W_{i,p}$ for $i=2,8,10,13$, and there exist vertices $x_{4}^i\in \bigcap_{p\in V(\tilde{X})}W^i_{4,p}$ for $i=3,4$, such that $x_2\leftarrow x^3_{4}\leftarrow x_8$ and $x_4^4\leftarrow x_{10}\leftarrow x_{13}$, where $\tilde{X}=X-\{x_4^1,x_4^2,x_4^5,x_4^6\}$.  We will prove first that $T$ contains $H_8$. If $x_3\rightarrow x_4^2$, $x_4^6\rightarrow  x_{12}$, and $x^4_4\rightarrow x_{13}$, then $(x_3,x_4^2,x_4^4,x_4^6,x_{10},x_{12},x_{13})$ is the forest ordering of $H_8$. Otherwise, $(x_{13},x_4^4)$ is an arc of $T$, or there exists $u\in \{x_4^2,x_4^6\}$, such that $x_3\leftarrow u\leftarrow x_{12}$. So, $(x_4^4,x_9,x_{10},x_{11}^1,x_{11}^2,x_{13},x_{14})$ or $(x_3,x_4^1,u,x_6,x_7,x_{12},x_{14})$ is the cyclic ordering of $H_8$. Hence $T$ contains $H_8$. To complete the proof, it remains to prove that $T$ contains $H_2$ and $H_{18}$. To this end, we will consider two cases: Either $x_1^2\leftarrow x_4^1$, or $x_1^2\rightarrow x_4^1$. Assume first that the former holds. Then $(x_1^1,x_1^2,x_1^3,x_4^1,x_4^5,x_5,x_7)$ and $(x_1^1,x_1^2,x_1^3,x_4^1,x_5,x_6,x_7)$ is the cyclic ordering of $H_{18}$ and $H_{2}$ respectively. Then $T$ contains $H_2$ and $H_{18}$, and we are done. Let's assume now that the latter holds. If there exists $u\in \{x_4^2,x_4^6\}$ such that $x_3\leftarrow u\leftarrow x_{12}$, then $(x_1^2,x_3,x_4^1,u,x_7,x_8,x_{12})$ and $(x_1^2,x_3,x_4^1,u,x_6,x_7,x_{12})$ is the cyclic ordering of $H_{2}$ and $H_{18}$ respectively. Otherwise, $x_3\rightarrow x_4^2$ and $x_4^6\rightarrow  x_{12}$. In this case $(x_3,x_4^1,x_4^2,x_4^5,x_4^6,x_{12},x_{14})$ is the forest ordering of $H_2$. So, in all possible cases $T$ contains $H_2$. It remains to prove that $T$ contains $H_{18}$ in case $x_1^2\rightarrow x_4^1$, $x_3\rightarrow x_4^2$, and $x_4^6\rightarrow  x_{12}$. If $(x_2,x_8)$ is an arc of $T$, then $(x_2,x_3,x_4^2,x_4^3,x_4^6,x_8,x_{12})$ is the forest ordering of $H_{18}$. Otherwise, $(x_8,x_2)$ is an arc of $T$, and so $(x_1^2,x_2,x_4^1,x_4^3,x_4^5,x_7,x_8)$ is the cyclic ordering of $H_{18}$, and we are done. This ends the proof. $\hfill { \square }$
\end{proof}
\vspace{3mm}\\ If there exists a smooth $(c,\lambda,w)$-structure $\Sigma$ as in Lemma \ref{C1} (resp. \ref{C2})(resp. \ref{H37H3})(resp. \ref{C4}), then we say that $\Sigma$ \textit{corresponds to} $\mathcal{C}_1$ (resp. $\mathcal{C}_2$) (resp. $\mathcal{C}_3$) (resp. $\mathcal{C}_4$). We are ready to state the main theorem in this section and prove that every tournament in $\mathcal{C}:=\bigcup_{i=1}^4\mathcal{C}_i$ satisfies EHC.
\begin{theorem}
Every tournament in $\mathcal{C}$ satisfies the Erd\"{o}s-Hajnal conjecture.
\end{theorem}
\begin{proof}
Let $H\in \mathcal{C}$. Then $H\in \mathcal{C}_i$, for some $i\in \{1,...,4\}$. Assume to that contrary that $H$ does not satify EHC. Take $\epsilon > 0$ small enough. Then there exists an $H$$-$free tournament $T'$ such that $tr(T') < \mid$$T'$$\mid^{\epsilon}$. Let $ T $ be the smallest $ H - $free tournament on $ n $ vertices such that $ tr(T) < \mid $$T$$ \mid^{\epsilon} $. Then $ T $ is $ \epsilon - $critical. By Corollary \ref{i}, $ T $ contains a $ (c,\lambda, w)- $smooth structure corresponding to $\mathcal{C}_i$, for some $c>0, \lambda >0$ small enough. Now Lemma \ref{C1}, Lemma \ref{C2}, Lemma \ref{H37H3}, or Lemma \ref{C4} implies that $T$ contains $H$, a contradiction. This completes the proof. $\hfill {\square }$
\end{proof}   
\section{The analogous conjecture of Erd\"{o}s and Hajnal}
Instead of forbidding just one tournament, one can state the analogous conjecture where we forbid two or more tournaments. In this section we prove the Erd\"{o}s-Hajnal conjecture for several couples and triplets of tournaments. 
\subsection{Forbidding two tournaments}
For each of the following tournaments we will define some special orderings of their vertex set and the set of backward arcs under each ordering (see Figure \ref{fig:orderingsC}).
\begin{itemize}[-]
\item \textit{Path ordering of $S_{7}$}, $\theta_{P_{7}} := (v,a,b,f,e,d,c)$, $E(\theta_{P_{7}}) = \lbrace (c,f),(f,v),(d,b),(e,a) \rbrace $ .
\item \textit{Forest ordering of $S_{8}$}, $\theta_{f_{8}} := (v,c,a,f,e,d,b)$, $E(\theta_{f_{8}}) = \lbrace (b,f),(f,v),(d,a),(b,c) \rbrace $ .
\item \textit{Cyclic ordering of $S_{12}$}, $\theta_{C_{12}} := (c,a,b,e,f,d,v)$, $E(\theta_{C_{12}}) = \lbrace (d,b),(d,c),(b,c),(v,e),(v,a) \rbrace $ .
\item \textit{Forest ordering of $S_{13}$}, $\theta_{f_{13}} := (b,c,d,a,e,f,v)$, $E(\theta_{f_{13}}) = \lbrace (v,d),(e,c),(a,b),(v,e),(f,b) \rbrace $ .
\item \textit{Forest ordering of $S_{14}$}, $\theta_{f_{14}} := (b,c,d,a,v,f,e)$, $E(\theta_{f_{14}}) = \lbrace (v,d),(e,c),(a,b),(e,v),(f,b) \rbrace $ .
\item \textit{Forest ordering of $H_{15}$}, $\theta_{f_{15}} := (v_{3},v_{4},v_{5},v_{1},v_{7},v_{2},v_{6})$,\\ $E(\theta_{f_{15}})= \lbrace (v_{2},v_{3}),(v_{2},v_{4}),(v_{1},v_{3}),(v_{6},v_{5}),(v_{6},v_{7}) \rbrace $.
\item \textit{Forest ordering of $H_{47}$}, $\theta_{f_{47}} := (v_{7},v_{1},v_{2},v_{3},v_{4},v_{6},v_{5})$,\\ $E(\theta_{f_{47}})= \lbrace (v_{5},v_{1}),(v_{5},v_{2}),(v_{4},v_{1}),(v_{3},v_{7}),(v_{6},v_{7}) \rbrace $.
\item \textit{Forest ordering of $H_{6}$}, $\theta_{f_{6}} := (v_{1},v_{7},v_{2},v_{3},v_{4},v_{6},v_{5})$, $E(\theta_{f_{6}})= \lbrace (v_{5},v_{2}),(v_{5},v_{1}),(v_{4},v_{1}) \rbrace $.
\item \textit{Cyclic ordering of $H_{10}$}, $\theta_{C_{10}} := (v_{2},v_{7}, v_{4},v_{1},v_{6},v_{3}, v_{5})$, $E(\theta_{C_{10}})= \lbrace (v_{1},v_{2}),(v_{5},v_{2}),(v_{5},v_{1}),(v_{3},v_{4}) \rbrace $.
\item \textit{Cyclic ordering of $H_{38}$}, $\theta_{C_{38}} := (v_{2},v_{4},v_{1},v_{3},v_{6},v_{5},v_{7})$,\\ $E(\theta_{C_{38}})= \lbrace (v_{1},v_{2}),(v_{5},v_{2}),(v_{5},v_{1}),(v_{7},v_{3}),(v_{3},v_{4}),(v_{7},v_{4}) \rbrace $.
\item \textit{Cyclic ordering of $H_{45}$}, $\theta_{C_{45}} := (v_{2},v_{7},v_{4},v_{1},v_{3},v_{6},v_{5})$,\\ $E(\theta_{C_{45}})= \lbrace (v_{1},v_{2}),(v_{5},v_{2}),(v_{5},v_{1}),(v_{3},v_{4}),(v_{6},v_{7}) \rbrace $.
\item \textit{Cyclic ordering of $H_{27}$}, $\theta_{C_{27}} := (v_{2},v_{4},v_{7},v_{1},v_{6},v_{3},v_{5})$, $E(\theta_{C_{27}})= \lbrace (v_{1},v_{2}),(v_{5},v_{2}),(v_{5},v_{1}),(v_{3},v_{4}) \rbrace $.
\item \textit{Cyclic ordering of $H_{44}$}, $\theta_{C_{44}}$$ := $$(v_{2},v_{7},v_{4},v_{1},v_{6},v_{3},v_{5})$,\\ $E(\theta_{C_{44}})$$=$$ \lbrace (v_{1},v_{2}),(v_{5},v_{2}),(v_{5},v_{1}),(v_{3},v_{4}),(v_{6},v_{7}) \rbrace $.
\item \textit{Cyclic ordering of $H_{4}$}, $\theta_{C_{4}}$$ := $$(v_{2},v_{7},v_{4},v_{1},v_{3},v_{6},v_{5})$, $E(\theta_{C_{4}})$$=$$ \lbrace (v_{1},v_{2}),(v_{5},v_{2}),(v_{5},v_{1}),(v_{3},v_{4}) \rbrace $.
\item \textit{Cyclic ordering of $H_{34}$}, $\theta_{C_{34}}$$ := $$(v_{2},v_{4},v_{1},v_{6},v_{3},v_{5},v_{7})$,\\ $E(\theta_{C_{34}})$$=$$ \lbrace (v_{1},v_{2}),(v_{5},v_{2}),(v_{5},v_{1}),(v_{3},v_{4}),(v_{7},v_{3}),(v_{7},v_{4}) \rbrace $.
 \item \textit{Forest ordering of $H_{13}$}, $\theta_{f_{13}} := (v_{1},v_{2},v_{3},v_{4},v_{6},v_{5},v_{7})$, $E(\theta_{f_{13}})= \lbrace (v_{5},v_{1}),(v_{5},v_{2}),(v_{4},v_{1}),(v_{7},v_{3}) \rbrace $.
 \item \textit{Forest ordering of $R_{1}$}, $\theta_{R_{1}} := (v_{1},v_{2},v_{3},v_{4},v_{5},v_{6},v_{7})$, $E(\theta_{R_{1}})= \lbrace (v_{6},v_{3}),(v_{5},v_{2}),(v_{4},v_{1}),(v_{6},v_{1}) \rbrace $.
 \item \textit{Forest ordering of $R_{2}$}, $\theta_{R_{2}} := (v_{1},v_7,v_{2},v_{3},v_{4},v_{5},v_{6})$, $E(\theta_{R_{2}})= \lbrace (v_{6},v_{3}),(v_{5},v_{2}),(v_{4},v_{1}),(v_{6},v_{1}) \rbrace $.
 \item \textit{Forest ordering of $R_{3}$}, $\theta_{R_{3}} := (v_{1},v_{2},v_{3},v_{4},v_{5},v_{6},v_{7})$,\\ $E(\theta_{R_{3}})= \lbrace (v_{6},v_{3}),(v_{5},v_{2}),(v_{4},v_{1}),(v_{6},v_{1}),(v_{7},v_{2}) \rbrace $.
 \item \textit{Forest ordering of $R_{4}$}, $\theta_{R_{4}} := (v_{1},v_{2},v_7,v_{3},v_{4},v_{5},v_{6})$, $E(\theta_{R_{4}})= \lbrace (v_{6},v_{3}),(v_{5},v_{2}),(v_{4},v_{1}),(v_{6},v_{1}) \rbrace $.
\item \textit{Forest ordering of $R_{7}$}, $\theta_{R_{7}} := (v_{1},v_7,v_{2},v_{3},v_{4},v_{5},v_{6})$,\\ $E(\theta_{R_{7}})= \lbrace (v_{6},v_{3}),(v_{5},v_{2}),(v_{4},v_{1}),(v_{6},v_{1}),(v_{5},v_{7}) \rbrace $.
\end{itemize} 
 \begin{figure}[h]
 \centering
	\includegraphics[scale=0.8, width=16cm, height=7.9cm]{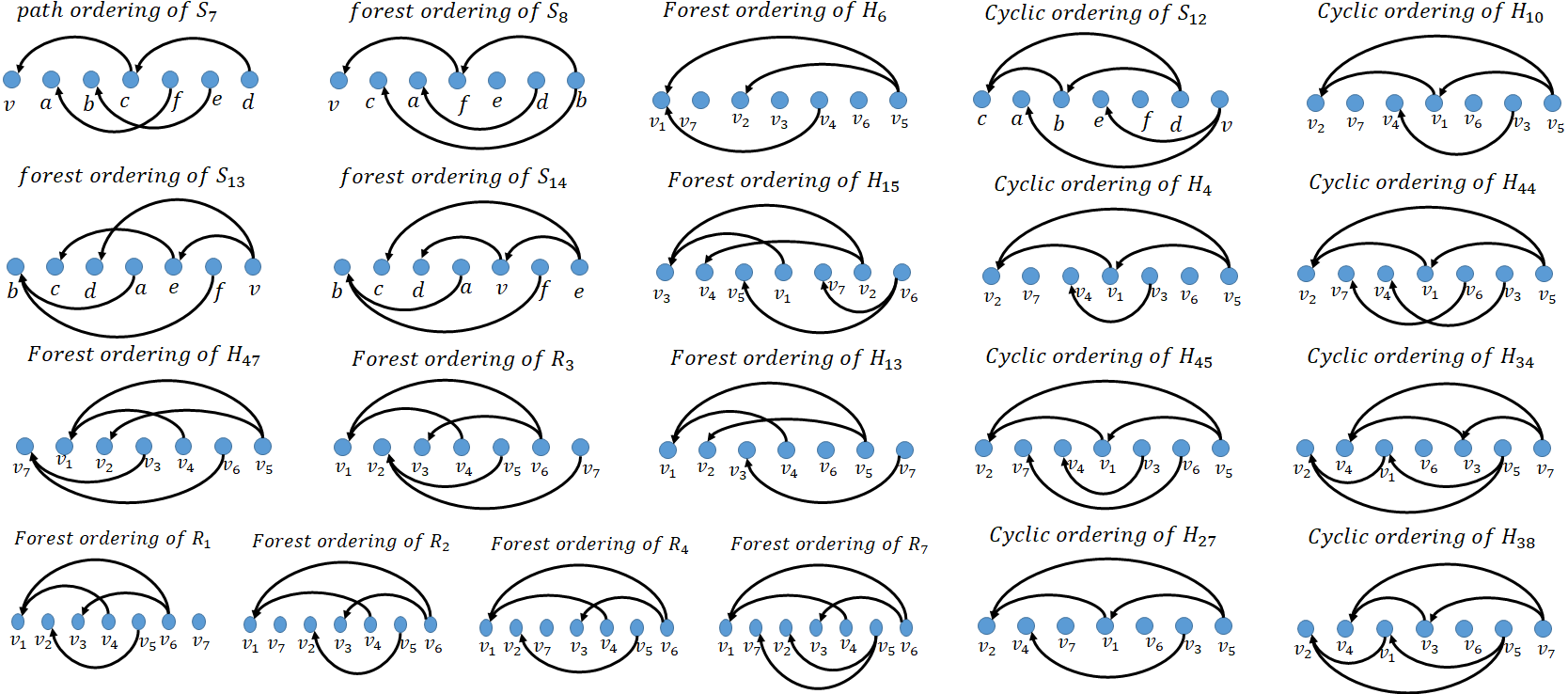}
	\caption{Crucial orderings of the vertices of some tournaments in classes $\mathcal{S}$, $\mathcal{H}$, and $\mathcal{R}$. All the non drawn arcs are forward.}
	\label{fig:orderingsC}
\end{figure}
The following are four classes of tournaments in $\mathcal{S}\cup\mathcal{H}\cup\mathcal{R}$. \\
- $\mathcal{F}_1:=\{H_{34},H_{38}\}$.\\
- $\mathcal{F}_2:=\{H_6,H_{13},H_{15},H_{47},R_1,...,R_4,R_7\}$.\\
- $\mathcal{F}_3:=\{H_4,H_{10},H_{27},H_{44},H_{45}\}$.\\
- $\mathcal{F}_4:=\{S_8,S_{13},S_{14}\}$.
\begin{lemma}\label{F1}
Let $c>0$, $0 < \lambda < \frac{1}{(10)^{7}}$ be constants, let $\epsilon >0$ be small enough, and $w=(0,0,1,0,0,0,0)$. Let $\tilde{c}=\frac{5c}{3},\tilde{\lambda}=\frac{3\lambda}{5}$. Let $(A_{1},...,A_{7})$ be a smooth $(\tilde{c},\tilde{\lambda} ,w)$$-$structure of an $n$-vertex $\epsilon$$-$critical tournament $T$, such that $\mid$$A_3$$\mid$ is divisible by $5$. Then for all $H\in\mathcal{F}_1$, $T$ contains $H$ or $T$ contains $S_7$. 
\end{lemma}
 \begin{proof}
Let $G$ be a super galaxy under an ordering  $\theta =(1,...,5)$ of its vertices, such that the connected components of $B(G,\theta)$ are the two right stars $\{2,5\}$ and $\{1,3,4\}$. Let $\hat{w}:=(1,0,0)$ be a $\{0,1\}$-vector. Clearly $\chi =(A_3^1\cup A_3^2\cup A_3^4,A_4,A_5)$ is a smooth $(c,\lambda,\hat{w})$-structure corresponding to $G$ under $\theta$. Let $\epsilon_l$ be as in Lemma  \ref{supergalaxy}. Since we can assume that $\epsilon <\epsilon_l$, then   by Lemma \ref{supergalaxy}, $\widehat{G}$ is well-contained in $\chi$. So there exist vertices $x_i\in A_i$ for $i=4,5$, and there exist vertices $x_3^i\in A_3^i$ for $i=1,2,4$, such that $(x_3^1,x_3^2,x_3^4,x_4,x_5)$ is the mutant super galaxy ordering of $\widehat{G}$. Let $X:=\{x_3^1,x_3^2,x_3^4,x_4,x_5\}$. Let $A_{i}^{*} = \displaystyle{\bigcap_{p\in X}A_{i,p}}$ for $i=2,7$, and let $A_{3}^{5*} = \displaystyle{\bigcap_{p\in \{x_4,x_5\}}A^5_{3,p}}$.  Then by Lemma \ref{u}, $\mid$$A_{i}^{*}$$\mid$ $\geq (1-5\tilde{\lambda})\mid$$A_{i}$$\mid$ $\geq \frac{1}{2}\mid$$A_{i}$$\mid$ $\geq  \frac{\tilde{c}}{2}n$ for $i=2,7$,  since $\tilde{\lambda} \leq \frac{1}{10}$, and $\mid$$A_{3}^{5*}$$\mid$ $\geq \frac{1-10\tilde{\lambda}}{5}\tilde{c}tr(T)\geq  \frac{\tilde{c}}{10}tr(T)$ since $\tilde{\lambda} \leq \frac{1}{20}$.  Now since  we can assume that $\epsilon < log_{\frac{\tilde{c}}{2}}(1-\frac{\tilde{c}}{40})$, then Lemma \ref{w}  implies that there exist vertices $x_{3}^5\in A_{3}^{5*}$, $x_i\in A_i^*$  for $i=2,7$, such that $x_2\leftarrow x^3_{5}\leftarrow x_7$. Since $\tilde{\lambda} \leq \frac{1}{40}$ and since we can assume that $\epsilon < log_{\frac{\tilde{c}}{2}}(1-\frac{\tilde{c}}{40})$, then similarly we prove that   there exist vertices $x_{3}^3\in A_{3}^{3}$, $x_i\in A_i$  for $i=1,6$, such that $x_1\leftarrow x^3_{3}\leftarrow x_6$, $x_1\rightarrow X\cup \{x_2,x_3^5,x_7\}$, $x_2\rightarrow x_3^3\rightarrow \{x_4,x_5,x_7\}$, and $X\cup \{x_2,x_3^5\}\rightarrow x_7$. We have two cases. If $(x_6,x_1)$ and $(x_7,x_2)$ are arcs of $T$, then $(x_1,x_2,x_3^3,x_3^4,x_3^5,x_6,x_7)$ and  $(x_1,x_2,x_3^3,x_3^5,x_5,x_6,x_7)$ is the cyclic ordering of $H_{34}$ and $H_{38}$ respectively. So for all $H\in\{H_{34},H_{38}\}$, $T$ contains $H$. Otherwise, $(x_1,x_3^1,x_3^2,x_3^3,x_4,x_5,x_6)$ or  $(x_2,x_3^1,x_3^2,x_3^5,x_4,x_5,x_7)$ is the path ordering of $S_{7}$. This confirms our lemma. $\hfill {\square}$
\end{proof}
\begin{lemma} \label{F2}
 Let $\tilde{c}>0$, $0 < \tilde{\lambda} < \frac{1}{3(30)^{17}}$ be constants, let $\epsilon >0$ be small enough, and let $w$ be a $\{0,1\}$-vector. Let $\chi =(W_{1},...,W_{9})$ be a smooth $(\tilde{c},\tilde{\lambda} ,w)$$-$structure of an $n$-vertex $\epsilon$$-$critical tournament $T$ corresponding to $\mathcal{K}_5$ under $\theta_{\mathcal{K}_5}$. If  $H\in \mathcal{F}_2$ and $H'\in \mathcal{F}_3$, then $T$ contains $H$ or $T$ contains $H'$. \end{lemma}
 \begin{proof}
 We are going to prove that $T$ contains $\widehat{\mathcal{K}}_5$, then we will study the orientation of some arcs in $T$$\mid$$\widehat{\mathcal{K}}_5$ to find $H$ or $H'$ in $T$ as a subtournament. Clearly, by Lemma \ref{path-galaxy}, $\widehat{\mathcal{K}}_5$ is well-contained in $\chi$. Let  $T$$\mid$$\{v_1,...,v_{17}\}$ be the copy $\widehat{\mathcal{K}}_5$ in $T$, with $(v_1,...,v_{17})$ is the forest ordering of $\widehat{\mathcal{K}}_5$.  Clearly, $\{v_4,...,v_{14}\}\subseteq W_3$, $T$$\mid$$\{v_4,...,v_{14}\}$ is a transitive tournament, and $(v_4,...,v_{14})$ is the transitive ordering of $T$$\mid$$\{v_4,...,v_{14}\}$. We will discuss according to the orientation of the edges $v_2v_8$, $v_{11}v_{17}$, and $v_{12}v_{18}$. If $(v_2,v_8)$, $(v_{11},v_{17})$, and $(v_{12},v_{18})$ are arcs of $T$, then $(v_2,v_4,v_8,v_9,v_{11},v_{13},v_{17})$, $(v_2,v_8,v_{10},v_{11},v_{15},v_{17},v_{19})$, $(v_2,v_8,v_{10},v_{11},v_{14},v_{17},v_{19})$, $(v_1,v_2,v_8,v_9,v_{11},v_{13},v_{17})$, $(v_2,v_6,v_8,v_{11},v_{15},v_{17},v_{18})$, $(v_2,v_4,v_5,v_8,v_{11},v_{16},v_{17})$, $(v_2,v_3,v_8,v_{11},v_{12},v_{17},$ 
 $v_{18})$, $(v_2,v_3,v_4,v_8,v_{11},v_{12},v_{17})$, and $(v_2,v_5,v_7,v_8,v_{11},v_{16},v_{17})$ is the forest ordering of $H_6$, $H_{13}$, $H_{15}$, $H_{47}$, $R_1,...,R_4$, and  $R_7$ respectively. So for all $H\in \{H_6,H_{13},H_{15},H_{47},R_1,...,R_4,R_7\}$, $T$ contains $H$. Otherwise, $(v_8,v_2)$ or $(v_{17},v_{11})$ or $(v_{18},v_{12})$ is an arc of $T$. Assume that $(v_8,v_2)$ is an arc of $T$ (else, the argument is similar, and we omit it).  Then $(v_2,v_4,v_6,v_8,v_{15},v_{16},v_{17})$, $(v_2,v_4,v_6,v_8,v_{12},v_{15},v_{17})$, $(v_2,v_6,v_7,v_8,v_{13},v_{15},v_{17})$, $(v_2,v_6,v_7,v_8,v_{15},v_{16},v_{17})$, and  $(v_2,v_5,v_6,v_8,v_{15},v_{16},v_{17})$ is the cyclic ordering of $H_4$, $H_{10}$, $H_{27}$, $H_{44}$, and $H_{45}$ respectively. So for all $H'\in \{H_4,H_{10},H_{27},H_{44},H_{45}\}$, $T$ contains $H'$. This confirms our lemma. $\hfill {\square}$ 
\end{proof}
\begin{lemma}\label{F4} 
Let $\epsilon >0$ be small enough, and let $w:=(0,1,0,0,1,0,1,0,0,0,0)$ be a $\{0,1\}$-vector. Let $\tilde{c}>0$, $0 < \tilde{\lambda} < \frac{1}{3(20)^{12}}$ be constants, and let $\overline{c}=\frac{7\tilde{c}}{5},\overline{\lambda}=\frac{5\tilde{\lambda}}{7}$. Let $(W_{1},...,W_{11})$ be a smooth $(\overline{c},\overline{\lambda} ,w)$$-$structure of an $n$-vertex $\epsilon$$-$critical tournament $T$, such that $\mid$$W_2$$\mid$ is divisible by $7$. Then for all $H\in \mathcal{F}_4$, $T$ contains $H$ or $T$ contains $S_{12}$. \end{lemma}
\begin{proof}
Since $T$ is $\epsilon$$-$critical and $\chi =(W_1,W_2-(W_2^3\cup W_2^4),W_3,W_5,W_6,W_7,W_9,W_{10},W_{11})$ is a smooth $(\tilde{c},\tilde{\lambda},\tilde{w})$-structure corresponding to $\mathcal{K}_6$ under $\theta_{\mathcal{K}_6}$, where $\tilde{w}=(0,1,0,1,0,1,0,0,0)$, then by Lemma \ref{path-galaxy}, $\widehat{\mathcal{K}}_6$ is well-contained in $\chi$. Let $(u_1,...,u_{15})$ be the forest ordering of the well-contained copy of $\widehat{\mathcal{K}}_6$. Let $U_1=\{u_1,..,u_{15}\}$ and $U_2=U_1-\{u_2,u_3,u_4,u_5,u_6\}$. Since we can assume that $\epsilon < min \{log_{\frac{\overline{c}}{4}}(1-\frac{\overline{c}}{14}),log_{\frac{\overline{c}}{8}}(\frac{1}{2})\}$, then analogously, as in the proof of Lemma \ref{path-galaxy}, using Lemma \ref{g}, one can prove that there exist vertices $x_2^i\in \displaystyle{\bigcap_{p\in V(U_2)}W^i_{2,p}}$ for $i=3,4$ and there exist vertices $x_i\in \displaystyle{\bigcap_{p\in V(U_2)}W_{i,p}}$ for $i=4,8$, such that $\{x_2^3,x_4\}\leftarrow x_8$ and $x_2^4\leftarrow x_4$. We will discuss according to the orientation of the edge $u_4u_9$.    Either $u_4\leftarrow u_9$, or $u_4\rightarrow u_9$. If  the former holds, then $(u_4,u_8,...,u_{11},u_{14},u_{15})$ is the cyclic ordering of $S_{12}$, and so $T$ contains $S_{12}$ as a subtournament and we are done. Let's assume now that the latter holds.   If $u_6\rightarrow u_{12}$, $u_2\rightarrow u_{13}$, $u_3\rightarrow u_{7}$, $x_2^4\rightarrow x_{8}$, and $x_2^3\rightarrow x_{4}$, then $(u_2,u_3,u_5,u_7,u_8,u_{11},u_{13})$, $(u_1,u_2,u_3,u_6,u_7,u_{12},u_{13})$, and $(u_1,x_2^3,x_2^4,u_6,x_4,u_{12},x_8)$ is the forest ordering of $S_8$, $S_{13}$, and $S_{14}$ respectively. So for all $H\in \{S_8,S_{13},S_{14}\}$, $T$ contains $H$ and we are done. Otherwise, there exist $z_1\in \{v_1,v_2,v_3,x_2^3,x_2^4\}$, $z_2\in \{v_6,v_7,x_4\}$, and $z_3\in \{v_{12},x_8,v_{13}\}$, such that $T$$\mid$$\{z_1,z_2,z_3\}$ is a cycle. In this case $(z_1,v_4,z_2,v_9,v_{10},z_3,v_{14})$ is the cyclic ordering of $S_{12}$, and we are done. In view of our observations, the lemma is confirmed. $\hfill {\square}$
\end{proof} 
\vspace{3mm}\\ Let $M_i\in \mathcal{F}_i$ for $i=1,...,4$.  If there exists a smooth $(c,\lambda,w)$-structure $\Sigma$ as in Lemma \ref{F1} (resp. \ref{F2})(resp. \ref{F4}), then we say that $\Sigma$ \textit{corresponds to} $(M_1, S_7)$ (resp. $(M_2,M_3)$) (resp. $(M_4,S_{12})$). 
\begin{theorem}
If $D_{1}\in \mathcal{F}_1$ and $D_{2}=S_7$, or: $D_{1}\in \mathcal{F}_2$ and $D_{2}\in \mathcal{F}_3$, or: $D_{1}\in \mathcal{F}_4$ and $D_{2}=S_{12}$, then $(D_{1},D_{2})$ satisfies the Erd\"{o}s-Hajnal Conjecture. 
\end{theorem}
\begin{proof}
Assume otherwise. Take $\epsilon > 0$ small enough. Then there exists an $(D_1,D_2)$$-$free $ \epsilon - $critical tournament $T$. By Corollary \ref{i}, $ T $ contains a $ (c,\lambda, w)- $smooth structure corresponding to $(D_{1},D_{2})$, for some $c>0, \lambda >0$ small enough. Now Lemma \ref{F1}, Lemma \ref{F2}, or Lemma \ref{F4} implies that $T$ contains $D_1$ or $T$ contains $D_2$, a contradiction. This completes the proof. $\hfill {\square }$
\end{proof}  
\subsection{Forbidding three tournaments}
\begin{figure}[h]
	\includegraphics[width=1\linewidth]{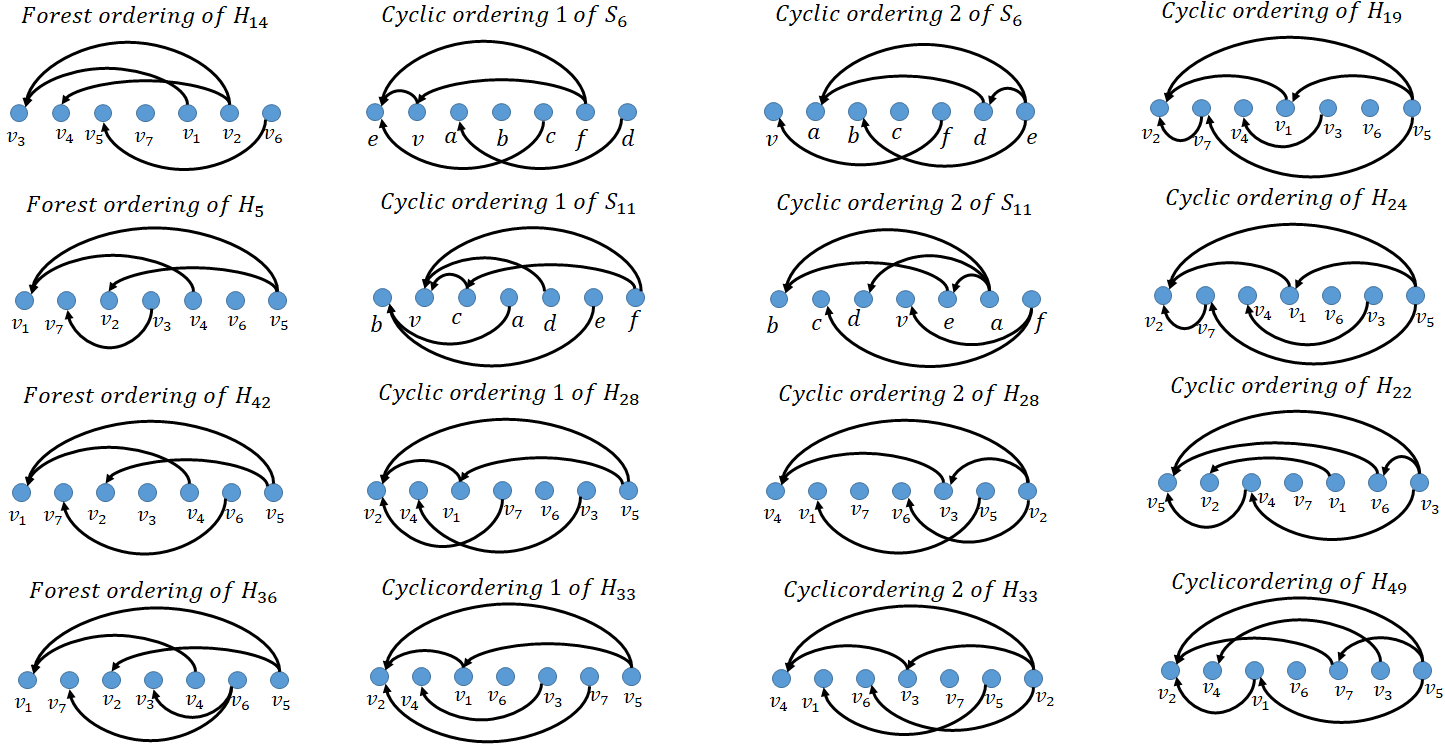}
	\caption{Crucial orderings of the vertices of some tournaments in classes $\mathcal{S}$ and $\mathcal{H}$. All the non drawn arcs are forward.}
	\label{fig:orderingsT}
\end{figure}
In what  follows we define special orderings of the vertex set of some tournaments in classes $\mathcal{S}$ and $\mathcal{H}$, and the set of backward arcs under each ordering (see Figure \ref{fig:orderingsT}).
\begin{itemize}[-]
\item \textit{Forest ordering of $H_{5}$}: $\theta_{f_{5}} = (v_{1},v_{7},v_{2},v_{3},v_{4},v_{6},v_{5})$, $E(\theta_{f_{5}})= \lbrace (v_{5},v_{2}),(v_{5},v_{1}),(v_{4},v_{1}),(v_{3},v_{7}) \rbrace $.
\item \textit{Forest ordering of $H_{14}$}: $\theta_{f_{14}} = (v_{3},v_{4},v_{5},v_{7},v_{1},v_{2},v_{6})$, $E(\theta_{f_{14}})= \lbrace (v_{2},v_{3}),(v_{2},v_{4}),(v_{1},v_{3}),(v_{6},v_{5}) \rbrace $.
\item \textit{Cyclic ordering of $H_{22}$}: $\theta_{C_{22}} = (v_{5},v_{2},v_{4},v_{7},v_{1},v_{6},v_{3})$, $E(\theta_{C_{22}})= \lbrace (v_{3},v_{4}),(v_{3},v_{5}),(v_{3},v_{6}),(v_{4},v_{5}),$ $(v_{6},v_{5}),(v_{1},v_{2}) \rbrace $.
\item \textit{Cyclic ordering $1$ of $S_{6}$}: $\theta^1_{C_{6}} = (e,v,a,b,c,f,d)$, $E(\theta^1_{C_{6}})= \lbrace (c,e),(d,a),(f,e),(f,v),(v,e) \rbrace $.
\item \textit{Cyclic ordering $2$ of $S_{6}$}: $\theta^2_{C_{6}} = (v,a,b,c,f,d,e)$, $E(\theta^1_{C_{6}})= \lbrace (e,d),(e,a),(e,b),(d,a),(f,v) \rbrace $.
\item \textit{Cyclic ordering $1$ of $S_{11}$}: $\theta^{1}_{C_{11}} = (b,v,c,a,d,e,f)$, $E(\theta^{1}_{C_{11}})= \lbrace (c,v),(a,b),(e,b),(d,v),(f,c),(f,v) \rbrace $;\\ \textit{Cyclic ordering $2$ of $S_{11}$}: $\theta^{2}_{C_{11}} = (b,c,d,v,a,e,f)$, $E(\theta^{2}_{C_{11}})= \lbrace (a,e),(a,b),(e,b),(a,d),(f,c),(f,v) \rbrace $.
\item \textit{Cyclic ordering $1$ of $H_{33}$}: $\theta^{1}_{C_{33}} = (v_{2},v_{4},v_{1},v_{6},v_{3},v_{7},v_{5})$, $E(\theta^{1}_{C_{33}})= \lbrace (v_{5},v_{2}),(v_{1},v_{2}),(v_{5},v_{1}),(v_{7},v_{2})$ $,(v_{3},v_{4}) \rbrace $.\\ \textit{Cyclic ordering $2$ of $H_{33}$}: $\theta^{2}_{C_{33}} = (v_{4},v_{1},v_{6},v_{3},v_{7},v_{5},v_{2})$, $E(\theta^{2}_{C_{33}})= \lbrace (v_{3},v_{4}),(v_{2},v_{4}),(v_{2},v_{3}),(v_{2},v_{6}),\\(v_{5},v_{1}) \rbrace $.
\item \textit{Cyclic ordering of $H_{19}$}: $\theta_{C_{19}} = (v_{2},v_{7},v_{4},v_{1},v_{3},v_{6},v_{5})$, $E(\theta_{C_{19}})= \lbrace (v_7,v_2),(v_{1},v_{2}),(v_{5},v_{2}),(v_{5},v_{1}),$ $(v_{5},v_{7}),(v_{3},v_{4}) \rbrace $.
\item \textit{Forest ordering of $H_{36}$}: $\theta_{f_{36}} = (v_{1},v_{7},v_{2},v_{3},v_{4},v_{6},v_{5})$, $E(\theta_{f_{36}})= \lbrace (v_{5},v_{2}),(v_{5},v_{1}),(v_{4},v_{1}),(v_{6},v_{7}),$ $(v_{6},v_{3}) \rbrace $. 
\item \textit{Cyclic ordering $1$ of $H_{28}$}: $\theta^{1}_{C_{28}} = (v_{2},v_{4},v_{1},v_{7},v_{6},v_{3},v_{5})$, $E(\theta^{1}_{C_{28}})= \lbrace (v_{5},v_{1}),(v_{7},v_{2}),(v_{1},v_{2}),(v_{3},v_{4}),$ $(v_{5},v_{2}) \rbrace $.\\\textit{Cyclic ordering $2$ of $H_{28}$}: $\theta^{2}_{C_{28}} = (v_{4},v_{1},v_{7},v_{6},v_{3},v_{5},v_{2})$, $E(\theta^{2}_{C_{28}})= \lbrace (v_{2},v_{3}),(v_{2},v_{6}),(v_{2},v_{4}),(v_{5},v_{1}),$ $(v_{3},v_{4}) \rbrace $. 
\item \textit{Cyclic ordering of $H_{24}$}: $\theta_{C_{24}} = (v_{2},v_{7},v_{4},v_{1},v_{6},v_{3},v_{5})$, $E(\theta_{C_{24}})= \lbrace (v_7,v_2),(v_{5},v_{1}),(v_{5},v_{7}),(v_{1},v_{2}),$ $(v_{3},v_{4}),(v_{5},v_{2}) \rbrace $.
\item \textit{Forest ordering of $H_{42}$}: $\theta_{f_{42}} = (v_{1},v_{7},v_{2},v_{3},v_{4},v_{6},v_{5})$, $E(\theta_{f_{42}})= \lbrace (v_{5},v_{2}),(v_{5},v_{1}),(v_{4},v_{1}),(v_{6},v_{7}) \rbrace $. 
\item \textit{Cyclic ordering of $H_{49}$}: $\theta_{C_{49}} = (v_{2},v_{4},v_{1},v_{6},v_{7},v_{3},v_{5})$, $E(\theta_{C_{49}})= \lbrace (v_{5},v_{1}),(v_{5},v_{7}),(v_{7},v_{2}),(v_{1},v_{2}),$ $(v_{3},v_{4}),(v_{5},v_{2}) \rbrace $.
\end{itemize}
\begin{lemma} 
 Let $\tilde{c}>0$, $0 < \tilde{\lambda} < \frac{1}{3(26)^{14}}$ be constants, let $\epsilon >0$ be small enough, and let $w$ be a $\{0,1\}$-vector. Let $\chi =(W_{1},...,W_{8})$ be a smooth $(\tilde{c},\tilde{\lambda} ,w)$$-$structure of an $n$-vertex $\epsilon$$-$critical tournament $T$ corresponding to $\mathcal{K}_7$ under $\theta_{\mathcal{K}_7}$. If  $H\in \{H_5,H_{42},H_{36}\}$, $H'\in \{S_6,S_{11},H_{28},H_{33}\}$, and $H''\in \{H_{19},H_{22},H_{24},H_{49}\}$, then $T$ contains $H$ or $T$ contains $H'$ or $T$ contains $H''$. \end{lemma}
 \begin{proof}
$\epsilon >0$ is small enough and $T$ is an $\epsilon$$-$critical tournament, then since $\chi$ is a smooth $(\tilde{c},\tilde{\lambda} ,w)$$-$structure of $T$ corresponding to $\mathcal{K}_7$ under $\theta_{\mathcal{K}_7}$, Lemma \ref{path-galaxy} implies that $\widehat{\mathcal{K}}_7$ is well-contained in $\chi$. Let $(v_1,...,v_{17})$ be the forest ordering of the well-contained copy $\mathcal{D}$ of $\widehat{\mathcal{K}}_7$. We are going to  study the orientation of some arcs in $T$$\mid$$V(\mathcal{D})$ to find $H$ or $H'$ or $H''$ in $T$ as a subtournament.  Clearly, $\{v_4,...,v_{13}\}\subseteq W_4$, with $\xi (v_i)=i$ for $i= 4,...,13$. We will discuss according to the orientation of the edges $v_2v_6$ and $v_{12}v_{16}$. If $(v_2,v_6)$ and $(v_{12},v_{16})$ are arcs of $T$, then $(v_2,v_3,v_6,v_7,v_{12},v_{15},v_{16})$, $(v_2,v_6,v_{9},v_{10},v_{12},v_{16},v_{17})$, $(v_2,v_4,v_{6},v_{11},v_{12},v_{14},v_{16})$, and $(v_2,v_4,v_6,v_8,v_{12},v_{14},v_{16})$ is the forest ordering of $H_5$, $H_{14}$, $H_{36}$, and $H_{42}$ respectively. So for all $H\in \{H_5,H_{42},H_{36}\}$, $T$ contains $H$. Otherwise if $(v_6,v_2)$ and $(v_{12},v_{16})$ are arcs of $T$, then $(v_2,v_6,v_9,v_{10},v_{12},v_{16},v_{17})$, $(v_1,v_2,v_6,v_{8},v_{12},v_{13},v_{16})$, $(v_2,v_4,v_{6},v_{12},v_{13},v_{14},v_{16})$, and $(v_2,v_3,v_6,v_8,v_{10},v_{12},v_{16})$ is the cyclic ordering $1$ of $S_6$, $S_{11}$, $H_{28}$, and $H_{33}$ respectively. So for all $H'\in \{S_6,S_{11},H_{28},H_{33}\}$, $T$ contains $H'$. Otherwise if $(v_2,v_6)$ and $(v_{16},v_{12})$ are arcs of $T$, then $(v_1,v_2,v_6,v_{7},v_{8},v_{12}$ $,v_{16})$, $(v_2,v_5,v_6,v_{9},v_{12},v_{16},v_{17})$, $(v_2,v_4,v_{5},v_{6},v_{12},v_{14},v_{16})$, and $(v_2,v_4,v_6,v_{12},v_{13},v_{14},v_{16})$ is the cyclic ordering $2$ of $S_6$, $S_{11}$, $H_{28}$, and $H_{33}$ respectively. So for all $H'\in \{S_6,S_{11},H_{28},H_{33}\}$, $T$ contains $H'$. Otherwise $(v_6,v_2)$ and $(v_{16},v_{12})$ are arcs of $T$, then $(v_2,v_4,v_6,v_{8},v_{12},v_{14},v_{16})$, $(v_2,v_6,v_{11},v_{12},v_{14},v_{15},v_{16})$, $(v_2,v_6,v_{11},v_{12},v_{13},v_{14},v_{16})$, and $(v_2,v_3,v_6,v_{8},v_{10},v_{12},v_{16})$ is the cyclic ordering  of $H_{49}$, $H_{19}$, $H_{24}$, and $H_{22}$ respectively. So for all $H''\in \{H_{19},H_{22},H_{24},H_{49}\}$, $T$ contains $H''$.  This completes the proof. $\hfill {\square}$ 
\end{proof}
\begin{theorem}
If $H\in \{H_5,H_{42},H_{36}\}$, $H'\in \{S_6,S_{11},H_{28},H_{33}\}$, and $H''\in \{H_{19},H_{22},H_{24},H_{49}\}$,  then $(H,H',H'')$ satisfies the Erd\"{o}s-Hajnal Conjecture. 
\end{theorem}
\begin{proof}
Assume otherwise. Take $\epsilon > 0$ small enough. Then there exists an $(H,H',H'')$$-$free $ \epsilon - $critical tournament $T$. By Corollary \ref{i}, $ T $ contains a $ (c,\lambda, w)- $smooth structure corresponding to $\mathcal{K}_7$ under $\theta_{\mathcal{K}_7}$, for some $c>0, \lambda >0$ small enough. By the previous lemma, $T$ contains $H$ or $T$ contains $H'$ or $T$ contains $H''$, a contradiction. This terminates the proof. $\hfill {\square }$
\end{proof}

\section{The landscape of seven-vertex tournaments}\label{landscape}
\noindent \sl {Proof of Theorem \ref{r}.} \upshape First note that, if there exists $ v$$\in$$ H $ such that $ d^{+}(v) = 6 $, then $ N^{+}(v) $ is a non-trivial homogeneous set and $H$ is not prime, and so we have two cases: Either $H$$\mid$$N^{+}(v) \ncong K_{6}$, and so $ 5 $ holds, or $H$$\mid$$N^{+}(v) \approx K_{6}$ then $ H \approx R_{1} $ and so $ 4 $ holds. Hence, we will suppose that  for all $ v \in H, d^{+}(v) \leq 5 $ and $ d^{-}(v) \leq 5 $.\\
\textbf{Case  1:} There exist $ v \in H$, such that  $ d^{+}(v) = 5, N^{+}(v) = \lbrace a,b,c,d,e \rbrace$, and $ N^{-}(v) = \lbrace f \rbrace. $ We can assume that $ 1 \leq d^{+}_{\lbrace a,b,c,d,e \rbrace}(f) \leq 4 $, since else $5$ holds.\\
\textbf{1.1} $ d^{+}_{\lbrace a,b,c,d,e \rbrace}(f) = 4, N^{+}_{\lbrace a,b,c,d,e \rbrace}(f) = \lbrace a,b,c,d \rbrace$, and $ N^{-}(f) = \lbrace e \rbrace $.\\ \textbf{1.1.1} $H$$\mid$$\lbrace a,b,c,d \rbrace$ is a transitive tournament (assume without loss of generality that $(a,b,c,d)$ is its transitive ordering). Then $(v,e,f,a,b,c,d)$ is a galaxy ordering of $H$ and so $ 1 $ holds.\\
\textbf{1.1.2} $H$$\mid$$\lbrace a,b,c,d \rbrace$ contains a cyclic triangle. Assume without loss of generality that $(a,b),(b,c),(c,a)$ are arcs of $H$. Now we will study the neighbors of $d$ in $ \lbrace a,b,c \rbrace $ and the neighbors of $e$ in $ \lbrace a,b,c,d \rbrace $.\\
$\bullet$ $ d^{+}_{\lbrace a,b,c \rbrace}(d) = 0 $: We have three cases. If $e$ has at most one in-neighbor in $\lbrace a,b,c,d \rbrace $, say $d$ or $b$ (without loss of generality), then $(v,e,f,a,b,c,d)$ is a galaxy ordering of $H$. If $ d^{-}_{\lbrace a,b,c,d \rbrace}(e) = 2$, then either $ N^{-}_{\lbrace a,b,c,d \rbrace}(e)= \lbrace a,d \rbrace $ (without loss of generality) and so $(f,v,a,b,d,e,c)$ is a galaxy ordering of $H$, or $ N^{-}_{\lbrace a,b,c,d \rbrace}(e)= \lbrace a,b \rbrace $ (without loss of generality) and so $(f,v,a,b,e,c,d)$ is a galaxy ordering of $H$. If $e$ has at most one out-neighbor in $\lbrace a,b,c,d \rbrace$, say $d$ or $b$ (without loss of generality) if one exists, then $(f,v,a,b,c,d,e)$ is a galaxy ordering of $H$. So in all cases $ 1 $ holds.\\
$\bullet$ $ d^{+}_{\lbrace a,b,c \rbrace}(d) = 1 $ (assume without loss of generality that $(d,b) \in E(H)$): If $ d^{-}_{\lbrace a,b,c,d \rbrace}(e) = 0 $, then $(f,v,e,c,a,d,b)$ is a galaxy ordering of $H$ and $ 1 $ holds. If $ d^{-}_{\lbrace a,b,c,d \rbrace}(e) = 1 $, then either $(c,e) \in E(H)$ and so $(f,v,c,e,a,d,b)$ is a galaxy ordering of $H$, or $N^{-}(e) \subseteq \lbrace v,a,b,d \rbrace$ and so $(v,e,f,c,a,d,b)$ is a galaxy ordering of $H$. If $ d^{-}_{\lbrace a,b,c,d \rbrace}(e) = 2 $, then $ N^{-}(e) = \lbrace b,d,v \rbrace $ and so $H \approx H^{c}_{1} $ and $ 3 $ holds, or $ N^{-}(e) = \lbrace a,d,v \rbrace $ and so $(f,v,a,d,e,b,c)$ is a galaxy ordering of $H$, or $ N^{-}(e) = \lbrace b,c,v \rbrace $ and so $(f,v,c,a,d,b,e)$ is a galaxy ordering of $H$, or $ N^{-}(e) = \lbrace a,b,v \rbrace $ and so $(f,v,a,b,e,c,d)$ is a galaxy ordering of $H$, or $ N^{-}(e) = \lbrace a,c,v \rbrace $ and so $(f,v,c,a,e,d,b)$ is a galaxy ordering of $H$, or $ N^{-}(e) = \lbrace c,d,v \rbrace $ and so $(f,v,c,a,d,e,b)$ is a galaxy ordering of $H$. So $ 1 $ holds in the above 5 cases. If $ d^{-}_{\lbrace a,b,c,d \rbrace}(e) = 3 $, we have four cases: $(e,a) $$\in$$ E(H)$ and so $(f,v,c,a,d,b,e)$ is a galaxy ordering of $H$, or $(e,b) \in E(H)$ and so $(f,v,c,a,d,e,b)$ is a galaxy ordering of $H$, or $(e,d) \in E(H)$ and so $(f,v,a,b,c,e,d)$ is a galaxy ordering of $H$, or $(e,c) \in E(H)$ and so $(f,v,a,d,b,e,c)$ is a galaxy ordering of $H$. So $1$ holds in the above four cases. If $ d^{-}_{\lbrace a,b,c,d \rbrace}(e) = 4 $, then $(f,v,a,b,c,d,e)$ is a galaxy ordering of $H$ and $ 1 $ holds.\\
Now note that if $ d^{+}_{\lbrace a,b,c \rbrace}(d) = 2 $ then this case is isomorphic to case 1.1.2.\\
$\bullet$ $ d^{+}_{\lbrace a,b,c \rbrace}(d) = 3 $: If $e$ has at most one in-neighbor in $\lbrace a,b,c,d \rbrace$, say $d$ or $b$ (without loss of generality) if one exists, then $(v,e,f,d,a,b,c)$ is a galaxy ordering of $H$. If $ d^{-}_{\lbrace a,b,c,d \rbrace}(e) = 2 $, then either $ N^{-}_{\lbrace a,b,c,d \rbrace}(e) = \lbrace b,d \rbrace $ (without loss of generality) and so $(v,e,f,d,a,b,c)$ is a galaxy ordering of $H$, or $ N^{-}_{\lbrace a,b,c,d \rbrace}(e) = \lbrace a,c \rbrace $ (without loss of generality) and so $(f,v,d,a,b,c,e)$ is a galaxy ordering of $H$. If $ d^{-}_{\lbrace a,b,c,d \rbrace}(e) \geq 3 $, then either $(e,d) $$\in$$ E(H)$, or $(e,b) $$\in$$ E(H)$ (without loss of generality) (if one exists), and in both cases $(f,v,d,a,b,c,e)$ is a galaxy ordering of $H$. So $ 1 $ holds. \\
\textbf{1.2} $ d^{+}_{\lbrace a,b,c,d,e \rbrace}(f) = 3 $, say $ N^{+}_{\lbrace a,b,c,d,e \rbrace}(f) = \lbrace a,b,c \rbrace , N^{-}(f) = \lbrace d,e \rbrace $, and assume without loss of generality that $(d,e) \in E(H)$.\\
\textbf{1.2.1} $H$$\mid$$\lbrace a,b,c\rbrace$ is a transitive tournament. Assume without loss of generality that $(a,b),(a,c),(b,c)$ are arcs of $H$. Now we are going to study the orientation of the arcs between $ \lbrace d,e \rbrace $ and $ \lbrace a,b,c \rbrace $.\\
\textbf{1.2.1.1} $ d^{+}_{\lbrace a,b,c \rbrace}(e) = 3 $. Then $(v,d,e,f,a,b,c)$ is a galaxy ordering of $H$, and so $ 1 $ holds.\\
\textbf{1.2.1.2} $ d^{+}_{\lbrace a,b,c \rbrace}(e) = 2 $.\\
$\bullet$ If $ N^{+}_{\lbrace a,b,c \rbrace}(e) = \lbrace b,c \rbrace $, then either $a \in N^{+}(d)$ and so $(v,d,e,f,a,b,c)$ is a galaxy ordering of $H$, or $a \in N^{-}(d)$ and so $(v,a,d,e,f,b,c)$ is a galaxy ordering of $H$. Hence in both cases $ 1 $ holds.\\
$\bullet$ If $ N^{+}_{\lbrace a,b,c \rbrace}(e) = \lbrace a,b \rbrace $, then either $c \in N^{+}(d)$ and so $(v,d,e,f,a,b,c)$ is a galaxy ordering of $H$ and $ 1 $ holds, or $c \in N^{-}(d)$ and here we have 4 cases: If $a$ and $b$ are out-neighbors or in-neighbors of $d$, then $ \lbrace a,b \rbrace $ is a non-trivial homogeneous set and $H$ is $K_{6}-$free, and so $ 5 $ holds. If $(a,d)$ and $(d,b)$ are arcs of $H$, then $H $$\approx$$ H_{2}$ and so $ 3 $ holds. If $(d,a)$ and $(b,d)$ are arcs of $H$, then $H \approx H^{c}_{3} $ and so $ 3 $ holds. \\
$\bullet$ If $ N^{+}_{\lbrace a,b,c \rbrace}(e) = \lbrace a,c \rbrace $, then $c$ is an in-neighbor of $d$. Now either $b \in N^{+}(d)$ and so $(v,d,e,f,a,b,c)$ is a galaxy ordering of $H$ and $ 1 $ holds, or $b \in N^{-}(d)$ and here we have two cases: either $a \in N^{-}(d) $ and so $H $$\approx$$ H_{4} $ and $ 3 $ holds, or $(d,a)\in E(H)$ and so $H $$\approx$$ S_{1}$ and $ 6 $ holds.\\
\textbf{1.2.1.3} $ d^{+}_{\lbrace a,b,c \rbrace}(e) = 1 $.\\
$\bullet$ If $(e,a) $$\in$$ E(H)$, then we have four cases: If $ d^{+}_{\lbrace a,b,c \rbrace}(d) = 3 $, then $(v,d,e,f,a,b,c)$ is a galaxy ordering of $H$ and $ 1 $ holds. If $ d^{+}_{\lbrace a,b,c \rbrace}(d) = 2 $, then either $(a,d) $$\in$$ E(H)$ and so $(v,d,e,f,a,b,c)$ is a galaxy ordering of $H$, or $N^{-}(d) \subseteq \lbrace v,b,c \rbrace$ and so $(d,f,v,a,b,c,e)$ is a galaxy ordering of $H$. So $ 1 $ holds. Now if $ d^{+}_{\lbrace a,b,c \rbrace}(d) = 1 $, then $(d,a) \in E(H)$ and so $(d,f,v,a,b,c,e)$ is a galaxy ordering of $H$ and $ 1 $ holds, or $(d,b) $$\in$$ E(H)$ and so $H $$\approx$$ H^{c}_{5} $ and $ 3 $ holds, or $(d,c) $$\in$$ E(H)$ and so $H $$\approx$$ H_{6} $ and $ 3 $ holds. Finally if $ d^{+}_{\lbrace a,b,c \rbrace}(d) = 0 $, then $\lbrace b,c \rbrace $ is a non-trivial homogeneous set and $H$ is $K_{6}-$free. So $ 5 $ holds.\\
$\bullet$ If $(e,b)$$\in$$ E(H)$, then we have four cases: If $ d^{+}_{\lbrace a,b,c \rbrace}(d) = 3 $, then $(v,d,e,f,a,b,c)$ is a galaxy ordering of $H$ and $ 1 $ holds. If $ d^{+}_{\lbrace a,b,c \rbrace}(d) = 2 $, then $(a,d) $$\in$$ E(H)$ and so $(v,a,d,e,f,b,c)$ is a galaxy ordering of $H$, or $(b,d)$$ \in $$E(H)$ and so $(v,d,e,f,a,b,c)$ is a galaxy ordering of $H$, or $(c,d)$$ \in$$ E(H)$ and so $(d,f,v,a,b,c,e)$ is a galaxy ordering. Hence $ 1 $ holds. Now if $ d^{+}_{\lbrace a,b,c \rbrace}(d) = 1 $, then $(d,c) $$\in$$ E(H)$ and so $(v,a,d,e,f,b,c)$ is a galaxy ordering of $H$ and $ 1 $ holds, or $(d,a) $$\in$$ E(H)$ and so $H$$ \approx$$ H_{7}^{c} $ and $ 3 $ holds, or $(d,b) $$\in$$ E(H)$ and so $\lbrace d,e \rbrace $ is a non-trivial homogeneous set, $H$ is $K_{6}-$free, and $ 5 $ holds. Finally if $ d^{+}_{\lbrace a,b,c \rbrace}(d) = 0 $, then $\lbrace a,v \rbrace$ is a non-trivial homogeneous set and $H$ is $K_{6}-$free. So $ 5 $ holds.\\
$\bullet$ $(e,c) $$\in$$ E(H)$. Then $c$ must be an in-neighbor of $d$, since else $d^{-}(c) = 6$, a contradiction. So we have four cases: If $\lbrace a,b \rbrace \subseteq N^{+}(d)$ or $\lbrace a,b \rbrace \subseteq N^{-}(d)$, then $\lbrace a,b \rbrace $ is a non-trivial homogeneous set and $H$ is $K_{6}-$free. So $ 5 $ holds. If $(d,a)$ and $(b,d)$ are arcs of $H$, then $(d,f,v,a,b,e,c)$ is a galaxy ordering of $H$ and $ 1 $ holds. Finally if $(a,d)$ and $(d,b)$ are arcs of $H$, then $(v,a,d,e,f,b,c)$ is a galaxy ordering of $H$ and $ 1 $ holds.\\
\textbf{1.2.1.4} $ d^{+}_{\lbrace a,b,c \rbrace}(e)= 0 $. Then for $ N^{-}(d) \subseteq \lbrace v,a,b,c \rbrace $, $(d,f,v,a,b,c,e)$ is a galaxy ordering of $H$ and $ 1 $ holds.\\
\textbf{1.2.2} $H$$\mid$$\lbrace a,b,c\rbrace$ is a cycle. Assume without loss of generality that $(a,b),(b,c),(c,a)$ are arcs of $H$. So we have four cases:\\
$\bullet$ $ d^{+}_{\lbrace a,b,c \rbrace}(e)= 3 $. Then if $d$ has at most one in-neighbor in $\lbrace a,b,c \rbrace$, say $b$ (without loss of generality) if one exists, then $(v,d,e,f,a,b,c)$ is a galaxy ordering of $H$. If $d$ has one out-neighbor in $\lbrace a,b,c \rbrace$, say $c$ (without loss of generality), then $H$$\approx$$ S_{2}$. Finally if $ d^{+}_{\lbrace a,b,c \rbrace}(d) = 0 $, then $\lbrace a,b,c \rbrace $ is a non-trivial homogeneous set and $H$ is $K_{6}-$free. Hence $1,5$, or $ 6 $ holds.\\
$\bullet$ $e$ has one in-neighbor in $\lbrace a,b,c \rbrace$, say $b$ (without loss of generality): If $d^{+}_{\lbrace a,b,c \rbrace}(d) = 3$, then $(v,d,e,f,a,b,c)$ is a galaxy ordering of $H$. If $d^{+}_{\lbrace a,b,c \rbrace}(d) = 2$, then either $N^{-}_{\lbrace a,b,c \rbrace}(d) \subseteq \lbrace b,c \rbrace$ and so $(v,d,e,f,a,b,c)$ is a galaxy ordering of $H$, or $(a,d) \in E(H)$ and so $(v,d,e,f,c,a,b)$ is a galaxy ordering of $H$. If $ d^{+}_{\lbrace a,b,c \rbrace}(d) = 1 $, then  $(d,a) \in E(H)$ and so $H $$\approx$$ S_{3} $, or $(d,b) \in E(H)$ and so $H $$\approx$$ H^{c}_{8}$, or $(d,c)\in E(H)$ and so $H $$\approx$$ H_{9} $. Finally if $ d^{+}_{\lbrace a,b,c \rbrace}(d) = 0 $, then $H \approx H_{10} $. Hence $1, 3$, or $6 $ holds.\\
$\bullet$ $ d^{+}_{\lbrace a,b,c \rbrace}(e)= 1 $. Assume without loss of generality that $(e,b) \in E(H)$. If $ d^{-}_{\lbrace a,b,c \rbrace}(d) \leq 1 $, then $N^{-}_{\lbrace a,b,c \rbrace}(d) \subseteq \lbrace a \rbrace$ and so $(d,f,v,c,a,e,b)$ is a galaxy ordering of $H$, or $(c,d) \in E(H)$ and so $(f,v,c,d,a,e,b)$ is a galaxy ordering of $H$, or $(b,d) \in E(H)$ and so $H\approx S_{4}$. Now if $ d^{+}_{\lbrace a,b,c \rbrace}(d) = 1 $, then $(d,b) \in E(H)$ and so $(f,v,c,a,d,e,b)$ is a galaxy ordering of $H$, or $(d,a) $$\in$$ E(H)$ and so $H $$\approx$$ H^{c}_{11} $, or $(d,c) \in E(H)$ and so $H $$\approx$$ R_{2} $. If $ d^{+}_{\lbrace a,b,c \rbrace}(d) = 0 $, then $H \approx H_{12}^{c} $. Hence $1,3,4,$ or $6$ holds.\\
$\bullet$ $ d^{+}_{\lbrace a,b,c \rbrace}(e)= 0 $. If $d$ has at most one in-neighbor in $\lbrace a,b,c \rbrace$, say $b$ (without loss of generality) if one exists, then $(d,f,v,a,b,c,e)$ is a galaxy ordering of $H$. If $ d^{+}_{\lbrace a,b,c \rbrace}(d) = 1 $, assume without loss of generality that $(d,c) $$\in$$ E(H)$. Then $(f,v,a,b,d,c,e)$ is a galaxy ordering of $H$. If $ d^{+}_{\lbrace a,b,c \rbrace}(d) = 0 $, then $(f,v,a,b,c,d,e)$ is a galaxy ordering of $H$. So $ 1 $ holds.\\
\textbf{1.3} $ d^{+}_{\lbrace a,b,c,d,e \rbrace}(f)= 2. $ We can assume without loss of generality that $ N^{+}(f)= \lbrace v,d,e \rbrace $, $ N^{-}(f)= \lbrace a,b,c \rbrace $, and $(e,d) \in E(H)$. Then $ 1 \leq d^{+}(d) \leq 3$.\\
\textbf{1.3.1} $H$$ \mid $$ \lbrace a,b,c \rbrace $ is a transitive tournament. Assume without loss of generality that $(a,b)$,$(b,c)$,$(a,c)$ are arcs of $H$. We will discuss according to the neighbors of $e$. \\
$\bullet$ If $ d^{+}_{\lbrace a,b,c \rbrace}(e)= 0 $, then $(v,a,b,c,f,e,d)$ is a galaxy ordering of $H$ and $ 1 $ holds.\\
$\bullet$ If $ d^{+}_{\lbrace a,b,c \rbrace}(e)= 1 $, then: \\
$\ast$ If $(e,a) \in E(H)$, then either $(a,d) \in E(H)$ and so $(v,a,b,c,f,e,d)$ is a galaxy ordering and $ 1 $ holds, or $(d,a) \in E(H)$ and here we have three cases: if $ \lbrace b,c \rbrace \subseteq N^{+}(d) $ or $ \lbrace b,c \rbrace \subseteq N^{-}(d) $, then $\lbrace b,c \rbrace $ is a non-trivial homogeneous set and $H$ is $K_{6}-$free and so $ 5 $ holds.  If $(b,d)$ and $(d,c)$ are arcs of $H$, then $H \approx H^{c}_{13}$ and so $ 3 $ holds. Finally  if $(d,b)$ and $(c,d)$ are arcs of $H$, then $H \approx H^{c}_{14} $ and $ 3 $ holds. \\
$\ast$ If $(e,b) \in E(H)$, then either $(b,d) \in E(H)$ and so $(v,a,b,c,f,e,d)$ is a galaxy ordering of $H$ and $ 1 $ holds, or $(d,b) \in E(H)$ and here we have four cases: if $ \lbrace a,c \rbrace \subseteq N^{-}(d) $, then $(v,a,b,c,f,e,d)$ is a galaxy ordering of $H$ and $ 1 $ holds. Otherwise if $ \lbrace a,c \rbrace \subseteq N^{+}(d) $, then $H \approx H^{c}_{15}$ and $ 3 $ holds. Otherwise if $(d,c)$ and $(a,d)$ are arcs, then $H \approx H^{c}_{1}$ and $ 3 $ holds. Otherwise, $(c,d)$ and $(d,a)$ are arcs, then $H \approx S_{5}$ and $ 6 $ holds. \\
$\ast$ If $(e,c) \in E(H)$, then for $N^{+}(d) \subseteq \lbrace a,b,c \rbrace$, $(v,a,b,e,c,f,d)$ is a galaxy ordering of $H$ and $ 1 $ holds.\\
$\bullet$ $ d^{+}_{\lbrace a,b,c \rbrace}(e)= 2 $. Here we will discuss according to the neighbors of $e$.\\
$\ast$ If $ N^{+}_{\lbrace a,b,c \rbrace}(e)= \lbrace b,c \rbrace $, then for $ N^{+}(d) \subseteq \lbrace a,b,c \rbrace$, $(v,a,e,b,c,f,d)$ is a galaxy ordering and $ 1 $ holds.\\
$\ast$ If $ N^{+}_{\lbrace a,b,c \rbrace}(e)= \lbrace b,a \rbrace $, then we have three cases: if $ \lbrace b,a \rbrace \subseteq N^{-}(d) $, then $(d,c)$$\in$$E(H)$ and so $(v,a,b,c,f,e,d)$ is a galaxy ordering. If $ \lbrace b,a \rbrace \subseteq N^{+}(d) $, then $\lbrace b,a \rbrace $ is a non-trivial homogeneous set and $H$ is $ K_{6}-$free whatever $c$ is an out-neighbor or in-neighbor of $d$. If $(d,a)$ and $(b,d)$ are arcs of $H$, then either $(d,c) $$\in$$ E(H)$ and so $H $$\approx $$H^{c}_{16}$, or $(c,d) $$\in$$ E(H)$ and so $H $$\approx $$S_{6}$. Now if $(a,d)$ and $(d,b)$ are arcs of $H$, then either $(c,d) $$\in$$ E(H)$ and so $H $$\approx $$S^{c}_{1}$, or $(d,c) $$\in$$ E(H)$ and so $H $$\approx$$ H^{c}_{17} $. Hence $1,3,5$, or $ 6 $ holds.\\
$\ast$ If $ N^{+}_{\lbrace a,b,c \rbrace}(e)= \lbrace a,c \rbrace $, then we have four cases: if $ \lbrace a,b \rbrace \subseteq N^{-}(d)$, then $(d,c)$$\in$$E(H)$ and so $(v,b,e,a,c,f,d)$ is a galaxy ordering. Otherwise if $\lbrace a,b \rbrace \subseteq N^{+}(d)$, then either $(c,d) $$\in$$ E(H)$ and so $H $$\approx$$ H_{18}$, or $(d,c) $$\in$$ E(H)$ and so $H $$\approx$$ H_{19}$. Otherwise if $(d,a)$ and $(b,d)$ are arcs, then either $(c,d) $$\in$$ E(H)$ and so $(a,f,v,b,e,c,d)$ is a galaxy ordering, or $(d,c) $$\in$$ E(H)$ and so $(v,b,e,d,a,c,f)$ is a galaxy ordering. Otherwise, $(a,d)$ and $(d,b)$ are arcs. Then either $(c,d) $$\in$$ E(H)$ and so $(v,a,b,c,f,e,d)$ is a galaxy ordering and $ 1 $ holds, or $(d,c) $$\in$$ E(H)$ and so $H $$\approx $$S_{7} $. Hence $1,3$, or $6$ holds.\\
$\bullet$ If $ d^{+}_{\lbrace a,b,c \rbrace}(e)= 3 $, then for $ N^{+}(d) \subseteq \lbrace a,b,c \rbrace$, $(v,e,a,b,c,f,d)$ is a galaxy ordering of $H$ and $ 1 $ holds.\\
\textbf{1.3.2} $H$$\mid$$\lbrace a,b,c \rbrace$ is a cyclic triangle. Assume without loss of generality that $(a,b),(b,c),(c,a)$ are arcs of $H$.\\
$\bullet$ If $ d^{+}_{\lbrace a,b,c \rbrace}(e)= 0 $, then $d$ has one out-neighbor in $\lbrace a,b,c \rbrace $, say $b$ (without loss of generality) and so $(v,a,b,c,f,e,d)$ is a galaxy ordering of $H$, or $d^{+}(d)= 2 $ (assume without loss of generality that $(d,b)$ and $(d,c)$ are arcs), and so $ H \approx S^{c}_{3}$, or $ d^{+}(d)= 3 $ and so $\lbrace b,a,c \rbrace $ is a non-trivial homogeneous set and $H$ is $ K_{6}-$free. Hence $1,5$, or $6$ holds. \\
$\bullet$ If $ d^{+}_{\lbrace a,b,c \rbrace}(e)= 1 $, assume without loss of generality that $(e,b)$$\in$$E(H)$, then we have three cases: if $ d^{+}_{\lbrace a,b,c \rbrace}(d)= 1 $, then either $N^{+}(d) \subseteq \lbrace a,b \rbrace$ and so $(v,a,b,c,f,e,d)$ is a galaxy ordering of $H$, or $(d,c) \in E(H)$ and so $(v,b,c,a,f,e,d)$ is a galaxy ordering of $H$. Now if $ d^{+}(d)= 2 $, then $(d,a)$ and $(d,b)$ are arcs and so $H \approx S_{8}$, or $(d,a)$ and $(d,c)$ are arcs and so $H \approx H_{20} $, or $(d,b)$ and $(d,c)$ are arcs and so $H \approx H^{c}_{21} $. Finally if $ d^{+}(d)= 3 $, then $H \approx H^{c}_{22} $. Hence $ 1,3 $, or $6$ holds.\\
$\bullet$ If $ d^{+}_{\lbrace a,b,c \rbrace}(e)= 2 $, assume without loss of generality that $(e,b)$ and $(e,c)$ are arcs, then we have three cases: if $ d^{+}(d)= 1 $, then $(d,b) \in E(H)$ and so $(v,a,e,b,c,f,d)$ is a galaxy ordering, or $(d,a) \in E(H)$ and so $ (v,a,e,b,c,f,d)$ is a constellation ordering (Note that $H$ is prime and not galaxy), or $(d,c) \in E(H)$ and so $(v,a,e,b,d,c,f)$ is a galaxy ordering. Now if $ d^{+}(d)= 2 $, then $(d,a)$ and $(d,b)$ are arcs, and so $H $$\approx$$ R^{c}_{3} $, or $(d,b)$ and $(d,c)$ are arcs and so $(v,a,e,d,b,c,f)$ is a galaxy ordering of $H$, or $(d,a)$ and $(d,c)$ are arcs and so $H $$\approx$$ H^{c}_{23} $. Finally if $ d^{+}(d)= 3 $, then $H $$\approx$$ H_{24} $. So $1,2,3$ or $ 4 $ holds.\\
$\bullet$ If $ d^{+}_{\lbrace a,b,c \rbrace}(e)= 3 $, then: $d$ has one out-neighbor in $\lbrace a,b,c \rbrace $ say $b$ (without loss of generality) and so $(v,e,a,b,c,f,d)$ is a galaxy ordering, or $d$ has two out-neighbors say $b$ and $c$ (without loss of generality) and so $(v,e,a,d,b,c,f)$ is a galaxy ordering, or $ d^{+}(d)$$=$$ 3 $ and so $(v,e,d,a,b,c,f)$ is a galaxy ordering. So $ 1 $ holds.\\
\textbf{1.4 $ d^{+}_{\lbrace a,b,c,d,e \rbrace }(f) = 1 $}. We can assume without loss of generality that $(f,e) \in E(H)$, then $1 \leq d^{+}(e) \leq 4$. If $H$$\mid$$\lbrace a,b,c,d \rbrace$ is a transitive tournament (assume without loss of generality that $a\rightarrow \lbrace b,c,d\rbrace$, $b\rightarrow\lbrace c,d \rbrace$, and $c\rightarrow d$), then $(v,a,b,c,d,f,e)$ is a galaxy ordering and $ 1 $ holds, and if $H$$ \mid$$ \lbrace a,b,c,d \rbrace $ contains a cycle of length $ 3 $, assume  without loss of generality that $(a,b),(b,c),(c,a)$ are arcs of $H$. Now we will study the neighbors of $d$ in $ \lbrace a,b,c \rbrace $ and the neighbors of $e$ in $ \lbrace a,b,c,d \rbrace $.\\
\textbf{1.4.1} $\lbrace a,b,c \rbrace $$\subseteq $$N^{-}(d)$. We have four cases: If $ d^{+}(e) = 1 $, then either $(e,d) $$\in$$ E(H)$, or $(e,b)$$ \in$$ E(H)$ (without loss of generality), and in both cases $(v,a,b,c,d,f,e)$ is a galaxy ordering of $H$. Otherwise if $ d^{+}(e) = 2 $, then either $(e,b)$ and $(e,c)$ are arcs (without loss of generality) and so $H$$ \approx$$ S_{9} $, or $(e,b)$ and $(e,d)$ are arcs (without loss of generality) and so $(v,a,b,c,d,f,e)$ is a galaxy ordering of $H$. Otherwise if $ d^{+}(e) = 3 $, then either $(d,e)$$ \in$$ E(H)$ and so $\lbrace b,a,c \rbrace $ is a non-trivial homogeneous set and $H$ is $K_{6}-$free, or $(a,e)$$ \in $$E(H)$ (without loss of generality) and so $(v,a,e,b,c,d,f)$ is a galaxy ordering of $H$. Otherwise,  $ d^{+}(e) = 4 $. Then $(v,e,a,b,c,d,f)$ is a galaxy orderingof $H$. So $ 1,5 $, or $6$ holds.\\
\textbf{1.4.2} If $ d^{+}_{\lbrace a,b,c \rbrace}(d) = 1 $, assume without loss of generality that $(d,c) $$\in$$ E(H)$ and so we have four cases:\\
$\bullet$ $ d^{+}(e) = 1 $. If $c$ is an out-neighbor of $e$, then $(v,a,b,d,e,c,f)$ is a galaxy ordering and $ 1 $ holds. Otherwise, $(v,a,b,d,c,f,e)$ is a galaxy ordering of $H$. So $ 1 $ holds.\\
$\bullet$ If $ d^{+}(e) = 2 $, then $(e,c)$ and $(e,d)$ are arcs of $H$ and so $(v,a,b,e,d,c,f)$ is a galaxy ordering, or $(e,a)$ and $(e,d)$ are arcs of $H$ and so $(v,b,c,e,a,d,f)$ is a galaxy ordering, or $(e,b)$ and $(e,c)$ are arcs and so $(v,a,d,e,b,c,f)$ is a galaxy ordering, or $(e,a)$ and $(e,b)$ are arcs of $H$ and so $H $$\approx$$ H_{25}$, or $(e,b)$ and $(e,d)$ are arcs of $H$, or $(e,a)$ and $(e,c)$ are arcs of $H$. In the last two cases $ \lbrace b,d \rbrace $ is a non-trivial homogeneous set and $H$ is $ K_{6}-$free. So $1,3$, or $ 5 $ holds. \\
$\bullet$ $ 3 \leq d^{+}(e) \leq 4 $. We have three cases: if $\lbrace b,d \rbrace \subseteq N^{+}(e)$, then $\lbrace b,d \rbrace $ is a non-trivial homogeneous set and $H$ is $K_{6}-$free and so $5$ holds. Otherwise if $ N^{+}(e) = \lbrace a,c,d \rbrace$, then $(v,b,e,c,a,d,f)$ is a galaxy ordering of $H$ and $ 1 $ holds. Otherwise $ N^{+}(e) = \lbrace a,b,c \rbrace$, then $H $$\approx $$S^c_{10} $. So $1,5$, or $ 6 $ holds.\\
\textbf{1.4.3} If $ d^{+}_{\lbrace a,b,c \rbrace}(d)= 2 $, assume without loss of generality that $\lbrace b,c \rbrace \subseteq N^{+}(d)$, and so this case is isomorphic to case $1.4.2$.\\
\textbf{1.4.4} If $ d^{+}_{\lbrace a,b,c \rbrace}(d)$$= $$3 $, then:\\
$\bullet$ If $ d^{+}(e)$$ = $$1 $, then either $(e,d)$$ \in$$ E(H)$, or $(e,b) $$\in$$ E(H)$ (without loss of generality). In both cases $(v,d,a,b,c,f,e)$ is a galaxy ordering of $H$ and $ 1 $ holds.\\
$\bullet$ If $ d^{+}(e) $$= $$2 $, then either $ N^{+}(e)$$=$$ \lbrace d,b \rbrace $ (without loss of generality) and so $(v,d,a,b,c,f,e)$ is a galaxy ordering, or $ N^{+}(e)$$=$$ \lbrace c,b \rbrace $ (without loss of generality) and so $(v,d,a,e,b,c,f)$ is a galaxy ordering, so $ 1 $ holds.\\
$\bullet$ If $ d^{+}(e)$$ =$$ 3 $, then either $ N^{+}(e)= \lbrace a,b,c \rbrace $ and so $(v,d,e,a,b,c,f)$ is a galaxy ordering of $H$ and $ 1 $ holds, or $ N^{+}(e)= \lbrace a,d,b \rbrace $ (without loss of generality) and so $ (v,c,e,d,a,b,f)$ is a constellation ordering of $H$ and $ 2 $ holds.\\
$\bullet$ If $ d^{+}(e) = 4 $, then $(v,e,d,a,b,c,f)$ is a galaxy ordering and $ 1 $ holds.
\vspace{3mm}

 Let us denote by $ n_{3,3} $ the number of vertices $v$ of $H$ such that $v$ has $ 3 $ out-neighbors and $ 3 $ in-neighbors, $ n_{4,2} $ the number of vertices $v$ of $H$ such that $v$ has $ 4 $ out-neighbors and $ 2 $ in-neighbors, and $ n_{2,4} $ the number of vertices $v$ of $H$ such that $v$ has $ 2 $ out-neighbors and $ 4 $ in-neighbors. Then we have:\\
$ 21 = 4n_{4,2} + 3n_{3,3} + 2n_{2,4} = 2n_{4,2} + 3n_{3,3} + 4n_{2,4} $. So we have four cases:
\vspace{1.5mm}\\
\textbf{Case 2: $n_{3,3} = 1, n_{4,2} = n_{2,4} = 3 $}.
Let $a,b,c$ be the vertices of $H$, such that $ d^{+}(a) = d^{+}(b) = d^{+}(c) = 4 $. Let $e,f,v$ be the vertices of $H$, such that $ d^{+}(e) = d^{+}(f) = d^{+}(v) = 2 $. Let $d$ be the vertex of $ H $, such that $ d^{+}(d) = 3 $.\\
\textbf{2.1} $H$$\mid$$\lbrace a,b,c \rbrace$ is a transitive tournament (assume without loss of generality that $(a,b),(b,c),(a,c)$ are arcs). Then $ \lbrace c \rbrace $ is complete to $ \lbrace d,e,f,v \rbrace $ and $b$ has one in-neighbor in $ \lbrace d,e,f,v \rbrace $. So without loss of generality we have two cases:\\
$\bullet$ If $(v,b) $$\in$$ E(H)$, then $a$ has two out-neighbors in $ \lbrace d,e,f,v \rbrace $, and so we have three cases: If $d$ is an out-neighbor of $a$, then $d$ is complete to $ \lbrace e,f,v \rbrace$ and so $\lbrace c,d \rbrace $ is a non-trivial homogeneous set of a $K_{6}-$free tournament $H$. Otherwise if $ N^{+}(a) = \lbrace b,c,e,f \rbrace $, then $ \lbrace v \rbrace $ is complete from $ \lbrace d,e,f \rbrace $, and so we can assume without loss of generality that $(e,d) $$\in$$ E(H)$. Then $(f,e)$ and $(d,f)$ are also arcs and so $H $$\approx$$ H_{26} $. Otherwise $ N^{+}(a) = \lbrace b,c,e,v \rbrace $, then $d$ has one in-neighbor in $ \lbrace e,f,v \rbrace $: $(f,d) $$\in$$ E(H)$ and so $H $$\approx$$ H_{27} $, or $(v,d) \in E(H)$ and so $H $$\approx$$ H_{28} $, or $(e,d) $$\in$$ E(H)$. In case $(e,d)$ is an arc of $H$, $e$ still has one in-neighbor in $ \lbrace f,v \rbrace $. Then either $(v,e) $$\in$$ E(H)$ and so $(a,b,c,e,d,f,v)$ is a galaxy ordering of $H$, or $(f,e)$$ \in$$ E(H)$ and so $H $$\approx$$ R_{4} $. Hence $1,3,4$, or $5$ holds.\\
$\bullet$ If $(d,b) \in E(H)$, then $a$ have two out-neighbors in $ \lbrace d,e,f,v \rbrace $, and so without loss of generality we have two cases: If $ N^{+}(a) = \lbrace b,c,e,f \rbrace $, then either $(d,e) \in E(H)$ (without loss of generality) and so $H $$\approx$$ H_{29} $, or $(d,v) \in E(H)$ and so $(v,e) \in E(H)$ (without loss of generality) and $H $$\approx$$ H^{c}_{30} $. Otherwise $ N^{+}(a) = \lbrace b,c,d,e \rbrace $. Then $d$ has one in-neighbor in $ \lbrace e,f,v \rbrace $ and so without loss of generality we have two cases. Either $(f,d) \in E(H)$ and so $H $$\approx$$ H^{c}_{26}$, or $(e,d) \in E(H)$ and so $(e,f) \in E(H)$ (without loss of generality) and $H $$\approx$$ H^{c}_{31} $. Hence  $3$ holds.\\
\textbf{2.2} $H$$\mid$$\lbrace a,b,c \rbrace$ is a cycle. Assume without loss of generality that $(a,b),(b,c),(c,a)$ are arcs of $H$. We will discuss according to the out-neighbors of $d$. \\
$\bullet$ If $ N^{+}(d) = \lbrace a,b,c \rbrace $, then $ \lbrace a,b,c \rbrace $ is a non-trivial homogeneous set of a $K_{6}-$free tournament $H$, and so $ 5 $ holds.\\
$\bullet$ If $ N^{+}(d)$$ =$$ \lbrace e,f,v \rbrace $, then we can assume without loss of generality that $ N^{+}(a)$$ =$$ \lbrace b,d,e,f \rbrace $, and so $b$ has two out-neighbors in $ \lbrace e,f,v \rbrace $. If $ N^{+}(b)$$ =$$ \lbrace c,d,e,f \rbrace $, then $v$ is complete from $ \lbrace e,f,c \rbrace $, and so we can assume without loss of generality that $(c,e) \in E(H)$. Then $ \lbrace e,d \rbrace $ is a non-trivial homogeneous set and $H$ is $K_{6}-$free. Otherwise, we can assume without loss of generality that $ N^{+}(b) = \lbrace c,d,v,f \rbrace $. Then $c$ has one in-neighbor in $ \lbrace e,f,v \rbrace $. If $(f,c) \in E(H)$,  then $e$ has one in-neighbor in $ \lbrace f,v \rbrace $, and so either $(f,e) $$\in$$ E(H)$ which implies that  $H$$ \approx $$R_{5} $, or $(v,e)$$ \in $$E(H)$ and so $H$$ \approx$$ S_{11} $. And if $(e,c)$$ \in$$ E(H)$ or $(v,c)$$ \in $$E(H)$, then $ \lbrace f,d \rbrace $ is a non-trivial homogeneous set and $H$ is $K_{6}-$free.  So $ 4,5 $, or $6$ holds.\\
$\bullet$ If $ N^{+}(d) = \lbrace a,b,e \rbrace $ (without loss of generality), then $ \lbrace a,b \rbrace $ is complete to $ \lbrace e,f,v \rbrace $ and $c$ has one in-neighbor in $ \lbrace e,f,v \rbrace $. So we have two cases: either $(e,c) \in E(H)$ and so we can assume without loss of generality that $(e,f)\in E(H)$ and $H \approx S_{12} $, or $(f,c) \in E(H)$ (without loss of generality) and so $ H \approx H_{30} $. Hence $3$ or $6$ holds.\\
$\bullet$ $ N^{+}(d) = \lbrace a,e,f \rbrace $ (without loss of generality). Then $ \lbrace a \rbrace $ is complete to $ \lbrace e,f,v \rbrace $ and $c$ has one in-neighbor in $ \lbrace e,f,v \rbrace $. If $(v,c) \in E(H) $, then we can assume without loss of generality that $(b,e) \in E(H)$.  So $ \lbrace c,d \rbrace $ is a non-trivial homogeneous set and $H$ is $K_{6}-$free, so $ 5 $ holds. Otherwise $(e,c) \in E(H)$ (without loss of generality). Then $b$ has one in-neighbor in $ \lbrace e,f,v \rbrace $.\\
$\ast$ $(e,b) \in E(H)$, then $H \approx H_{31} $ and $3$ holds. \\
$\ast$ $(f,b) \in E(H)$, then $e$ has one out-neighbor in $ \lbrace f,v \rbrace $ and so we have two cases: either $(e,f) \in E(H)$ and so $H \approx S_{13} $, or $(e,v) \in E(H)$ and so $H \approx S_{14} $. Hence $6$ holds.\\
$\ast$ $(v,b) \in E(H)$, then $H \approx R^{c}_{4} $ and $ 4 $ holds.
\vspace{1.5mm}\\
\textbf{Case 3: $ n_{3,3} = 3, n_{4,2} = n_{2,4} = 2 $}. Let $a,b,c$ be the vertices of $H$, such that $ d^{+}(a) = d^{+}(b) = d^{+}(c) = 3 $, let $ e,d$ be the vertices of $H$, such that $ d^{+}(e) = d^{+}(d) = 4 $, and let $f,v$ be the vertices of $H$, such that $ d^{+}(f) = d^{+}(v) = 2 $.\\
\textbf{3.1} $(a,b,c)$ is a transitive ordering. We can assume without loss of generality that $(d,e)$ and $(f,v)$ are arcs of $H$. Then $a$ has one out-neighbor in $ \lbrace d,e,f,v \rbrace $.\\
\textbf{3.1.1} $(a,d) \in E(H)$, then $c$ has one in-neighbor in $ \lbrace d,e,v \rbrace $.  If $(d,c)$$\in$$E(H)$, then $e$ is complete to $ \lbrace b,v \rbrace $ and so $v$ have one out-neighbor in $ \lbrace d,b \rbrace $, which implies that either $(v,b)$$\in$$E(H)$ and so $ \lbrace c,d \rbrace $ is a non-trivial homogeneous set of a $K_{6}-$free tournament $H$, or $(v,d)$$\in$$E(H)$ and so $H $$\approx$$ H_{32} $. Otherwise if $(e,c)$$\in$$E(H)$, then $d$ is complete to $ \lbrace b,v \rbrace $ and so we have two cases: either $(v,e)$$\in$$E(H)$ and so $H \approx H_{33} $, or $(v,b)$$\in$$E(H)$ and so $H \approx R_{6} $. Otherwise $(v,c)$$\in$$E(H)$. Then $H \approx H_{34} $. So $3,4$, or $5$ holds.\\
\textbf{3.1.2}  $(a,f)$$\in$$E(H)$. Then $c$ has one in-neighbor in $ \lbrace d,e,f,v \rbrace $.\\
$\bullet$ If $(d,c)$$\in$$E(H)$, then $v$ has one out-neighbor in $ \lbrace d,b \rbrace $, and so we have two cases: If $(v,b)$$\in$$E(H)$, then $H \approx H_{35} $ and $ 3 $ holds. Otherwise $(v,d)$$\in$$E(H)$. Then either $(b,d)$$\in$$E(H)$ and so $H \approx H_{36} $ and $ 3 $ holds, or $(b,f)$$\in$$E(H)$ and so $H \approx H^{c}_{37} $ and $ 3 $ holds. \\
$\bullet$ If $(f,c)$$\in$$E(H)$, then either $(v,b)$$\in$$E(H)$ and so $H \approx H_{38} $ and $ 3 $ holds, or $(v,d)$$\in$$E(H)$ and so $ \lbrace b,f \rbrace $ is a non-trivial homogeneous set, $H$ is $K_{6}-$free and $ 5 $ holds. \\
$\bullet$ If $(v,c)$$\in$$E(H)$, then either $(b,d)$$\in$$E(H)$ and so $H \approx R^{c}_{6} $ and $ 4 $ holds, or $(b,f)$$\in$$E(H)$ and so $H \approx H^{c}_{33}$ and $ 3 $ holds. \\
$\bullet$ If $(e,c)$$\in$$E(H)$, then $v$ have one out-neighbor in $ \lbrace b,d,e \rbrace $ and so we have three cases: If $(v,b)$$\in$$E(H)$, then $(f,d)$$\in$$E(H)$ and so $ \lbrace c,f \rbrace $ is a non-trivial homogeneous set and $H$ is $K_{6}-$free, or $(f,e)$$\in$$E(H)$ and so $H \approx H_{40} $, or $(f,b)$$\in$$E(H)$ and so $H \approx H_{41} $. Otherwise if $(v,d)$$\in$$E(H)$, then either $(f,e)$$\in$$E(H)$ and so $H \approx H_{42}$, or $(f,b)$$\in$$E(H)$ and so $H \approx H_{43} $. Otherwise $(v,e)$$\in$$E(H)$. Then $ \lbrace a,e \rbrace $ is a non-trivial homogeneous set and $H$ is $K_{6}-$free whatever $d$ or $b$ is an out-neighbor of $f$. So $ 3$ or $5 $ holds.\\
\textbf{3.1.3} $(a,e)\in E(H)$. Then either $(v,d)$$\in$$E(H)$ and so $ \lbrace e,b \rbrace $ is a non-trivial homogeneous set and $H$ is $K_{6}-$free, or $(v,b)$$\in$$E(H)$ and so $H \approx H^{c}_{38} $. Hence $ 3 $ or $5$ holds.\\
\textbf{3.1.4} $(a,v)$$\in$$E(H)$. Then $b$ has one out-neighbor in $ \lbrace d,e,v \rbrace $, and so we have three cases: If $(b,d)$$\in$$E(H)$, then $(d,c)$$\in$$E(H)$ and so $H \approx H_{44} $, or $(e,c)$$\in$$E(H)$ and so $H \approx H^{c}_{35} $, or $(v,c)$$\in$$E(H)$ and so $ \lbrace a,v \rbrace $ is a non-trivial homogeneous set and $H$ is $K_{6}-$free. Otherwise if $(b,e)$$\in$$E(H)$, then $H \approx H^{c}_{36} $. Otherwise $(b,v)$$\in$$E(H)$. Then either $(v,c)$$\in$$E(H)$ and so $H \approx H^{c}_{32} $, or $(v,e)$$\in$$E(H)$ is an arc and so $H \approx H_{37} $. Hence $ 3 $ or $5$  holds.\\
\textbf{3.2} $H$$\mid$$\lbrace a,b,c \rbrace$ is a cyclic triangle. Assume without loss of generality that $(a,b),(b,c),(c,a)$ are arcs. Then $a$ has two out-neighbors in $ \lbrace d,e,f,v \rbrace $.\\
\textbf{3.2.1} $ N^{+}(a) = \lbrace b,d,e \rbrace $. Then $b$ has two out-neighbors in $ \lbrace d,e,f,v \rbrace $, and so we have three cases: If $ N^{+}(b) = \lbrace c,f,v \rbrace $, then $(v,c)$$\in$$E(H)$ and so without loss of generality $(c,e) \in E(H)$ and $ \lbrace d,b \rbrace $ is a non-trivial homogeneous set, or $(v,f)$$\in$$E(H)$ and so $c$ or $d$ is an out-neighbor of $f$ (without loss of generality), which implies that $ \lbrace e,b \rbrace $ is a non-trivial homogeneous set, or $(v,d)$$\in$$E(H)$ (without loss of generality) and so $ \lbrace e,b \rbrace $ is a non-trivial homogeneous set. Note that $H$ is $K_{6}-$free in these cases and so $ 5 $ holds. Now if $ N^{+}(b) = \lbrace c,d,f \rbrace $, then $ \lbrace c,f \rbrace $ is a non-trivial homogeneous set and $H$ is $K_{6}-$free, so $ 5 $ holds. Finally if $ N^{+}(b) = \lbrace c,d,e \rbrace $ then $e$ or $d$ has three in-neighbors, a contradiction.\\
\textbf{3.2.2} $ N^{+}(a) = \lbrace b,f,v \rbrace $. Then $b$ still needs two out-neighbors.\\
\textbf{3.2.2.1} $ N^{+}(b) = \lbrace c,d,e \rbrace $. If $f$ is an in-neighbor of $v$, then $ \lbrace a,f \rbrace $ is a non-trivial homogeneous set and $H$ is $K_{6}-$free. Otherwise $f$ is an out-neighbor of $v$, then $ \lbrace a,v \rbrace $ is a non-trivial homogeneous set and $H$ is $K_{6}-$free whatever the out-neighbor of $c$ in $ \lbrace d,e,f \rbrace $ is. So $ 5 $ holds.\\
\textbf{3.2.2.2} $ N^{+}(b) = \lbrace c,f,v \rbrace $. Then $v$ has two out-neighbors in $ \lbrace c,d,e,f \rbrace $.\\
$\bullet$ If $ N^{+}(v) = \lbrace d,e \rbrace $, then either $(f,c)$$\in$$E(H)$ and so we can assume without loss of generality that $(c,d)$$\in$$E(H)$ which implies that $ \lbrace b,f \rbrace $ is a non-trivial homogeneous set, or $(f,e)$$\in$$E(H)$ (without loss of generality) and so $ \lbrace a,b,c \rbrace $ is a non-trivial homogeneous set. Note that in both cases $H$ is $K_{6}-$free. So $ 5 $ holds.\\
$\bullet$ If $ N^{+}(v) = \lbrace d,f \rbrace $ (without loss of generality), then $f$ has one in-neighbor in $ \lbrace c,d,e \rbrace $. If $(e,f)$$\in$$E(H)$, then $ \lbrace a,e \rbrace $ is a non-trivial homogeneous set and $H$ is $K_{6}-$free. Now if $(d,f)$$\in$$E(H)$, then either $(c,e)$$\in$$E(H)$ and so $H $$\approx$$ H_{45}$, or $(c,d)$$\in$$E(H)$ and so $H \approx H_{46} $. Finally if $(c,f)$$\in$$E(H)$, then $ \lbrace a,b,c \rbrace $ is a non-trivial homogeneous set and $H$ is $K_{6}-$free. So $3$ or $ 5 $ holds.\\
$\bullet$ If $ N^{+}(v) = \lbrace f,c \rbrace $, then $ \lbrace b,v \rbrace $ is a non-trivial homogeneous set and $H$ is $K_{6}-$free whatever the orientation of the rest of the arcs is. So $ 5 $ holds.\\
$\bullet$ If $ N^{+}(v) = \lbrace d,c \rbrace $ (without loss of generality), then $c$ has one in-neighbor in $ \lbrace f,d,e \rbrace $:\\ If $(f,c)$$\in$$E(H)$, then $ \lbrace a,e \rbrace $ is a non-trivial homogeneous set and $H$ is $K_{6}-$free. Otherwise if $(e,c)$$\in$$E(H)$,  then $H$$ \approx$$ H_{45} $. Otherwise $(d,c)$$\in$$E(H)$. Then either $(d,f)$$\in$$E(H)$ and so $H$$ \approx$$ H_{46} $, or $(d,e)$$\in$$E(H)$ and so $ \lbrace a,e \rbrace $ is a non-trivial homogeneous set and $H$ is $K_{6}-$free. So $3$ or $ 5 $ holds.\\
\textbf{3.2.2.3} $ N^{+}(b) = \lbrace c,d,f \rbrace $ (without loss of generality). Then $v$ has one out-neighbor in $ \lbrace c,d,e,f \rbrace $.\\
$\bullet$ If $(v,d)$$\in$$E(H)$, then either $(e,f)$$\in$$E(H)$ and so $ \lbrace a,e \rbrace $ is a non-trivial homogeneous set and $H$ is $K_{6}-$free, or $(c,f)$$\in$$E(H)$ and so $H \approx H_{45} $. Hence $ 3 $ or $5$ holds. \\
$\bullet$ If $(v,c)$$\in$$E(H)$, then $f$ has one out-neighbor in $ \lbrace c,d,e \rbrace $. If $(f,e)$$\in$$E(H)$, then either $(d,e)$$\in$$E(H)$ and so $H $$\approx$$ H_{47} $, or $(c,e)$$\in$$E(H)$ and so $H$$ \approx$$ H^{c}_{48} $. Now if $(f,d)$$\in$$E(H)$, then $H$$ \approx$$ R_{7} $. Finally if $(f,c)$$\in$$E(H)$, then $ \lbrace a,e \rbrace $ is a non-trivial homogeneous set and $H$ is $K_{6}-$free. So $3,4$ or $ 5 $ holds.\\ 
$\bullet$ If $(v,f)$$\in$$E(H)$, then either $(c,d)$$\in$$E(H)$ and so $H$$ \approx $$R^{c}_{7}$ and $4$ holds, or $(d,c)$$\in$$E(H)$ and so $\lbrace a,v \rbrace $ is a non-trivial homogeneous set of a $K_{6}-$free tournament $H$ whatever the orientation of the rest of the arcs is. So $5$ holds. \\
$\bullet$ If $(v,e)$$\in$$E(H)$, then $(f,d)$$\in$$E(H)$ and so $H$$ \approx $$H_{46} $, or $(f,e)$$\in$$E(H)$ and so $ \lbrace c,d \rbrace $ is a non-trivial homogeneous set and $H$ is $K_{6}-$free, or $(f,c)$$\in$$E(H)$ and here we have two cases: either $(c,d)$$\in$$E(H)$ and so $H$$ \approx$$ H^{c}_{48}$, or $(c,e)$$\in$$E(H)$ and so $H $$\approx$$ H_{49}$. Hence  $ 3 $ or $5$ holds.\\
\textbf{3.2.3} $ N^{+}(a) = \lbrace b,d,f \rbrace $ (without loss of generality). Then we have six cases:\\
$\bullet$ If $ N^{+}(b) = \lbrace c,d,e \rbrace $, then $ \lbrace c,e \rbrace $ is a non-trivial homogeneous set and $H$ is $K_{6}-$free. So $ 5 $ holds.\\
$\bullet$ If $ N^{+}(b) = \lbrace c,d,f \rbrace $, then either $(c,e)$$\in$$E(H)$ and so $H $$\approx$$ H_{50} $, or $(c,f)$$\in$$E(H)$ and so $H $$\approx$$ H_{47} $. Hence $ 3 $ holds.\\
$\bullet$ If $ N^{+}(b) = \lbrace c,d,v \rbrace $, then we have three cases: If $(v,c)$$\in$$E(H)$, then $H $$\approx $$H_{51}$. Now if $(v,e)$$\in$$E(H)$, then $H $$\approx$$ H_{48} $. Finally if $(v,f)$$\in$$E(H)$, then either $(c,e)$$\in$$E(H)$ and so $H $$\approx$$ H_{52} $, or $(c,f)$$\in$$E(H)$ and so $H $$\approx $$R^{c}_{7} $. Hence $3$ or $ 4 $ holds.\\
$\bullet$ $ N^{+}(b) = \lbrace c,e,f \rbrace $. Then $c$ has one out-neighbor in $ \lbrace d,e,f \rbrace $ and so we have three cases: If $(c,d)$$\in$$E(H)$, then $H$$ \approx $$H_{52} $. Otherwise if $(c,e)$$\in$$E(H)$, then $H \approx H_{51} $. Otherwise $(c,f)$$\in$$E(H)$, then either $(f,d)$$\in$$E(H)$ and so $H \approx R_{7} $, or $(f,e)$$\in$$E(H)$ and so $H \approx H^{c}_{48}$. Hence $ 3 $ or $4$ holds.\\
$\bullet$ If $ N^{+}(b) = \lbrace c,e,v \rbrace $, then $v$ has one out-neighbor in $ \lbrace c,d,e,f \rbrace $ and so we have four cases: If $(v,c)$$\in$$E(H)$, then either $(c,d)$$\in$$E(H)$ and so $H $$\approx$$ H_{50} $, or $(c,e)$$\in$$E(H)$ and so $H$$ \approx $$H_{52} $. Otherwise if $(v,d)$$\in$$E(H)$, then $ \lbrace c,e \rbrace $ is a non-trivial homogeneous set and $H$ is $K_{6}-$free. Otherwise if $(v,e)$$\in$$E(H)$, then $H $$\approx$$ H_{49} $. Otherwise $(v,f)$$\in$$E(H)$. Then $c$ has one out-neighbor in $ \lbrace f,d,e \rbrace $. In this case $(c,d)$$\in$$E(H)$ and so $H $$\approx$$ H_{51} $, or $(c,e)$$\in$$E(H)$ and so $H $$\approx$$ H_{50} $, or $(c,f)$$\in$$E(H)$ and so $ \lbrace c,v \rbrace $ is a non-trivial homogeneous set and $H$ is $K_{6}-$free whatever $e$ or $d$ is an out-neighbor of $f$. So $3$ or $ 5 $ holds.\\
$\bullet$ $ N^{+}(b) = \lbrace c,f,v \rbrace $. Then $v$ has one out-neighbor in $ \lbrace c,d,e,f \rbrace $:\\
$\ast$ If $(v,f)$$\in$$E(H)$, then $c$ has one out-neighbor in $ \lbrace d,e,f \rbrace $. If $(c,d)$$\in$$E(H)$, then $H \approx H_{48} $. Now if $(c,f)$$\in$$E(H)$, then $ \lbrace c,v \rbrace $ is a non-trivial homogeneous set and $H$ is $K_{6}-$free. Finally if $(c,e)$$\in$$E(H)$, then either $(f,e)$$\in$$E(H)$ and so $ \lbrace b,d \rbrace $ is a non-trivial homogeneous set and $H$ is $K_{6}-$free, or $(f,d)$$\in$$E(H)$ and so $H \approx H_{49} $. So $ 3 $ or $5$ holds.\\
$\ast$ If $(v,e)$$\in$$E(H)$, then $f$ has one out-neighbor in $ \lbrace c,d,e \rbrace $. If $(f,d)$$\in$$E(H)$, then $H \approx H_{46} $. Now if $(f,e)$$\in$$E(H)$, then $ \lbrace b,d \rbrace $ is a non-trivial homogeneous set and $H$ is $K_{6}-$free. Finally if $(f,c)$$\in$$E(H)$, then either $(c,e)$$\in$$E(H)$ and so $ \lbrace b,f \rbrace $ is a non-trivial homogeneous set and $H$ is $K_{6}-$free, or $(e,c)$$\in$$E(H)$ and so $H \approx R^{c}_{7} $. Hence $3, 4 $, or $5$ holds.\\
$\ast$ If $(v,d)$$\in$$E(H)$, then either $(f,c)$$\in$$E(H)$ and so $ \lbrace b,f \rbrace $ is a non trivial homogeneous set and $H$ is $K_{6}-$free, or $(e,c)$$\in$$E(H)$ and so $H \approx H_{45} $. Hence $ 3 $ or $5$ holds.\\
$\ast$ If $(v,c)$$\in$$E(H)$, then $(f,c)$$\in$$E(H)$ and so $ \lbrace b,f \rbrace $ is a non-trivial homogeneous set and $H$ is $K_{6}-$free, or $(f,d)$$\in$$E(H)$ and so $H \approx H_{48}^{c}$, or $(f,e)$$\in$$E(H)$ and here we have two cases: either $(c,e)$$\in$$E(H)$ and so $H \approx R^{c}_{7} $, or $(c,d)$$\in$$E(H)$ and so $H \approx H_{47} $. Hence $ 3,4 $, or $5$ holds.
\vspace{1.5mm}\\
\textbf{Case 4: $ n_{3,3} = 5, n_{4,2} = n_{2,4} = 1 $}. Let $c,d,e,f,v$ be the vertices of $H$, such that $ d^{+}(c) = d^{+}(d) = d^{+}(e) = d^{+}(f) = d^{+}(v) = 3 $, let $a$ be the vertex of $H$, such that $ d^{+}(a) = 4 $, and let $b$ be the vertex of $H$, such that $ d^{+}(b) = 2 $. Without loss of generality we have two cases: either $ N^{+}(a) = \lbrace c,d,e,f \rbrace $, or $ N^{+}(a) = \lbrace c,d,e,b \rbrace $.\\
\textbf{4.1} If $ N^{+}(a) = \lbrace c,d,e,f \rbrace $, then $b$ has one out-neighbor in $\lbrace c,d,e,f,v \rbrace $.\\
$\bullet$ If $ N^{+}(b) = \lbrace v,a \rbrace $, then we can assume without loss of generality that $ N^{+}(v) = \lbrace e,f,a \rbrace $ and so $c$ has one out-neighbor in $ \lbrace d,e,f \rbrace $. Which implies that either $(c,e)$$\in$$E(H)$ (without loss of generality) and so $H $$\approx$$ H_{54}$, or $(c,d)$$\in$$E(H)$ and so by assuming without loss of generality that $(d,e)$$\in$$E(H)$, $H \approx H_{54}$. So $ 3 $ holds.\\
$\bullet$ If $ N^{+}(b) = \lbrace c,a \rbrace $ (without loss of generality), then if $(v,c)$$\in$$E(H)$, we can assume without loss of generality that $(f,e)$$\in$$E(H)$ and so $ \lbrace d,e,f \rbrace $ is a non-trivial homogeneous set of a $K_{6}-$free tournament $H$,  and if without loss of generality $(d,c)$$\in$$E(H)$, then either $(d,e)$$\in$$E(H)$ (without loss of generality) and so $H \approx H_{53} $, or $(d,v)$$\in$$E(H)$ and so by assuming without loss of generality that $(v,f)$$\in$$E(H)$, $H $$\approx$$ R_{8} $. Hence  $3, 4 $, or $5$ holds.\\
\textbf{4.2} If $ N^{+}(a) = \lbrace c,d,e,b \rbrace $, then we can assume without loss of generality that the set of out-neighbors of $b$ is $\lbrace f,v \rbrace $, or $\lbrace c,d \rbrace $, or $\lbrace f,e \rbrace $.\\
$\bullet$ $ N^{+}(b)$$ =$$ \lbrace f,v \rbrace $. If $ N^{+}(v) $$=$$ \lbrace f,a,e \rbrace $ (without loss of generality), then either $(e,f)$$\in$$E(H)$ and so we can assume without loss of generality that $(e,d)$$\in$$E(H)$, which implies that $H $$\approx$$ H_{55}$, or $(d,f)$$\in$$E(H)$ (without loss of generality) and so $ \lbrace b,d \rbrace $ is a non-trivial homogeneous set and $H$ is $K_{6}-$free. Now if $ N^{+}(v) $$=$$ \lbrace d,a,e \rbrace $ (without loss of generality), then either $(f,e)$$\in$$E(H)$ (without loss of generality) and so $ \lbrace b,c \rbrace $ is a non-trivial homogeneous set of a $K_{6}-$free tournament $H$, or $(f,c)$$\in$$E(H)$ and so by assuming without loss of generality that $(c,d)$$\in$$E(H)$, $H$$ \approx$$ H_{55} $.  Hence $ 3 $ or $5$ holds.\\
$\bullet$ $ N^{+}(b) = \lbrace c,d \rbrace $. Then we have two cases: If $ N^{+}(e) = \lbrace f,v,b \rbrace $, then either $(v,f)$$\in$$E(H)$ and so we can assume without loss of generality that $(f,d)$$\in$$E(H)$, or $(v,c)$$\in$$E(H)$ (without loss of generality). In both cases $H \approx H_{56} $ and $ 3 $ holds. And if $ N^{+}(e) = \lbrace f,d,b \rbrace $, then $H \approx H_{57} $ and $ 3 $ holds.\\
$\bullet$ $ N^{+}(b) = \lbrace e,f \rbrace $. Then $e$ has one in-neighbor in $\lbrace c,d,f,v \rbrace$.\\
$\ast$ If $(v,e)$$\in$$E(H)$, then we can assume without loss of generality that $(f,d)$$\in$$E(H)$, and so $H \approx H_{59} $ and 3 holds.\\
$\ast$ If $(f,e)$$\in$$E(H)$, then either $(f,v)$$\in$$E(H)$ and so $H \approx R_{9}^{c}$, or $(f,d)$$\in$$E(H)$ (without loss of generality) and so $H \approx H_{58} $. Hence $ 3 $ or $4$ holds.\\
$\ast$ If  $(d,e)$$\in$$E(H)$ (without loss of generality), then we have three cases: If $(d,c)$$\in$$E(H)$, then $ \lbrace a,d \rbrace $ is a non-trivial homogeneous set and $H$ is $K_{6}-$free. Otherwise if $(f,c)$$\in$$E(H)$, then either $(d,f)$$\in$$E(H)$ and so $ \lbrace b,d \rbrace $ is a non-trivial homogeneous set and $H$ is $K_{6}-$free, or $(d,v)$$\in$$E(H)$ and so $H \approx R_{9}$. Finally if $(v,c)$$\in$$E(H)$, then $H \approx R_{10}$. So  $4$ or $5$ holds.
\vspace{1.5mm}\\
\textbf{Case 5: $ n_{3,3} = 7, n_{4,2} = n_{2,4} = 0 $} \textbf{(Regular case)}\\
In this case its easy to check that if $H$ is a regular $7-$vertex tournament, then $H$$\approx$$ H_{39}$ or $H$$\approx$$R_{11}$ or $H$$\approx$$S_{15}$. Then $3$ or $4$ or $6$ holds. $\hfill { \square }$

\end{document}